\documentclass[11pt]{article}
\usepackage{amssymb,amsmath,amsfonts,amsthm,mathtools,color}

\usepackage{mathrsfs}
\usepackage{dsfont}

\usepackage{bm} 

\usepackage{comment} 
\usepackage[title,titletoc,header]{appendix}
\usepackage{graphicx,subfig}
\usepackage{float} 

\usepackage{paralist}

\usepackage{tikz}
\usetikzlibrary{arrows,positioning,shapes.geometric}

\usepackage[linesnumbered, ruled,vlined]{algorithm2e}
\SetKwInput{KwInput}{Input}                

\graphicspath{ {./figures/} }

\usepackage{indentfirst}
\usepackage{multicol}
\usepackage{booktabs}
\usepackage{url}
\usepackage[outdir=./]{epstopdf} 

\usepackage[shortlabels]{enumitem}

\usepackage{hyperref}
\hypersetup{
    colorlinks=true, 
    linktoc=all,     
    linkcolor=blue,  
}
\numberwithin{equation}{section}

\newtheorem{Theorem}{Theorem}[section]
\newtheorem{Lemma}[Theorem]{Lemma}
\newtheorem{Proposition}[Theorem]{Proposition}
\newtheorem{Corollary}[Theorem]{Corollary}
\newtheorem{Assumption}{H.\!\!}

\theoremstyle{definition}

\theoremstyle{remark}
\newtheorem{Remark}{Remark}[section]

 \def\p{\partial} \def\nb{\nonumber}
\def\to{\rightarrow}
 \def\ol{\overline}    \def\ul{\underline}
\def\Om{\Omega}  \def\om{\omega} 
 
\newcommand{\q}{\quad}

\def\l{\label}  
  \def\fa{\forall}
  \def\a{\alpha} 
 
\def\eps{\varepsilon}

 \def\t{\times}  
\def\ms{\medskip}

\def \la{\langle} \def\ra{\rangle}

\def\cA{\mathcal{A}}
\def\cB{\mathcal{B}}
\def\cC{\mathcal{C}}
\def\cD{\mathcal{D}}

\def\cF{\mathcal{F}}

\def\cH{\mathcal{H}}
\def\cI{\mathcal{I}}

\def\cL{\mathcal{L}}
\def\cM{\mathcal{M}}
\def\cN{\mathcal{N}}
\def\cO{\mathcal{O}}
\def\cP{\mathcal{P}}

\def\cS{\mathcal{S}}
\def\cT{\mathcal{T}}

\def\cV{\mathcal{V}}

\def\d{{ \mathrm{d}}}

\def\sE{{\mathbb{E}}}
\def\sF{{\mathbb{F}}}

\def\sN{{\mathbb{N}}}
\def\sP{\mathbb{P}}

\def\sR{{\mathbb R}}
\def\sS{{\mathbb{S}}}

\newcommand{\tr}{\textnormal{tr}}

\DeclareMathOperator*{\esssup}{ess\,sup}

\newcommand{\lc}
{\mathrel{\raise2pt\hbox{${\mathop<\limits_{\raise1pt\hbox
{\mbox{$\sim$}}}}$}}}

\newcommand{\gc}
{\mathrel{\raise2pt\hbox{${\mathop>\limits_{\raise1pt\hbox{\mbox{$\sim$}}}}$}}}

\newcommand{\ec}
{\mathrel{\raise2pt\hbox{${\mathop=\limits_{\raise1pt\hbox{\mbox{$\sim$}}}}$}}}

\def\bb{\begin{equation}} \def\ee{\end{equation}}
\def\bbn{\begin{equation*}} \def\een{\end{equation*}}

\def\beqn{\begin{eqnarray}}  \def\eqn{\end{eqnarray}}

\def\beqnx{\begin{eqnarray*}} \def\eqnx{\end{eqnarray*}}

\def\bn{\begin{enumerate}} \def\en{\end{enumerate}}

\def\bd{\begin{description}} \def\ed{\end{description}}



\begin{document}

\title{
Convergence of 
policy gradient methods   
  for   
  finite-horizon
exploratory  
  linear-quadratic  control
  problems 
  }

\author{
Michael Giegrich\thanks{
Mathematical Institute, University of Oxford, Oxford OX2 6GG, UK
 ({\tt michael.giegrich@maths.ox.ac.uk,
 christoph.reisinger@maths.ox.ac.uk})}
\and
Christoph Reisinger\footnotemark[1]
\and
Yufei Zhang\thanks{Department of Mathematics, Imperial College London,  London,  UK 
({\tt yufei.zhang@imperial.ac.uk})}
}

\date{}
\maketitle

\noindent\textbf{Abstract.} 
We study the global linear convergence of policy gradient
(PG) methods for
finite-horizon 
continuous-time
exploratory  linear-quadratic control (LQC)   problems. 
The setting includes  stochastic LQC problems
with indefinite costs
and   allows    additional     entropy  regularisers in the objective. 
We consider a continuous-time 
Gaussian policy whose mean is linear in the state variable and whose covariance is state-independent.  
Contrary to discrete-time problems,
the cost is noncoercive in the policy and   not all  descent directions    lead to bounded iterates.
We propose geometry-aware gradient descents    for the mean and covariance of the policy using   the Fisher geometry  and 
the Bures-Wasserstein geometry, respectively. 
The policy iterates are   shown to 
satisfy an a-priori bound, and converge
  globally to  the 
optimal policy with a linear rate.
We further propose a novel PG method 
with   discrete-time policies.
The algorithm  leverages the continuous-time analysis,
and 
   achieves a robust  linear convergence across different action frequencies. 
 A numerical experiment confirms the convergence and robustness of the proposed algorithm.


\ms



%
%

\medskip
 
 \section{Introduction}
 
In recent years, the policy gradient (PG) method 
and its variants 
 have become an effective tool in   seeking optimal polices to control stochastic  systems 
 (see e.g., \cite{konda1999actor, sutton1999policy,kakade2001natural,schulman2015trust,
 schulman2017proximal}).
These algorithms parametrise the policy as a function   of the system state, and update the policy parametrisation  based on  the   gradient of the control objective.
Most of the progress,  especially the     convergence analysis       of PG methods,
 has been in discrete-time   
 Markov decision processes (MDPs)
 (see e.g., 
 \cite{fazel2018global, hambly2020policy, 
 mei2020global,
 zhang2021policy,
  kerimkulov2022convergence}).
However,
most real-world control systems, such as those in aerospace, the automotive industry and robotics, are naturally continuous-time dynamical systems,
and hence do not fit in the MDP setting. 

 One of the most fundamental stochastic  control problems is the finite-horizon  linear-quadratic control (LQC) problem.
 It aims to control a 
 linear stochastic  differential equation 
  over a given time horizon,
subject to     a quadratic cost.
This problem is important 
as it provides a  reasonable approximation  of many nonlinear control problems, and has been used in a wide range of applications,
including
portfolio optimisation \cite{zhou2000continuous, wang2020continuous}, algorithmic trading \cite{cartea2018algorithmic} and  production management of exhaustible resources \cite{graber2016linear}.
 Moreover, the optimal policy of an LQC problem admits a natural parameterisation as 
 a (time-dependent) linear function  of the state,
 and hence it suffices to determine the coefficients of this linear function. 
 All these properties make the LQC problem    an  important theoretical benchmark for 
studying learning-based control. 
 
\paragraph{Issues and challenges   from continuous-time models.}
It is insufficient and improper to rely solely on the analysis and algorithms for   discrete-time MDPs to solve   continuous-time problems, including LQC problems.
 There is a mismatch between the algorithm timescale for the former and the underlying systems timescale for the latter. 
This model mismatch can make 
conventional  discrete-time algorithms 
 very sensitive to the     discretisation stepsize. 
For instance, the empirical studies 
in \cite{munos2006policy, park2021time}
  suggest that    standard  PG methods   exhibit degraded performance as 
  the agent's action frequency increases
  (see  Section \ref{sec:numerical} for more details).
Similar performance degradation  
 has   been observed 
in \cite{tallec2019making}
for Q-learning methods. 
Recently, \cite{jia2022policy} and   \cite{jia2023q} 
extend PG and Q-learning methods, respectively,
to continuous-time problems 
without time discretisation,
in order to  develop algorithms that are robust across different timescales.
Nevertheless, the convergence   of these algorithms has not been studied, even for LQC problems.

%
There are technical reasons behind the limited theoretical progress of PG methods  for  continuous-time     LQC problems. 
The   objective of a LQC problem is typically nonconvex with respect to the policies (see Proposition \ref{prop:non_coercive}),
analogous to its discrete-time   counterpart
\cite{fazel2018global,zhang2021policy}. 
This links    the convergence analysis   of   PG methods
 to 
 the    analysis of gradient search for nonconvex objectives, which has always been one of the formidable challenges in optimisation theory.
The time-dependent nature  of the optimal policy for   finite-horizon 
 LQC problems   poses new challenges.
It  requires analysing the    optimisation landscape 
over  a suitable infinite-dimensional 
 policy space,
 instead of in a finite-dimensional parameter space.

One significant new feature of  LQC problems 
with continuous-time policies,
in contrast to discrete-time policies, 
  is the \emph{noncoercivity}    of the cost function
   (see Proposition \ref{prop:non_coercive}). 
Coercivity of the cost means that each 
  sublevel set of the cost is bounded, and this
implies that the  iterates  
of a discrete-time algorithm 
remain bounded 
as long as the cost  decreases along the iteration.
This
  can   be  ensured by 
updating the policies along 
\emph{any descent direction} of the cost
with a sufficiently small stepsize. 
The lack of coercivity of the continuous-time cost function    complicates the analysis of PG methods,
since 
 for a given descent direction, 
  there may not  exist   a constant stepsize 
such that  the iterates remain bounded 
    as the algorithm proceeds.

\paragraph{Our  contributions.}
 
This paper proposes   convergent   PG methods to solve  finite-horizon 
 exploratory   LQC problems,
   which generalise     classical LQC problems
   by allowing   an   entropy regulariser  in the objective. 
   
\begin{itemize}
\item We reformulate the exploratory LQC problem into a  minimisation  over   Gaussian polices.  
Each  Gaussian  policy is parameterised by two time-dependent  functions $(K,V)$: 
  the mean is 
  linear in the state 
   with the coefficient $K$, 
and the covariance is the function $V$.  
The policy gradient  of the cost 
 is   characterised by     the Pontryagin optimality principle.
The cost is shown to satisfy 
a non-uniform   {\L}ojasiewicz condition 
and 
a non-uniform smoothness condition 
(Propositions \ref{prop:Lojaiewicz} and \ref{prop:Lipshcitz_smooth}).
We then prove that 
the cost is neither coercive nor quasiconvex in $K$,
even in a one-dimensional deterministic setting
(Proposition \ref{prop:non_coercive}).

\item 
 We propose a geometry-aware PG method to solve the   LQC problem
in continuous time. 
The gradient for  $K$ 
adapts to   the geometry induced by
the Fisher information metric
(also known as the natural gradient),
while 
the gradient for  $V$ 
adapts to the geometry induced by the
Bures-Wasserstein metric.  
These 
 geometry-aware
 gradient directions 
 are proved to 
 enjoy 
an \emph{implicit regularisation} property, 
i.e., they  preserve an   $L^2$-bound 
of $K$, and pointwise upper and lower bounds of $V$ 
without an explicit projection step
(Proposition \ref{prop:uniform_bound}). 
This allows for 
 exploiting the local regularity of the cost,
 and proving   the PG method 
 converges globally to  the 
optimal policy with a linear rate (Theorem \ref{thm:linear_conv}).

\item 
By leveraging the continuous-time analysis,
we propose   practically implementable    PG methods that
 take actions at discrete time points,
 and achieve a   linear convergence guarantee   independent of the action frequency. 
Our analysis shows that scaling the discrete-time gradients linearly   with respect to action frequency is critical for a robust performance of the algorithm in different timescales (Remark \ref{rmk:hyperparameter_scaling}). 
The  theoretical property is verified    through a numerical experiment on an    exploratory
  LQC problem arising from      mean-variance portfolio selection problems. 
This shows that 
the number of required  iterations 
for conventional PG methods 
grows linearly in   the number of  action time points,
while  the proposed  PG methods 
achieve a robust linear convergence rate  over a wide range of action frequencies.

\end{itemize}

\paragraph{Our approaches and related works.}
Most existing theoretical works of PG methods for LQC problems consider  the  setting of   infinite horizon and   deterministic dynamics (see e.g., 
\cite{fazel2018global,bu2020policy}).
For the case with noisy dynamics,
  existing works focus on  discrete-time problems.
  This includes   the setting of 
    infinite horizon and additive noise
    \cite{jin2020analysis, zhang2021policy},
    finite horizon and additive noise
\cite{hambly2020policy},
and 
    infinite horizon and multiplicative noise \cite{gravell2020learning}.
    We further refer the reader to 
    \cite{hambly2021policy, wang2021global, zhang2021derivative} for LQ games. 
  In all of these settings, the optimal policy admits a 
    \emph{finite-dimensional}   parameterisation.
     
     Compared to existing works, 
      our technical difficulties are three-fold.
      First, 
analysing the optimisation landscape over  
infinite-dimensional
continuous-time  policies
requires  continuous-time control theory.
For instance, 
{the policy gradient is derived via   Pontryagin's maximum principle.}
The cost regularity
(such as  {\L}ojasiewicz and smoothness conditions)
 is proved by  using partial differential equation techniques. 
 The lack of cost coercivity also adds complexity to the choice of appropriate descent directions, as discussed in Remark \ref{rmk:implicit_regularisation}. 
 {Notably, the noncoercivity of the cost function in this context primarily stems from the 
 fact that a policy can  have an infinite number of changes in values, occurring at arbitrary time points.
 This characteristic distinguishes our problem from   aforementioned discrete-time scenarios, in which policies change solely at predetermined time points.}

 Second, 
the finite-horizon     continuous-time setting  requires
more advanced techniques   
  for the nondegeneracy of  the state covariance than the   discrete-time setting. 
In  
  \cite{hambly2020policy, zhang2021policy},
the state covariance is lower bounded   by the  minimum eigenvalue
  of the  covariance of   system noises,
  \emph{uniformly over all   policies}. 
This bound vanishes as the time discretisation  stepsize tends to zero,  as  the  covariance of noise  increment typically scales  
linearly to the stepsize. 
Moreover, in the present setting, 
the system noise can degenerate due to a controlled diffusion coefficient.  
We overcome this difficulty by establishing 
 the positive definiteness of the state covariance 
 \textit{along the policy iterates}.
 This is possible by
a) 
first estimating   the state covariance  explicitly using  the magnitude of   policies, but independent 
of the system noise
(Lemma \ref{lemma:Sigma_bdd}), 
and   b) then  proving that
the geometry-aware gradient directions induce  a uniform bound of the iterates. 
This approach is different from the contraction argument in \cite{reisinger2022linear}
for     problems with uncontrolled diffusion coefficients.

Finally, the possible degeneracy of cost matrices requires  sharper   estimate   of the cost regularity. 
All existing   works assume 
a running cost of the form   $f(x,a)=x^\top Q x +a^\top R a$, with positive definite matrices $Q$ and $R$,
and estimate optimisation landscape using   minimum eigenvalues of $Q$ and $R$.
However, for many applications of  stochastic LQC problems,
the cost can involve the product of state and control variables \cite{cartea2018algorithmic},
or an indefinite   weight $R$ \cite{zhou2000continuous, wang2020continuous}. 
Here, we derive tighter  
    {\L}ojasiewicz and   smoothness bounds of the cost 
       using solutions to  
Lyapunov equations, 
instead of   the cost coefficients.  
This allows us to  consider a general setting
where both the drift and diffusion coefficients
of the state   
 are controlled,
and all cost weights can be negative definite.

\paragraph{Notation.} 
 For  each Euclidean space $E$, 
we denote by 
$\la \cdot, \cdot \ra$ its usual inner product 
and 
$|\cdot|$ the  norm induced by $\la \cdot, \cdot \ra$.
For each $A\in \sR^{n\t m}$, 
we denote by  $A^\top$ the transpose of $A$,
by  $\tr(A)$ the trace of $A$,
and by
$\|A\|_2$ the spectral norm of $A$.
 For each  $n\in \sN$, we denote by 
 $I_n$ the $n\t n$ identity matrix,
 by $\sS^n$,  
$\ol{\sS^n_{+}}$ 
and $\sS^n_+$
 the space of $n\t n$ symmetric,
 symmetric positive semidefinite,
 and symmetric positive definite  matrices, respectively,
 and  by 
  $\lambda_{\max}(A)$ and 
 $\lambda_{\min}(A)$ the  largest and smallest eigenvalues of   $A\in \sS^n$,
 respectively.
 We equip $\sS^n$ with the 
  Loewner (partial) order such that 
for each $A,B\in \sS^{n}$, 
$A\succeq B$ if $A-B\in \ol{\sS^n_+}$.
For every  measurable functions $F,G:[0,T]\to \sS^n$,  
  $F\succeq G$ stands for  $F(t)-G(t)\in \ol{\sS^n_+}$ for a.e.~$t\in [0,T]$.  
   
For each $T> 0$, filtered probability space $(\Om,\cF,\sF,\sP)$ 
 satisfying the usual condition
 (of  right continuity and completeness)
  and  Euclidean space $(E,|\cdot|)$, we introduce the following spaces:
 \begin{itemize}[leftmargin=*,noitemsep,topsep=0pt]
 \item
  $\cB(0,T;E)$ 
 is the space of Borel  measurable  functions
 $\phi:[0,T]\to E$.

\item 
 $L^p(0,T;E)$, $p\in [1,\infty]$,
 is the space of Borel  measurable  functions
 $\phi:[0,T]\to E$ satisfying
 $\|\phi\|_{L^p}=(\int_0^T |\phi_t|^p\, \d t)^{1/p}<\infty$ if $p\in [1,\infty)$
 and 
 $\|\phi\|_{L^\infty}=\esssup_{t\in [0,T]}|\phi_t|<\infty$.
 \item 
 $C([0,T];E)$ 
 is the space of  continuous  functions
 $\phi:[0,T]\to E$ endowed   with the norm 
 $\|\cdot\|_{L^\infty}$.
\item
 $\cS^2(0,T;E)$
  is the space of 
$\sF$-progressively  measurable c\`{a}dl\`{a}g
processes
$X: \Om\t [0,T]\to E$ 
satisfying $\|X\|_{\cS^2}=\sE[\esssup_{t\in [0,T]}|X_t|^2]^{1/2}<\infty$;
\item 
$\cM(E)$ is the set of   measures on $E$,
 $\cP(E)$ is the set of probability measures on $E$,
and  $\cP_2(E)$ is the set of 
square integrable probability measures on $E$ endowed with the $2$--Wasserstein distance. 
\end{itemize}
For each $\mu\in \sR^n$ and $\Sigma\in\ol{\sS^n_+}$,
we denote by $\cN(\mu,\Sigma)$  the Gaussian measure on $\sR^n$ with mean $\mu$ and  covariance matrix $\Sigma$.
We also write 
$\sN_{  0}=\sN\cup\{0\}$
for notation simplicity.

\section{Problem formulation  and main results}
This section introduces    exploratory LQC  problems,
proposes a class of geometry-aware PG  algorithms 
to seek the optimal policy,
and presents their convergence properties.

\subsection{Regularised  stochastic LQ control  problems
with indefinite costs}
This section  recalls the        regularised LQC problem
introduced in \cite{wang2020reinforcement,wang2020continuous}
and  its optimal feedback controls.
Let  $T>0$ be a finite time horizon, 
  $(\Omega, \cF,  \sP)$
be a  complete filtered probability space
on which a
$d$-dimensional 
 standard    Brownian motion  $W =
(W_t)_{t\ge 0}$ is defined,
and   $\sF = (\cF_t)_{t\ge 0}$ 
 be the natural filtration of $W$ augmented by 
   an   independent $\sigma$-algebra $\cF_0$.

We first   introduce the admissible controls and the associated state dynamics. 
Let 
 $\cA$ be  the set of (relaxed)  controls 
$\mathfrak{m}: \Om \to \cM([0,T]\t \sR^k)$ 
such that 
$\mathfrak{m}_t(\d t,\d a)=\mathfrak{m}_t(\d a)\d t$  for a.e.~$t\in [0,T]$, where
$\mathfrak{m}_t: \Om\to   \cP(\sR^k)$ is $\cF_t$-measurable for all $t\in [0,T]$
and $\sE[\int_0^T \int_{\sR^k} |a|^2 \mathfrak{m}_t(\d a)\d t]<\infty$.
For each $\mathfrak{m}\in \cA$, 
consider the following   controlled dynamics:
\bb\label{eq:reg_state_open_loop}
\d X_t =\Phi_t(X_t, \mathfrak{m}_t ) \, \d t+ \Gamma_t(X_t, \mathfrak{m}_t ) \, \d W_{t}, \q t\in [0,T];
\q X_0=\xi_0,
\ee
where  
 $\xi_0\in L^2(\Om;\sR^d)$ is a given $\cF_0$-measurable   random variable, 
and the functions $\Phi:[0,T]\t \sR^d\t \cP_2(\sR^k)\to \sR^d$ and 
$\Gamma:[0,T]\t \sR^d\t \cP_2(\sR^k)\to \ol {\sS^d_+}$
satisfy
for all $(t,x,m)\in [0,T]\t \sR^d\t \cP_2(\sR^k)$,
\begin{equation}
\label{eq:reg_coefficient}
	\Phi_t(x, m) = \int_{\sR^k}(A_tx+B_ta)\,m(  \d a),
\quad 
	\Gamma_t(x, m )=\left(\int_{\sR^k} (C_tx+D_ta)(C_t x+D_ta)^\top \, m( \d a)\right)^{\frac{1}{2}},
\end{equation} 
where $(\cdot)^\frac{1}{2}:\ol {\sS^d_+}\to \ol {\sS^d_+}$ is the  matrix square root such that 
$M^{\frac{1}{2}}(M^{\frac{1}{2}})^\top =M$ for all $M\in \ol {\sS^d_+}$,
and   $ A,B,C,D$ are 
  measurable functions 
such that 
\eqref{eq:reg_state_open_loop}
admits a unique  strong solution $X^\mathfrak{m}\in \cS^2(0,T;\sR^d)$
(see 
(H.\ref{assum:coefficident}) for   precise conditions). 

{The state dynamics \eqref{eq:reg_state_open_loop}
is commonly referred to as an exploratory dynamics  
(see, e.g.,~\cite{wang2020reinforcement,wang2020continuous,vsivska2020gradient}).
It models   interacting with the  system  by repeatedly sampling  random actions   according to  a  given measure-valued control   $\mathfrak{m}$. 
As a consequence of these random  actions, the system's state evolves
with the  aggregated coefficients \eqref{eq:reg_coefficient},
which indicates that 
 the infinitesimal change of the  state at $t$ has a  mean and  variance  integrated with respect to    the sampling  distribution $\mathfrak{m}_t$.
In the special case where    $\mathfrak{m}_t(\d t,\d a)=\boldsymbol{\delta}_{\a_t}(\d a)\d t$
for some $\alpha_t:\Om\t [0,T]\to \sR^k$,
with $\boldsymbol{\delta}_a$ being the  Dirac measure   on $a\in \sR^k$,
  \eqref{eq:reg_state_open_loop} 
simplifies into
\bb
\label{eq:state_strict}
\d X_t =(A_tX_t+B_t  \alpha_t ) \, \d t+(C_t X_t +D_t\alpha_t ) \, \d W_{t}, \q t\in [0,T];
\q X_0=\xi_0,
\ee
which is the   dynamics studied in  the classical LQC problem \cite{yong1999stochastic}.
 See   the end of Section \ref{sec:mesh-independence}  
for more details on the connection   between  an  exploratory state dynamics and
controlling \eqref{eq:state_strict} with  random actions. 

}

We now 
 consider 
minimising the following    cost functional over  all ${\mathfrak{m}\in\cA}$, {which is  known as the exploratory/entropy-regularised control problem  
\cite{wang2020reinforcement,wang2020continuous,vsivska2020gradient,jia2022policy, jia2023q}}:
 \begin{align}\label{eq:reg_cost_open_loop}
    \begin{split}
   %
  \sE\bigg[
\int_0^T\int_{\sR^k}\left(
\frac{1}{2}
\left \la \begin{pmatrix}
Q_t &S^\top_t
\\
S_t & R_t
\end{pmatrix}
  \begin{pmatrix}
X^{\mathfrak{m}}_t
\\
a 
\end{pmatrix},
  \begin{pmatrix}
X^{\mathfrak{m}}_t
\\
a 
\end{pmatrix}
\right\ra  
\,\mathfrak{m}_t(\d a) + \rho\cH(\mathfrak{m}_t\| \ol{\mathfrak{m}}_t) \right)\d t
 +\frac{1}{2}(X^{\mathfrak{m}}_T)^\top G X^{\mathfrak{m}}_T
\bigg],
    \end{split}
\end{align}
where $X^{\mathfrak{m}}$ satisfies   the state dynamics \eqref{eq:reg_state_open_loop}.
Here
 $Q,S, R$  are given matrix-valued functions of proper dimensions,
 $G\in \sR^{d \t d} $ and  
 $\rho\ge  0$ are given constants, $(\ol{\mathfrak{m}}_t)_{t\in [0,T]}$ are given   measures on $\sR^k$,
and  
for each 
$t\in [0,T]$, 
 $\cH(\cdot \Vert \ol{\mathfrak{m}}_t):
\cP(\sR^k)\to [0,\infty]$ 
is the relative  entropy 
with respect to $\ol{\mathfrak{m}}_t$ such that  
 for all $m \in \cP(\sR^k)$,
$$
  \cH(m \Vert \ol{\mathfrak{m}}_t)
=\begin{cases}
\int_\sR \ln \big(\frac{m(\d a)}{\ol{\mathfrak{m}}_t(\d a)}\big)\,m(\d a), 
& \textnormal{$ m$   is absolutely continuous with respect to $\ol{\mathfrak{m}}_t$,}
\\
\infty,
& \textnormal{otherwise.}
\end{cases}
 $$
 Note that   the cost \eqref{eq:reg_cost_open_loop} 
is aggregated with respect to the control distribution  $\mathfrak{m}_t$  from which   the random  actions are sampled.
The entropy    $ \cH(\cdot\| \ol{\mathfrak{m}}_t)$ 
serves as a regularisation term 
 to encourage the minimiser of \eqref{eq:reg_cost_open_loop} to be close to the provided reference measures $(\ol{\mathfrak{m}}_t)_{t\in [0,T]}$,
and the  weight parameter $\rho\ge 0$ controls the strength of this regularisation.

 {The entropy-regularised control problem \eqref{eq:reg_cost_open_loop}, initially introduced in \cite{wang2020reinforcement}, represents a natural extension of the well-established regularised MDPs (see e.g., \cite{geist2019theory, mei2020global}) into the continuous domain. 
Common choices of $(\ol{\mathfrak{m}}_t)_{t\in [0,T]}$  in the existing literature
include Gibbs measures \cite{vsivska2020gradient}
and the Lebesgue measure \cite{wang2020reinforcement,wang2020continuous,firoozi2022exploratory}.
}
 
 The following   assumptions
 on the coefficients of  \eqref{eq:reg_state_open_loop}-\eqref{eq:reg_cost_open_loop}
  are   imposed  throughout this paper.

\begin{Assumption}
\phantomsection
\label{assum:coefficident}
\begin{enumerate}[(1)]
\item\label{assum:integrability}
$T>0$, 
$\xi_0\in L^2(\Om;\sR^d)$,
$A\in L^1(0,T;\sR^{d\t d})$,
$B\in L^2(0,T;\sR^{d\t k})$,
$C\in L^2(0,T;\sR^{d\t d})$, 
$D\in L^\infty(0,T;\sR^{d\t k})$,
$Q\in L^1(0.T;\sS^{d})$,
$S\in L^2(0.T;\sR^{k\t d})$,
$R\in L^\infty(0,T; \sS^k)$ 
and $G\in \sS^d$.
\item \label{item:regulariser}
$\rho>0$,
   $\ol{\mathfrak{m}}_t=\cN(0,\bar{V}_t)$ for all $t\in [0,T]$,
   $\bar{V} \in L^\infty(0,T; \sS^k_+)$
and     
$\bar{V} \succeq \delta I_k$ for some $\delta>0$.

\end{enumerate}
\end{Assumption}

\begin{Remark}
\label{rmk:regularity_solv}
Condition 
   (H.\ref{assum:coefficident}\ref{assum:integrability})
ensures that 
for all  $\mathfrak{m}\in \cA$,
\eqref{eq:reg_state_open_loop} admits a unique  strong solution
 in $ \cS^2(0,T;\sR^d)$
(see Proposition \ref{prop:state_wp}),
and the associated regularised cost  
 is well-defined. 
Note that 
 (H.\ref{assum:coefficident}\ref{assum:integrability}) allows  the   coefficients $Q, S, R$ and $G$
 to be indefinite or even negative definite
(provided that (H.\ref{assum:nondegeneracy}) holds).
 Such a control problem is often called indefinite stochastic LQ problem 
 (see e.g.~\cite{sun2016open} and the references therein) 
 and has important applications in 
 optimal liquidation \cite{cartea2018algorithmic} and 
 mean-variance portfolio selection  
 \cite{zhou2000continuous} in finance. 
 
 {Condition              (H.\ref{assum:coefficident}\ref{item:regulariser})
assumes that 
for each $t\in [0,T]$,
the reference  measure
$\ol{\mathfrak{m}}_t$ 
in \eqref{eq:reg_cost_open_loop} 
is a Gaussian measure.
This ensures that the optimal strategy of \eqref{eq:reg_state_open_loop}-\eqref{eq:reg_cost_open_loop}
is Gaussian (see \eqref{eq:KV_star}), which  in turn  implies that
\eqref{eq:reg_state_open_loop}-\eqref{eq:reg_cost_open_loop} 
can be reformulated 
as an optimisation problem over Gaussian policies.
A similar reformulation also holds  if 
  $\ol{\mathfrak{m}}_t$ is the Lebesgue measure
\cite{wang2020reinforcement,wang2020continuous,firoozi2022exploratory},
and our proposed policy descent algorithm and its convergence analysis can   be naturally extended to this case. 
}

 \end{Remark}

We also impose 
 the following well-posedness condition of the corresponding Riccati equation 
  for the closed-loop solvability of 
  the (possibly indefinite)  control problem 
 \eqref{eq:reg_state_open_loop}-\eqref{eq:reg_cost_open_loop}.

 \begin{Assumption}
  \label{assum:nondegeneracy}
There exists 
 $P^\star\in C([0,T];\sS^d)$
satisfing the following Riccati equation:
for a.e.~$t \in [0,T]$, 
 \begin{equation}
 \label{eq:riccati_star_ent}
\left\{ 
\begin{aligned}
&(\tfrac{\d}{\d t}{P})_t +A_t^\top P_t + P_t A_t +C_t^\top P_t C_t + Q_t 
	\\
	&\quad 
	-(B^\top_t P_t + D_t^\top P_t C_t+S_t)^\top
	(D^\top_t P_t D_t+R_t+\rho \bar{V}_t^{-1})^{{-1}}(B^\top_t P_t + D_t^\top P_t C_t+S_t)=0;
\\
&	P_T =G,
\end{aligned}
\right.
\end{equation}
and 
$D^\top P^\star D+R+\rho \bar{V}^{-1}\succeq  \widetilde{\delta} I_k$ for some $\widetilde{\delta}>0$.
 \end{Assumption}

 \begin{Remark}
 \label{rmk:solvability_riccati}
 
Condition  (H.\ref{assum:nondegeneracy})
is   
 called  the  strongly regular solvability   of  
\eqref{eq:riccati_star_ent} in  \cite{sun2016open} 
and 
ensures that \eqref{eq:reg_state_open_loop}-\eqref{eq:reg_cost_open_loop}
admits an optimal feedback control.
{Note that  it suffices to  
 assume  the existence of a strongly regular solution,
 as the uniqueness of a strongly regular  solution  to  \eqref{eq:riccati_star_ent} follows directly from Gronwall's inequality  (see \cite{sun2016open} and also \cite[Proposition 7.1, p.~319]{yong1999stochastic}).}
 One can easily show that 
  (H.\ref{assum:nondegeneracy}) holds 
 if the unregularised \eqref{eq:riccati_star_ent}  is strongly regular solvable, i.e., 
\eqref{eq:riccati_star_ent}   with $\rho=0$
admits a solution $P^{\star,0}\in C([0,T];\sS^n)$
and $D^\top P^{\star,0} D+R \succeq  \widetilde{\delta} I_k$. 
 This is due to the fact that 
$P^\star\succeq P^{\star,0} $
(see 
\cite[Theorem 5.3]{sun2016open}), and hence
$D^\top P^\star  D+R+\rho \bar{V}^{-1}\succeq  
D^\top P^{\star,0} D+R$ by    (H.\ref{assum:coefficident}\ref{item:regulariser}).

Moreover, 
by virtue of   the regularisation term $\rho \bar{V}^{-1}$, 
(H.\ref{assum:nondegeneracy})
may hold 
even when the unregualised LQ problem (with $\rho=0$) is not closed-loop solvable. 
This indicates that    the entropy term 
$ \rho\cH(\cdot\| \ol{\mathfrak{m}}_t) $
indeed   regularises the cost landscape. 
 Such a regularisation effect may not hold 
 if   
 the reference measure
 $\ol{\mathfrak{m}}_t$, $t\in [0,T]$,  is chosen as 
  the Lebesgue measure $\cL_k$ on $\sR^k$.
In fact, 
as shown in 
\cite{wang2020reinforcement, wang2020continuous, firoozi2022exploratory},
if  $\ol{\mathfrak{m}}_t=\cL_k$ for all $t\in [0,T]$,
then 
the closed-loop solvability of the regularised problem is equivalent to that of the unregularised problem,
and 
 the entropy term 
will not   modify the cost landscape
over policies.

 \end{Remark}

 Under (H.\ref{assum:coefficident})
 and (H.\ref{assum:nondegeneracy}),
standard verification arguments 
(see, e.g., \cite{yong1999stochastic})
show that 
the optimal control $\mathfrak{m}^\star\in \cA$
of    \eqref{eq:reg_cost_open_loop} 
is of the   form
$
\mathfrak{m}^\star_t = \nu^\star_t(X^{\mathfrak{m}^\star}_t ),
$
where 
$\nu^\star:[0,T]\t \sR^d\to \cP_2(\sR^k)$ satisfies 
for all $(t,x)\in [0,T]\t \sR$,
$\nu^\star_t(x)=\cN(K^\star_t x,V_t^\star)$
and 
\begin{align}
\label{eq:KV_star}
\begin{split}
K^\star_t &= -(D^\top_t P^\star_t D_t+R_t+\rho \bar{V}_t^{-1})^{{-1}}(B^\top_t P^\star_t + D_t^\top P^\star_t C_t+S_t),
\\ 
V^\star_t &=\rho(D^\top_t P^\star_t D_t+R_t+\rho \bar{V}_t^{-1})^{{-1}}.
\end{split}
\end{align}
By (H.\ref{assum:coefficident})
and (H.\ref{assum:nondegeneracy}),
$K^\star\in L^2(0,T;\sR^{k\t d})$, 
$V^\star\in L^\infty(0,T; {\sS^{k}_+})$
and 
$V^\star\succeq  \eps I_k$ for  some $\eps>0$.
{Note that the optimality of $\mathfrak{m}^\star$ in $\cA$
implies that the policy $\nu^\star$ is optimal among all Markovian feedback controls
$\nu:[0,T]\times \sR^d\to \cP_2(\sR^k)$
for which 
 the resulting open-loop control $\mathfrak{m}_\cdot= \nu_\cdot(X^{\nu}_\cdot)$ is square integrable. Here, $X^{\nu}$ denotes the state dynamics controlled by $\nu$, as defined  in \eqref{eq:reg_state_closed_loop}.  

}
\subsection{Optimisation   over Gaussian policies and landscape analysis}
\label{sec:landscape}

Motivated by 
the optimal    Gaussian policy    $\nu^\star$
in \eqref{eq:KV_star},
this section   reformulates  \eqref{eq:reg_state_open_loop}-\eqref{eq:reg_cost_open_loop} 
 as an equivalent minimisation problem over Gaussian policies,
 and presents   key properties of the optimisation landscape $\cC:\Theta\to \sR$.
 The proofs of these properties will be given  in Section \ref{sec:proof_cost_lanscape}.

\paragraph{Policy optimisation.}
 
Let $\Theta$ be the following parameter space 
$$
\Theta\coloneqq 
\left\{\theta= (K,V)\in \cB(0,T; \sR^{k\t d}\t \sS^k_+) 
\,\Big\vert\,
\|K\|_{L^2}<\infty, 
\;
\eps I_k\preceq V\preceq \tfrac{1}{\eps} I_k
\; \textnormal{for some $\eps>0$}
\right\},
  $$
  and $\cV$ be the  
space of   Gaussian policies parameterised by $\Theta$: 
\bb
\label{eq:Gaussian_V}
\cV \coloneqq
\left\{\nu^\theta: [0,T]\t \sR^d
\ni (t,x)\mapsto 
 \cN(K_tx,V_t)
\in \cP(\sR^k)
\,\Big\vert\,
 \theta= (K,V)\in \Theta
 \right\}.\footnotemark
\ee
  \footnotetext{
As  
 $\rho>0$,
 we require the Gaussian policies in  $\cV$ 
 to have   nondegenerate covariances. 
If $\rho=0$,
one can restrict   admissible policies 
to be $\nu^\theta_t (x)=\cN(K_t x,0)$. 
Our analysis and results can be naturally extended to this setting.
}%
We shall  identify 
$\nu^\theta\in \cV$ with its parameter
$\theta=(K,V)\in \Theta$.
For each $\nu^\theta\in \cV$, 
consider  the associated controlled  
 dynamics (cf.~\eqref{eq:reg_state_open_loop}):
\bb\label{eq:reg_state_closed_loop}
\d X_t =\Phi_t(X_t, \nu^\theta_t(X_t) ) \, \d t+ \Gamma_t(X_t, \nu^\theta_t(X_t) ) \, \d W_{t}, \q t\in [0,T];
\q X_0=\xi_0,
\ee
with $\Phi$ and $\Gamma$ defined in \eqref{eq:reg_coefficient},
and let 
 $X^\theta\in \cS^2(0,T;\sR^d)$ be the 
 unique solution to 
\eqref{eq:reg_state_closed_loop} 
(see  Proposition \ref{prop:state_wp}).
 Then  we consider  minimising the following   cost functional:
 \begin{align}\label{eq:reg_cost_closed_loop}
    \begin{split}
   %
\cC(\theta)\coloneqq 
 & \sE\bigg[
\int_0^T\int_{\sR^k}\left( 
\frac{1}{2}
\left \la \begin{pmatrix}
Q_t &S^\top_t
\\
S_t & R_t
\end{pmatrix}
  \begin{pmatrix}
X^{\theta}_t
\\
a 
\end{pmatrix},
  \begin{pmatrix}
X^{\theta}_t
\\
a 
\end{pmatrix}
\right\ra   \,\nu^\theta_t(X^\theta_t;\d a) + \rho\cH(\nu^\theta_t(X^\theta_t)\| \ol{\mathfrak{m}}_t) \right)\d t
\\
&\q +\frac{1}{2}(X^{\theta}_T)^\top G X^{\theta}_T
\bigg]
    \end{split}
\end{align}
over all $\theta\in \Theta$, or equivalently all $\nu^\theta\in \cV$.
It is clear that
 the cost   $\cC$ is minimised at $\theta^\star= (K^\star,V^\star)$ defined in 
\eqref{eq:KV_star},
and the minimum value $\inf_{\theta\in \Theta} \cC(\theta)$ is the   minimum cost of 
 \eqref{eq:reg_state_open_loop}-\eqref{eq:reg_cost_open_loop}.
 
 \paragraph{Optimisation landscape.}

To investigate the regularity  of the map 
$\cC:\Theta\to \sR$, 
we introduce two important quantities: 
for each $\theta=(K,V)\in \Theta$,
let 
$P^\theta\in C([0,T];\sS^d)$ 
be the   solution to 
following (backward) Lyapunov equation:
\begin{align}
\label{eq:lyapunov_reg}
\begin{split}
(\tfrac{\d }{\d t}P)_t + & (A_t+B_tK_t)^\top P_t + P_t^\top (A_t+B_tK_t) 
+(C_t+D_tK_t)^\top P_t (C_t+D_tK_t) 
\\
+ & K_t^\top (R_t+\rho\bar{V}_t^{{-1}})K_t
+ S_t^\top K_t+K_t^\top S_t
 + Q_t = 0,
\quad \textnormal{a.e.~$t \in [0,T]$}; 
\quad 
P_T  = G,
\end{split}
\end{align}
and let $\Sigma^\theta\in C([0,T]; \ol{\sS^d_{+}})$
be the   solution to the following Lyapunov equation:
for {a.e.~$t \in [0,T]$, 
\begin{align}
	\begin{split}\label{eq:lq_sde_K_reg_cov}
			(\tfrac{\d }{\d t}{\Sigma})_t=& (A_t+B_tK_t)\Sigma_t+ \Sigma_t(A_t+B_tK_t)^\top +  (C_t+D_tK_t)\Sigma_t(C_t+D_tK_t)^\top +  D_t V_t D_t^\top, \\
		\Sigma_0=&\mathbb{E}[\xi_0\xi^\top_0].
	\end{split}
\end{align}
Under (H.\ref{assum:coefficident}),
$P^\theta$ and $\Sigma^\theta$  are  well-defined
by  standard well-posedness results of linear differential equations. 
Note that $P^\theta$ depends only on $K$ and is independent of $V$.
Moreover, 
 let  $X^\theta $ be the state process  governed by \eqref{eq:reg_state_closed_loop},
then 
$\Sigma^\theta_t =\sE[X^\theta_t(X^\theta_t)^\top]$
for all $t\in [0,T]$,\footnotemark
due to a straightforward application of 
  It\^{o}'s formula to 
$t\to  X^\theta_t(X^\theta_t)^\top$ 
and the definition \eqref{eq:reg_coefficient}
(see also  Lemma  \ref{lemma:representation_KV}). 

\footnotetext{
{Given a state variable $X_t$,
the second-moment matrix $\Sigma_t=\sE[X_tX^\top_t]$ is often referred to as the state covariance matrix 
in the reinforcement learning literature (see e.g.,  \cite{fazel2018global, hambly2020policy}).
We follow this convention throughout this paper. 
}
}

Based on the notation $P^\theta$ and $\Sigma^\theta$,
the following proposition characterises the  Gateaux derivatives 
of $\cC$ at each $\theta\in \Theta$. 
The proof relies on first  reformulating the 
minimisation problem \eqref{eq:reg_cost_closed_loop}
into a 
deterministic control problem for 
$\Sigma^\theta$,
and then applying the Pontryagin  optimality principle. 

\begin{Proposition}
\label{prop:Gateaux}
Suppose (H.\ref{assum:coefficident}) holds.
For each $\theta\in \Theta$,
let   $P^\theta\in C([0,T];\sS^d)$ satisfy     \eqref{eq:lyapunov_reg},
and let 
 $\Sigma^\theta\in C([0,T]; \ol{\sS^d_{+}})$
 satisfy  \eqref{eq:lq_sde_K_reg_cov}. 
Then for all $\theta, \theta'\in \Theta$,
\begin{align*}
\frac{\d}{\d \eps} \cC(K+\eps K', V)\Big|_{\eps=0}
&=\int_0^T  \la 
\cD_K(\theta)_t\Sigma^\theta_t,
K'_t
\ra\,\d t,
\\
\frac{\d}{\d \eps} \cC(K, V+\eps( V'-V))\Big|_{\eps=0}
&=\int_0^T  \la 
\cD_V(\theta)_t,
V'-V
\ra\,\d t,
\end{align*}
where 
for a.e.~$t\in [0,T]$,
  \begin{align}
\cD_K(\theta)_t&\coloneqq   B_t^\top P^\theta_t
+D_t^\top P^\theta_t(C_t+D_tK_t) 
+S_t+
 ( R_t + {\rho}\bar{V}_t^{{-1}})K_t,
 \label{eq:D_K}
 \\
\cD_V(\theta)_t&\coloneqq
 \frac{1}{2}
 (D_t^\top P^\theta_t D_t
+ R_t
+{\rho}(\bar{V}_t^{{-1}}-{V}^{{-1}}_t)).
 \label{eq:D_V}
\end{align}
\end{Proposition}

We then estimate the regularity  of 
$\cC:\Theta\to \sR$  by using  the gradient terms 
$\cD_K(\theta)$ and $\cD_V(\theta)$.
The following proposition proves that the   functional $\cC$ satisfies  
a non-uniform {\L}ojasiewicz condition 
in $\theta$. 
As $\cC$ is typically nonconvex in $K$
 (see Proposition \ref{prop:non_coercive}),
such a   {\L}ojasiewicz condition 
is  critical for the  global convergence   of gradient-based algorithms. 

\begin{Proposition}
\label{prop:Lojaiewicz}
Suppose (H.\ref{assum:coefficident}) 
and (H.\ref{assum:nondegeneracy})
hold.
Let $\theta^\star\in \Theta$ be 
defined by \eqref{eq:KV_star}.
For each $\theta\in \Theta$,
let
   $P^\theta\in C([0,T];\sS^d)$ satisfy     \eqref{eq:lyapunov_reg},
 let 
   $\Sigma^\theta \in C([0,T];\ol{\sS^d_+})$ satisfy \eqref{eq:lq_sde_K_reg_cov},
   and let 
   $\cD_K(\theta)$ and 
 $\cD_V(\theta)$ be defined by 
 \eqref{eq:D_K} and \eqref{eq:D_V}, respectively.
Then for all $\theta\in \Theta$,
   \begin{align}
 \label{eq:gd_KV_statement}
 \begin{split}
 \cC(\theta)-\cC(\theta^\star)
& \le
\int_0^T 
\bigg(
\frac{1}{2}  \la
(D^\top_t  P^\theta_t  D_t   +R_t+\rho \bar{V}_t^{-1})^{-1} \cD_K(\theta)_t, \cD_K(\theta)_t\Sigma^{\theta^\star}_t\ra
\\
&\quad \quad 
+\frac{1 }{\rho }\max(\|V^\star_t\|^2_2,  \|  V_t\|^2_2 ) |\cD_V(\theta)_t|^2
\bigg)
\, \d t.
\end{split}
\end{align}
 \end{Proposition}
 
The next proposition proves that for any $\theta,\theta'\in \Theta$,
the cost difference $\cC(\theta')-\cC(\theta)$ can be upper bounded by the first and second order terms in $\theta'-\theta$. Such a  property is often referred to as the 
``almost smoothness" condition in the
 literature on PG methods
(see e.g., \cite{fazel2018global, hambly2020policy, zhang2021policy}). 

\begin{Proposition}
\label{prop:Lipshcitz_smooth}
Suppose (H.\ref{assum:coefficident}) holds.
For each $\theta\in \Theta$,
let
   $P^\theta\in C([0,T];\sS^d)$ satisfy     \eqref{eq:lyapunov_reg},
     let 
   $\Sigma^\theta \in C([0,T];\ol{\sS^d_+})$ satisfy \eqref{eq:lq_sde_K_reg_cov},
   and let 
   $\cD_K(\theta)$ and 
 $\cD_V(\theta)$ be defined by 
 \eqref{eq:D_K} and \eqref{eq:D_V}, respectively.
Then for all $\theta,\theta'\in \Theta$,
 \begin{align*}
 \cC(\theta')-\cC(\theta)
 &\le 
\int_0^T 
\bigg(
\la  K'_t-K_t, \cD_K(\theta)_t\Sigma^{\theta'}_t\ra 
+
\frac{1}{2}\la K'_t-K_t , (D^\top_t  P^\theta_t  D_t   +R_t+\rho \bar{V}_t^{-1}) (K'_t-K_t)\Sigma^{\theta'}_t\ra
\\
&\quad +\la 
\cD_V(\theta)_t, V'_t-V_t\ra 
+
\frac{\rho}{4} 
\frac{ | V'_t-V_t  |^2}{
 \min(\lambda^2_{\min}(V_t),    \lambda^2_{\min}(V'_t) )
 }
\bigg)\, \d t.
\end{align*}

 \end{Proposition}

 Note that the {\L}ojasiewicz   condition  in Proposition
\ref{prop:Lojaiewicz} and 
the smoothness condition in Proposition 
 \ref{prop:Lipshcitz_smooth} are local properties. 
The estimates therein depend explicitly on $P^\theta$ and $\Sigma^\theta$,
 which admit no     uniform   bound  
 over the unbounded parameter set  $  \Theta$.  
 For PG methods with finite-dimensional parameter spaces, 
this difficulty is often  overcome by 
first proving the sublevel set    $ \{\theta\in \Theta\mid \cC(\theta)<\beta\}$
is bounded for any $\beta>0$, 
and then designing algorithms
 whose iterates remain  in a fixed   sublevel set (see e.g.,~\cite{fazel2018global,
gravell2020learning,
 hambly2020policy}). 
 However, the following example shows that  
 in the   setting with continuous-time policies,
 the cost is typically noncoercive,\footnotemark
 \footnotetext{ Let $(X,\|\cdot\|)$ be a normed space. A function $f:X\to \sR$ is called 
 coercive if $\lim_{\|x\|\to \infty}f(x)=\infty$.}
 and hence
the above  argument  cannot be applied.  
  The proof follows from a straightforward computation,
  and is given in Appendix  \ref{appendix:technical}.

\begin{Proposition}
\label{prop:non_coercive}
Let $\cC:L^2(0,1; \sR)\to \sR$ be such that  
for all $K\in L^2(0,1;\sR)$,
 \begin{align}
\cC(K)\coloneqq 
 & 
\int_0^1 
\left(K_tX_t\right)^2\, \d t,
\quad \textnormal{
with $X_t=1+\int_0^t  K_s X_s\, \d s$,\; $t\in [0,1]$.
}
\end{align}
Then $\cC:L^2(0,1; \sR)\to \sR $  is neither coercive nor   quasiconvex. 
In particular, 
  let $K^\eps\in L^2(0,1;\sR)$, $\eps>0$,
be such that $K^\eps_t=- (1+\eps-t)^{-1}$ for all $t\in [0,1]$. Then $\lim_{\eps\to 0}\|K^\eps\|_{L^1}=\infty$ and 
$\sup_{\eps>0} \cC(K^\eps) =1$.
Moreover, there exists $\eps_0>0$ such that 
for all $\eps\in (0,\eps_0]$, 
$\cC(0.5 K^{\eps})>\max\{\cC(\boldsymbol{0}),\cC(K^{\eps})\}$,
with $\boldsymbol{0}$ being the zero function.

\end{Proposition}

\subsection{Policy gradient method and its convergence analysis}
\label{sec:PGM} 
This section proposes   a 
geometry-aware PG method for \eqref{eq:KV_star} 
that preserves an a-priori bound,
and 
proves  its global linear convergence 
based on the landscape properties in Section \ref{sec:landscape}.


\paragraph{Geometry-aware
policy gradient method.}
For each initial guess $\theta^0=(K^0,V^0)\in \Theta$
and 
   stepsize $\tau>0$, 
 consider   $(\theta^n)_{n\in \sN }\subset \cB(0,T; \sR^{k\t d}\t \sS^k) $
 such that for all $n\in \sN_0$, 
\begin{align}
\label{eq:NPG}
K^{n+1}_t&=K^n_t-\tau \cD_K(\theta^n)_t, 
\quad 
V^{n+1}_t=V^n_t-\tau
\cD^{\rm bw}_V(\theta^n)_t,
\q \textnormal{a.e.~$t\in [0,T]$},
\end{align}
with 
\bb\label{eq:gradient_BW}
\cD^{\rm bw}_V(\theta^n)_t= 
\cD_V(\theta^n)_t V_t^n+V^n_t \cD_V(\theta^n)_t,
\footnotemark
\ee
where $\cD_K(\theta)$ and $\cD_V(\theta)$ are defined 
by 
 \eqref{eq:D_K} and 
  \eqref{eq:D_V}, respectively.
Here 
 we update $K$ and $V$ with the same  stepsize $\tau$  for the  clarity of presentation,
  but the   results can be naturally extended to the setting where different constant stepsizes are adopted to update 
$K$ and $V$. 

\footnotetext{
For an arbitrary stepsize $\tau>0$,
$(V^{n})_{n\in \sN}$ may  not be positive definite and hence may not be invertiable.
In this case, $\cD_V$ is defined by  replacing ${V}^{{-1}}_t$ in   \eqref{eq:D_V} with the (symmetric) Moore-Penrose inverse of $V_t$.
We   prove that $(\theta^n)_{n\in \sN }\subset \Theta$
for all sufficiently small stepsizes
(see  Proposition \ref{prop:uniform_bound}).
 }
 
Algorithm \eqref{eq:NPG}
normalises the    (Fr\'{e}chet) derivatives of $\theta^n$
(cf.~Proposition \ref{prop:Gateaux})
to incorporate  
 the local  geometry of the parameter space. 
Specifically,
it  updates  
   $(K^n)_{n\in \sN}$  
 by  the   steepest descent     on the   manifold of Gaussian policies
  endowed with the Fisher information metric
  (also known as the natural gradient). 
To see this, for each $n\in \sN_0$,
  consider the following   
  natural gradient 
   update for $K^n$
   (see \cite{kakade2001natural}):
 \bb
 \label{eq:inexact_npg_conts}
{K^{n+1}_t}=  {K^n_t}-\tau   {\cI}({\theta}^n)^{-1}_t  {\nabla_K \cC(\theta^n)_t},\footnotemark
\q \textnormal{a.e.~$t\in [0,T]$,}
\footnotetext
{  For each $A\in \sR^{kd\t kd}$ and $B\in \sR^{k\t d}$,
indexed by $A_{ij,i'j'}$ and $B_{ij}$
with  $i,i'\in \{1,\ldots, k\}$ and 
 $j,j'\in \{1,\ldots, d\}$, 
   we define
   $AB\in \sR^{k \t d}$ with
   $(AB)_{ij}= \sum_{k,l}A_{ij,kl}B_{kl}$.
   This is equivalent to reshaping $B$ 
(with   row-major ordering)
   into a vector,   performing the  standard matrix-vector multiplication,
   and reshaping the result into a   matrix.
}
\ee
where 
$ \nabla_K \cC(\theta^n)
=\cD_K(\theta^n)\Sigma^{\theta^n}$
is the   derivative in $K^n$, 
$ {\cI}({\theta}^n)_t\in \sR^{kd\t kd}$ 
is    the    Fisher information matrix
satisfying
for all $i,i'\in \{1,\ldots, k\}$ and 
 $j,j'\in \{1,\ldots, d\}$, 
 $$
( {\cI}(\theta^n)_t)_{ij,i'j'}  \coloneqq \sE
\left[
\int_{\sR^k}
 \left[\p_{(K^n_t)_{ij} }
 \ln \left(\hat{\nu}^{\theta^n}_t(X^{\theta^n }_t; a)\right)
\p_{(K^n_t)_{i'j'} }
\ln \left (\hat{\nu}^{\theta^n }_t(X^{\theta^n }_t; a)\right)
\right]
\, \hat{\nu}^{\theta^n}_t(X^{\theta^n}_t; a) \, \d a
\right],
 $$
%
 and $\hat{\nu}^{\theta^n}_t(X^{\theta^n}_t; \cdot)$ is the   density of 
 $ \cN(K^n_t X^{\theta^n}_t, I_k)$.
Then by  a  similar computation
as in  \cite{fazel2018global, hambly2021policy}, 
$
 {\cI}(\theta^n)^{-1}_t  \nabla_K \cC(\theta^n)_t
=
{  \nabla_K \cC(\theta^n)_t (\Sigma^{\theta^n}_t)^{-1}}
=\cD_K(\theta^n)_t
 $.

On the other hand, 
\eqref{eq:NPG} updates 
   $(V^n)_{n\in \sN}$  
by the steepest descent  on the matrix   manifold   $\sS^k_+$ endowed with the Bures-Wasserstein   metric
\cite{han2021riemannian}.
It corresponds to   the geometry induced by 
the 2-Wasserstein metric over the space of 
 centered nondegenerate
Gaussian measures. 
By normalising 
$\cD_V$  
  according to  $V$,
  the Riemannian gradient
$\cD^{\rm bw}_V$ in 
 \eqref{eq:gradient_BW}
preserves a pointwise upper and lower  bound of $(V^n)_{n\in \sN}$ without the use of  projection
(see Remark \ref{rmk:implicit_regularisation}).

 \paragraph{Convergence   analysis.}
 The   key challenge  in the convergence analysis 
 of \eqref{eq:NPG} is to establish  a uniform  bound 
 for 
 the corresponding $(P^{\theta^n})_{n\in \sN}$
and $(\Sigma^{\theta^n})_{n\in \sN}$,
as shown in Proposition \ref{prop:uniform_bound}.
  This is achieved by   proving a uniform bound of the iterates $(\theta^n)_{n\in \sN}$ and quantifying the explicit dependence of $\Sigma^\theta$ on $\theta$.
The proof is given in Section \ref{sec:proof_a_priori_bound}
(Propositions \ref{prop:K_bdd},
\ref{prop:V_bound}, and \ref{prop:Sigma_bdd}).

\begin{Proposition}
\label{prop:uniform_bound}
  
 Suppose (H.\ref{assum:coefficident})
  and (H.\ref{assum:nondegeneracy})
 hold.
For each  $\theta\in \Theta$,
let   $P^{\theta}\in C([0,T];\sS^d)$ satisfy  \eqref{eq:lyapunov_reg}
and 
let
   $\Sigma^\theta \in C([0,T];\ol{\sS^d_+})$ satisfy \eqref{eq:lq_sde_K_reg_cov}.
Let $\theta^0\in \Theta$
 and 
  $\ol{\lambda}_{ 0}>0$
 such that $\ol{\lambda}_{0} I_k \succeq D^\top P^{\theta^0} D+R+\rho \bar{V}^{-1}$.
For each 
$\tau>0$,
let $(\theta^n)_{n\in \sN}\subset   \cB(0,T; \sR^{k\t d}\t \sS^k) $ be defined in \eqref{eq:NPG}.
 Then
 \begin{enumerate}[(1)]
 \item 
There exists $\widetilde{C} ,\ol{\lambda}_V,\ul{\lambda}_V>0$
such that 
for all 
$\tau \in (0,1/\ol{\lambda}_0]$,
$n\in \sN_{ 0}$, 
$\|K^n\|_{L^2}\le \widetilde{C}  $
and
$\ul{\lambda}_V I_k
 \preceq V^n
 \preceq \ol{\lambda}_V I_k$.
 \item For all 
$\tau \in (0,2/\ol{\lambda}_0]$,
$n\in \sN_{ 0}$,  $P^{\theta^{n}}\succeq P^{\theta^{n+1}}\succeq P^\star$,
with $P^\star\in C([0,T];\sS^d)$ in  (H.\ref{assum:nondegeneracy}),
 \item
 Assume further that    $\sE[\xi_0\xi_0^\top]\succ0$.
Then  
there exists  $\ol{\lambda}_X,\ul{\lambda}_X> 0$ 
  such that 
for all $\tau \in (0,{1}/\ol{\lambda}_0]$ and $n\in \sN_{  0}$, 
$\ul{\lambda}_X I_d\preceq \Sigma^{\theta^n}\preceq \ol{\lambda}_X I_d$.

 \end{enumerate}

\end{Proposition}

 \begin{Remark}[\textbf{Implicit regularisation}]
 \label{rmk:implicit_regularisation}
The    uniform  bounds of $(K^n)_{n\in \sN}$
and  $(V^n)_{n\in \sN}$
are achieved by an \textit{implicit regularisation} feature of the geometry-aware gradient directions $\cD_K$
and $\cD^{\rm bw}_V$.
Here, ``implicit regularisation"   means that the iterates preserve certain constraints  without  an explicit projection step.
Note that     applying projection to enforce
a  pointwise  lower   bound  for  minimum eigenvalues 
of $(V^n)_{n\in \sN}$  
 is    computationally expensive. It     
 requires performing an eigenvalue decomposition of $V^n_t$ for every time $t\in [0,T]$ and   iteration $n\in \sN$.

A similar implicit regularisation property holds     if $(K^n)_{n\in\sN}$ is updated by a preconditioned natural gradient   descent:
for all $n\in \sN_0$, 
$$
K^{n+1}_t=K^n_t-\tau H^n_t \cD_K(\theta^n)_t, 
\quad 
\textnormal{with $
\tfrac{1}{L} I_k \preceq
H^n 
\preceq LI_k$ for some $ {L}>0$ independent of $n$}.
$$
This includes 
%
the Gauss-Newton   method 
with $H^n=\left(D^\top P^n D+R+\rho\bar{V}^{{-1}}\right)^{-1}$ 
as a special case
(see \cite{fazel2018global,
gravell2020learning}).
However, due to the  noncoercivity of $\cC$,
 it is unclear 
 whether  an implicit regularisation holds
for an arbitrary descent direction of $\cC$ in $K$
(e.g., the   
vanilla gradient direction $\nabla_K \cC(\theta)=\cD_K(\theta)\Sigma^{\theta}$),
in contrast with   PG methods for discrete-time problems \cite{fazel2018global, hambly2021policy};
 see the discussion 
above Proposition \ref{prop:non_coercive}. 

{It is noteworthy that an implicit regularisation feature of natural policy gradient algorithms was observed in   \cite{zhang2021policy}. In their study, 
an agent optimises over   stationary linear policies  to stablise  a linear system with additive noise over an infinite horizon while adhering to robustness constraints on the sup-norm of the input-output transfer matrix. They show that a natural policy gradient algorithm naturally
preserves    the transfer matrix's sup-norm throughout the iterations, eliminating the need for explicit projection.

The challenges faced in the current setting differ from those in   \cite{zhang2021policy}.
Firstly, 
as Proposition \ref{prop:non_coercive} shows, the cost of a finite-horizon continuous-time LQC problem is already noncoercive without any robustness constraints. This is primarily because 
a policy can  have an infinite number of changes in values, occurring at arbitrary time points.
Such a feature is not present in the scenarios studied in \cite{zhang2021policy}, where stationary policies are considered.
Secondly, instead of optimising a deterministic policy, we optimise both the mean and covariance of a Gaussian policy, for which 
we  derive natural gradient updates with respect to different geometries. 
Our result implies    that Wasserstein gradient descent of negative entropy preserves a-priori bounds on the variance of Gaussian measures, which is novel and  of  independent interest.
Lastly, the possible degeneracy of the system noise and the cost coefficients (Remark \ref{rmk:regularity_solv}) necessitates a more precise quantification of the desired implicit regularisation within appropriate function spaces.
 
 }

\end{Remark}

 Proposition \ref{prop:uniform_bound}
     implies  that  the   functional 
$\cC$ 
satisfies   uniform {\L}ojasiewicz and smoothness conditions 
\textit{along the   iterates $(\theta^n)_{n\in \sN}$}.
Based on  this local regularity,  
 the following theorem establishes 
the global linear convergence of 
  \eqref{eq:NPG} for all sufficiently small   stepsizes $\tau$.
 The proof is  given in Section \ref{sec:proof_linear_conv}.

%

 \begin{Theorem}
 \label{thm:linear_conv}
 Suppose (H.\ref{assum:coefficident}) 
   and (H.\ref{assum:nondegeneracy}) hold,
 and $\sE[\xi_0\xi_0^\top]\succ 0$.
Let
 $\theta^0 \in \Theta$,
 and 
for  each 
$\tau>0$, 
let $(\theta^n)_{n\in \sN}\subset  \cB(0,T; \sR^{k\t d}\t \sS^k) $ be defined in \eqref{eq:NPG}.
Then
there exists $ \tau_0,  {C}_1,C_2>0$
such that 
for all $\tau \in (0, \tau_0]$  and $n\in \sN_{  0}$, 
\begin{enumerate}[(1)]
\item 
\label{item:value_conv}
$\cC(\theta^{n+1})\le \cC(\theta^{n})$
and 
$\cC(\theta^{n+1})-\cC(\theta^\star)\le (1-\tau  {C}_1)(\cC(\theta^{n})-\cC(\theta^\star))$,
with $\theta^\star$   defined in \eqref{eq:KV_star},

\item
$\|K^n-K^\star\|^2_{L^2}+\|V^n-V^\star\|^2_{L^2}\le  {C}_2(1-\tau  {C}_1)^n$.
\label{item:control_conv}
\end{enumerate}

 \end{Theorem}
 {The precise expressions of  the constants $\tau_0$, $C_1$ and $C_2$ can be found    in the proof of the statement. 
  These constants  depend    on the regularisation weight $\rho$ 
  in    \eqref{eq:reg_cost_open_loop},
  the constant $\widetilde{\delta}$ in  (H.\ref{assum:nondegeneracy}),
   the     initial guess  $\theta^0$,
   and the a-priori bounds
   $\ul{\lambda}_X, \ol{\lambda}_X, \ul{\lambda}_V, \ol{\lambda}_V$
     in Proposition \ref{prop:uniform_bound}. 
     Achieving more precise dependencies in terms of model parameters is challenging. 
     It would entail deriving
     precise a-priori bounds of solutions to \eqref{eq:riccati_star_ent}
and \eqref{eq:lyapunov_reg} in terms of the coefficients given in 
(H.\ref{assum:coefficident}\ref{assum:integrability}).
This remains an open problem, particularly when the diffusion coefficient is controlled ($D\neq0$) and when cost coefficients $Q$, $R$, and $G$ are not positive definite.

 } 
\subsection{Mesh-independent linear convergence with discrete-time policies}
\label{sec:mesh-independence}

 By leveraging Theorem \ref{thm:linear_conv},
 this section proposes PG methods that take actions at discrete time points and achieve  a robust convergence
 behaviour across different mesh sizes. 
%
Our analysis shows that   a proper scaling of the discrete-time gradients   in terms of mesh size is critical for a robust performance of the algorithm. 
 
 More precisely, 
let $ \mathscr{P}_{[0,T]}$ be 
the collection of all partitions of $[0,T]$.
For each $\pi=\{0=t_0<\cdots<t_N=T\} \in  \mathscr{P}_{[0,T]}$,
let  $|\pi|=\max_{i=0,\ldots, N-1}(t_{i+1}-t_i)$
be the   mesh size of $\pi$, and
let $\Theta^\pi\subset \Theta$   be the set of 
piecewise constant policies 
on $\pi$:
\begin{align}\label{eq:Theta_pi}
\begin{split}
\Theta^\pi
&=\left\{\theta \in \Theta 
\mid 
\theta_t= \theta_{t_i},
\;
\textnormal{a.e.~$t\in [t_i,t_{i+1})$ and all $i\in \{0,\ldots, N-1\}$}\right\}.
\end{split}
\end{align}
Then
define 
the minimum cost  $\cC $ over $\Theta^\pi$:
\bb\l{eq:reg_cost_discrete}
\cC_{\pi}^\star=\inf_{\theta\in \Theta^{\pi}} \cC(\theta).
\ee
Note  that
by $\Theta^\pi\subset \Theta$,
 $\cC_{\pi}^\star\ge \inf_{\theta\in \Theta}  \cC(\theta)
=\cC(\theta^\star)>-\infty$.

We now introduce 
a family of gradient descent schemes 
 for \eqref{eq:reg_cost_discrete}. 
 Let 
 $\theta^{\pi,0}\in \Theta^\pi$
be an initial guess
and 
 $\tau>0$ 
be a stepsize.
Consider the following sequence $(\theta^{\pi,n})_{n\in \sN_{ 0}}\subset \Theta^\pi$  (cf.~\eqref{eq:NPG}) such that 
for all $n\in \sN_0$, 
\bb\label{eq:NPG_discrete}
K^{\pi,n+1}_t=K^{\pi,n}_t-\tau \cD^\pi_K(\theta^{\pi,n})_t, 
\quad 
V^{\pi,n+1}_t= V^{\pi,n}_t-\tau \cD^\pi_V(\theta^{\pi,n})_t, 
\q \textnormal{a.e.~$t\in [0,T]$},
\ee
where 
 $(\cD^\pi_K,\cD^\pi_V): \Theta^\pi\to \Theta^\pi$
 approximates
the gradient operators $(\cD_K,\cD^{\rm bw}_V)$  
in \eqref{eq:NPG} as $|\pi|\to 0$;
 see (H.\ref{assum:D_pi}) for 
the precise     condition.

The convergence  behaviour    of \eqref{eq:NPG_discrete} is measured by 
the   number of required iterations 
 $N^\pi(\eps)$
to achieve a certain accuracy $\eps>0$: 
let  $(\theta^{\pi,n})_{n\in\sN_{ 0}}$ be generated by \eqref{eq:NPG_discrete}
(with some $\theta^{\pi,0}\in \Theta^\pi$ and
$\tau>0$),
and for each 
  $\eps>0$,
define
\bb\label{eq:N_pi_eps}
N^{\pi}(\eps)\coloneqq \min
\left\{n\in \sN_{  0} 
\,\Big\vert\,
 \cC(\theta^{\pi,n})-\inf_{\theta\in \Theta^{\pi}}\cC(\theta)< \eps
\right\}
\in \sN_0\cup\{\infty\}.
\ee 
Note that $N^{\pi}$ is defined for  a fixed mesh $\pi$, 
and hence the residual is defined using the minimum cost $\cC_{\pi}^\star$ over piecewise constant policies $\Theta^\pi$. 
Similarly, let  $(\theta^{n})_{n\in\sN_{ 0}}$ be a sequence   generated by \eqref{eq:NPG}
(with some $\theta^{0}\in \Theta$ and
$\tau>0$), and for each $\eps>0$, define  
\bb\label{eq:N_eps}
N(\eps)\coloneqq \min
\left\{n\in \sN_{  0} 
\,\Big\vert\, 
\cC(\theta^{n})-\inf_{\theta\in \Theta}\cC(\theta)< \eps
\right\}\in \sN_0\cup\{\infty\}.
\ee

The main result of this section shows that 
if the   gradient operators
$(\cD^\pi_K,\cD^\pi_V)_{\pi}$
in \eqref{eq:NPG_discrete} satisfy
the   consistency condition (H.\ref{assum:D_pi}), then 
for all sufficiently fine grids $\pi$,
$N^{\pi}(\eps)$ is essentially equal to $N(\eps)$.

%
%
%

 \begin{Assumption}
 \label{assum:D_pi}
For every $\theta \in L^2(0,T; \sR^{k\t d})\t C([0,T]; \sS^k_+) $, 
  every   
sequence $(\pi_m)_{m\in \sN}\subset \mathscr{P}_{[0,T]}$
such that $\lim_{m\to \infty} |\pi_m|=0$,
and   every $(\theta^{m})_{m\in \sN}\subset \Theta$  
such that 
$\theta^{m} \in \Theta^{\pi_m}$ for all $m\in \sN$
and 
  $\lim_{m\to \infty}\|\theta^{m}-\theta\|_{L^2\t L^\infty}=0$, 
 we have
$$
\lim_{m\to \infty}\|\cD^{\pi_m}_K(\theta^{m})-\cD_K(\theta)\|_{L^2}=0,
\quad 
\textnormal{and}\q
\lim_{m\to \infty}\|\cD^{\pi_m}_V(\theta^{m})-\cD^{\rm bw}_V(\theta)\|_{L^\infty}=0.
$$

  \end{Assumption}

 \begin{Theorem}
\label{thm:mesh_independent_general}
Suppose (H.\ref{assum:coefficident}),
  (H.\ref{assum:nondegeneracy})
and (H.\ref{assum:D_pi})
 hold,
   $\sE[\xi_0\xi_0^\top]\succ 0$,
 $D\in C([0,T];\sR^{d\t k})$,
 $R\in C([0,T]; \sS^k)$ 
 and 
  $\bar{V} \in C([0,T]; \sS^k_+)$.
 Let  
 $\theta^0   \in 
 L^2(0,T; \sR^{k\t d})\t C([0,T]; \sS^k_+)
 $,
  let $(\pi_m)_{m\in \sN}\subset \mathscr{P}_{[0,T]}$    be such that $\lim_{m\to \infty}|\pi_m|=0$  
  and 
let
$(\theta^{\pi_m,0})_{m\in \sN}\subset \Theta$
be such that 
$\theta^{\pi_m,0}\in \Theta^{\pi_m}$ for all $m\in \sN$
and 
  $\lim_{m\to \infty}\|\theta^{\pi_m,0}-\theta^0\|_{L^2\t L^\infty}=0$. 
Then
there exists $\tau_0>0$
 such that for all $\tau\in (0,\tau_0)$ and 
 $\eps>0$, there exists $\ol{m}\in \sN$ such that 
 \bb\label{eq:N_pi_bound_1}
N(\eps)-1 \le N^{\pi_m}(\eps)\le N(\eps),
\quad \fa m\in \sN\cap[ \ol{m},\infty).
\ee
\end{Theorem}

The proof of Theorem \ref{thm:mesh_independent_general} is given in 
Section \ref{sec:proof_mesh_independence}.

{Theorem \ref{thm:mesh_independent_general} indicates that
\eqref{eq:NPG_discrete} achieves  linear convergence  uniformly across different timescales.
Indeed,   by Theorem    \ref{thm:linear_conv},
there exists $ \tau_0,  {C}_1 >0$
such that 
for all $\tau \in (0, \tau_0]$  and $n\in \sN_{  0}$, 
$\cC(\theta^{n+1})-\cC(\theta^\star)\le (1-\tau  {C}_1)^n(\cC(\theta^{0})-\cC(\theta^\star))$.
This  implies that 
$N(\eps)\le \frac{\ln\left(\frac{\eps}{\cC(\theta^{0})-\cC(\theta^\star)}\right)}{\ln(1-\tau C_1)}$
for all $\eps>0$.
By the identity that $\lim_{x\to 0}\frac{\ln(1+x)}{x}=1$,  
$N^{\pi_m}(\eps)\approx \frac{1}{C_1\tau} \ln\left(\frac{\cC(\theta^{0})-\cC(\theta^\star)}{\eps}\right)  $
for all $m\ge \ol{m}$ and   sufficiently small $\tau$ and $\eps$.

}

To design a concrete  gradient methods satisfying  (H.\ref{assum:D_pi}),
     for each $\pi=\{0=t_0<\cdots<t_N=T\} \in \mathscr{P}_{[0,T]}$,
we  identify
$ \Theta^\pi$ with $(\sR^{k\t d} \t \sS_+^k)^N$
by the natural parameterisation:
\bb\label{eq:finite_parameterisation}
(\sR^{k\t d} \t \sS_+^k)^N
\ni (K_i,V_i)_{i=1}^{N-1}
\mapsto   
\left(
 \sum^{N-1}_{i=0}  {K}_i \mathds{1}_{[t_i,t_{i+1})}(t),
 \,
  \sum^{N-1}_{i=0}  {V}_i \mathds{1}_{[t_i,t_{i+1})}(t)
 \right)_{t\in [0,T]}\in \Theta^\pi,
\ee
and 
by  abuse of notation, 
write $\cC:(\sR^{k\t d} \t \sS_+^k)^N\to \sR$
as  the cost 
of a Gaussian policy induced by the parameterisation
\eqref{eq:Gaussian_V} and 
\eqref{eq:finite_parameterisation}.
Then for each  
 $\theta^{\pi,0}\in \Theta^\pi$ and
 $\tau>0$,
  consider the following sequence 
  $(\theta^{\pi,n})_{n\in \sN_{ 0}}\subset  \Theta^\pi$ such that 
for all $n\in \sN_0$
and $i\in \{0,\ldots, N-1\}$,
$\theta^{\pi,n+1}_t= (K^{\pi,n+1}_{i},V^{\pi,n+1}_{i})$
for all $t\in [t_i,t_{i+1})$, with
\begin{align}\label{eq:NPG_discrete_finite}
\begin{split}
K^{\pi,n+1}_i &=K^{\pi,n}_i-
\frac{\tau}{\Delta_i}
{\nabla_{K_i}\cC}(\theta^{\pi,n})
\left(\Sigma_{t_i}^{\theta^{\pi,n}}
\right)^{-1}, 
\\ 
V^{\pi,n+1}_i  &= V^{\pi,n}_i
-\frac{\tau}{\Delta_i}
\left(
  V^{\pi,n}_i{\nabla_{V_i}\cC}(\theta^{\pi,n})
  +{\nabla_{V_i}\cC}(\theta^{\pi,n})  V^{\pi,n}_i
\right), 
\end{split}
\end{align}
where
$\Delta_i= t_{i+1}-t_i$, 
 $\Sigma ^{\theta^{\pi,n}}_{t_i}=\sE[X^{\theta^{\pi,n}}_{t_i}(X^{\theta^{\pi,n}}_{t_i})^\top]$,
and ${\nabla_{K_i}\cC} $ (resp.~${\nabla_{V_i}\cC}$) is the partial derivative of   $\cC$ with respect to the matrix $K_i$
(resp.~${V_i}$). 
The practical implementation of the algorithm is further discussed at the end of this section.

%
%

The following corollary shows that 
 \eqref{eq:NPG_discrete_finite}
satisfies    (H.\ref{assum:D_pi}),
whose  proof  is given in 
Section \ref{sec:proof_mesh_independence}.

\begin{Corollary}
\label{corollary:mesh_independent_piecewise}
Suppose (H.\ref{assum:coefficident})
and    (H.\ref{assum:nondegeneracy}) hold,
    $\sE[\xi_0\xi_0^\top]\succ 0$,
 $D\in C([0,T];\sR^{d\t k})$,
 $R\in C([0,T]; \sS^k)$ 
 and 
  $\bar{V} \in C([0,T]; \sS^k_+)$.
Then Theorem \ref{thm:mesh_independent_general}
holds for   \eqref{eq:NPG_discrete_finite}.
\end{Corollary}

\begin{Remark}
[\textbf{Scaling hyper-parameters with timescales}]
\label{rmk:hyperparameter_scaling}
It is critical to scale the stepsize $\tau$ in 
  \eqref{eq:NPG_discrete_finite}
  with respect to $\Delta_i$  for the  robustness of \eqref{eq:NPG_discrete_finite} for all small mesh sizes.
Indeed,   standard discrete-time natural PG methods
correspond to 
setting $\Delta_i=1$ in \eqref{eq:NPG_discrete_finite}
for all grids.
For   sufficiently fine grids,
this    is equivalent to adopting a vanishing stepsize 
$\tau \Delta_i$ in \eqref{eq:NPG},
as 
$ {\nabla_{K_i}\cC}(\theta)
\left(\Sigma_{t_i}^{\theta}
\right)^{-1} \approx \cD_K(\theta)_{t_i}\Delta_i $
and 
${\nabla_{V_i}\cC}(\theta^{\pi,n})\approx 
\cD_V(\theta)_{t_i}\Delta_i$
 (see  Proposition \ref{prop:Gateaux}).
This explains the   degraded performance
of conventional  discrete-time PG methods
 for small    mesh sizes. 
In contrast, by normalising the 
stepsize
with $\Delta_i$, 
 \eqref{eq:NPG_discrete_finite}
  admits a continuous-time limit \eqref{eq:NPG} as 
the time stepsize 
  $|\pi|\ $ vanishes., and achieves mesh-independent convergence;
  see Section \ref{sec:numerical} for more details.

 \end{Remark}
 
 {\begin{Remark}[\textbf{Extensions to discrete-time models}]
 Corollary \ref{corollary:mesh_independent_piecewise} can be extended to incorporate   time discretization of the underlying system. Here we provide a heuristic explanation of such an extension. Consider a sequence of time grids $(\pi_m)_{m\in \mathbb{N}} $ with $\lim_{m\to \infty} |\pi_m| = 0$. For each $m\in \mathbb{N}$, let $X^{m}$ be the discrete-time state dynamics resulting from the Euler--Maruyama discretization of \eqref{eq:reg_state_closed_loop} on the grid $\pi_m$, and let $\cC^{\pi_m}:\Theta^{\pi_m}\to \mathbb{R}$ be the associated cost functional  \eqref{eq:reg_cost_closed_loop}. 
 Introduce an analogue of \eqref{eq:NPG_discrete_finite}, where $ \nabla_{K_i}\cC(\theta^{\pi,n})$ and $ \nabla_{V_i}\cC (\theta^{\pi,n})$ are replaced by $ \nabla_{K_i}\cC^{\pi_m}(\theta^{\pi,n})$ and $ \nabla_{V_i}\cC^{\pi_m} (\theta^{\pi,n})$, respectively, and $\Sigma_{t_i}^{\theta^{\pi,n}}$ is replaced by the covariance of the discrete-time state $X^{m}$ controlled by $\theta^{\pi_m,n}$.
If   the coefficients in    (H.\ref{assum:coefficident}\ref{assum:integrability})
  are  sufficiently regular in time, 
 one can show that these discrete-time gradients   converge   to the continuous-time gradients in \eqref{eq:NPG} as $m\to \infty$,
 due to   the weak convergence of the Euler--Maruyama scheme.
 This would verify   Condition (H.\ref{assum:D_pi}), which along with   Theorem \ref{thm:mesh_independent_general} implies that   these discrete-time algorithms achieve mesh-independent linear convergence uniformly in $m$.

Similar analyses  can be carried out for   various time discretizations of the state system.  
Making these arguments precise for general time discretizations would require accurately quantifying the regularity conditions of the coefficients for the weak convergence of the discretization, and is left   for future work.
 
  \end{Remark}
}

{We end this section by describing a possible practical implementation of the algorithm \eqref{eq:NPG_discrete_finite}
   which allows for unknown coefficients in \eqref{eq:reg_state_closed_loop} and   \eqref{eq:reg_cost_closed_loop}.
Recall that, as shown in \cite{szpruch2022optimal}, for a given Gaussian policy $\nu^\theta$, the aggregated dynamics \eqref{eq:reg_state_closed_loop} and the associated cost \eqref{eq:reg_cost_closed_loop} can be approximated by interacting with the linear dynamics \eqref{eq:state_strict} with random actions.
More precisely,
  let $\tilde{\pi}=\{0=\tilde{t}_0<\cdots<\tilde{t}_M=T\} $ 
be a time mesh at which random actions are sampled. 
Consider $X^{\theta,\zeta}$ governed by  the following   dynamics:
\bb
\label{eq:state_randomise}
\d X_t =(A_tX_t+B_t  \phi^{\theta}_t(X_t) ) \, \d t+(C_t X_t +D_t\phi^{\theta}_t(X_t) ) \, \d W_{t}, \q t\in [0,T];
\q X_0=\xi_0,
\ee
where 
$$\phi^{\theta}_t(x)= K_t x+ V^{1/2}_t\vartheta_t, 
\quad \textnormal{with $\vartheta_t\coloneqq \sum_{i=0}^{M-1} \zeta_i \mathds{1}_{[\tilde{t}_i,\tilde{t}_{i+1})}(t)$,}
$$
and $(\zeta_i)_{i=0}^{M-1}$ are mutually independent standard normal vectors that are   independent of $\xi_0$ and $W$.
The associated cost   with fixed realisations of $\vartheta$, $\xi_0$ and $W$ is defined as:
 \begin{align}\label{eq:reg_cost_randomised}
    \begin{split}
   %
\hat{\cC}(\theta)\coloneqq 
 & 
\int_0^T \left( 
\frac{1}{2}
\left \la \begin{pmatrix}
Q_t &S^\top_t
\\
S_t & R_t
\end{pmatrix}
  \begin{pmatrix}
X^{\theta,\zeta}_t
\\
\phi_t( X^{\theta,\zeta}_t)
\end{pmatrix},
  \begin{pmatrix}
X^{\theta,\zeta}_t
\\
\phi_t( X^{\theta,\zeta}_t)
\end{pmatrix}
\right\ra    + \rho\cH(\nu^\theta_t(X^{\theta,\zeta}_t)\| \ol{\mathfrak{m}}_t) \right)\d t
 +\frac{1}{2}(X^{\theta,\zeta}_T)^\top G X^{\theta,\zeta}_T,
    \end{split}
\end{align}
where by $\ol{\mathfrak{m}}_t=\cN(0,\bar{V}_t)$ (see Lemma \ref{lemma:representation_KV}),
$$
\cH(\nu^\theta_t(X^{\theta,\zeta}_t)\| \ol{\mathfrak{m}}_t) 
= \frac{1}{2}\left(\tr\Big(K^\top_t \bar{V}_t^{{-1}} K_t  X^{\theta,\zeta}_t (X^{\theta,\zeta}_t)^\top+ \bar{V}_t^{{-1}}V_t\Big)-k+\ln\left(\frac{\det(\bar{V}_t)}{\det(V_t)} \right)\right).
$$
In \eqref{eq:state_randomise}, 
the linear dynamics
\eqref{eq:state_strict} is controlled by sampling      actions 
 from   $\nu^{\theta^\pi}$   using the injected    noises $(\zeta_i)_{i=0}^{M-1}$,
and \eqref{eq:reg_cost_randomised} is the   quadratic cost induced by   these random actions. 
Then, 
by arguments similar to those in  \cite[Theorem 2.2]{szpruch2022optimal},  
$|\sE[X^{\theta,\zeta}_t (X^{\theta,\zeta}_t)^\top]-\Sigma^{\theta}_t|\le C |\tilde{\pi}|  $ 
for all $t\in [0,T]$,
and $|\sE[ \hat{\cC}(\theta)]-{\cC}(\theta)|\le C |\tilde{\pi}| $,
with a constant $C$ independent of $\tilde{\pi}$. 
One can also establish an error bound of the order    $\cO(\sqrt{|\tilde{\pi}|}) $   in the
high-probability sense with respect to the noise process $\vartheta$. 

The above observation suggests that,
at each iteration of \eqref{eq:NPG_discrete_finite}, 
 the  gradients   $ \nabla_{K_i}\cC(\theta^{\pi,n})$, $ \nabla_{V_i}\cC (\theta^{\pi,n})$
      and the state covariance $\Sigma_{t_i}^{\theta^{\pi,n}}$ at all grid points of $\pi$
      can be estimated 
using Monte Carlo methods without relying on knowledge of the  coefficients in \eqref{eq:reg_state_closed_loop} and   \eqref{eq:reg_cost_closed_loop}.   
By choosing a sufficiently fine randomisation grid $\tilde{\pi}$, 
the  covariance $\Sigma_{t_i}^{\theta^{\pi,n}}$ can be estimated by  
the empirical covariance of   $X^{\theta^{\pi,n},\zeta}$ 
corresponding to   different realisations of $\vartheta$, $W$ and $\xi_0$. 
 The gradients of the cost $\cC  (\theta^{\pi,n})$ can be approximated by 
 suitable zero-order optimisation methods 
 based on    trajectories of the cost \eqref{eq:reg_cost_randomised} 
(see e.g.,  \cite{fazel2018global, hambly2020policy, berahas2022theoretical}). 
It   would be interesting to quantify the precise sample efficiency of 
such a model-free implementation of \eqref{eq:NPG_discrete_finite}.
This would entail  
estimating the  approximation errors of $ \nabla_{K_i}\cC(\theta^{\pi,n})$, $ \nabla_{V_i}\cC (\theta^{\pi,n})$ 
and $\Sigma_{t_i}^{\theta^{\pi,n}}$
in terms of the sample frequency $|\tilde{\pi}|^{-1}$ 
and the sample size,
and  quantifying the  precise error propagation  through the gradient descent iteration.
We leave a rigorous analysis of
such a model-free algorithm  for future research.  
}      

\section{Proofs}
\subsection{Analysis of optimisation landscape}
\label{sec:proof_cost_lanscape}

This section proves  the regularity of the cost functional $\cC$ in \eqref{eq:reg_cost_closed_loop} given in Section \ref{sec:landscape}.

We start by proving several technical lemmas. 
The following lemma expresses the coefficients of \eqref{eq:reg_state_closed_loop}
and 
the cost function \eqref{eq:reg_cost_closed_loop} in terms of $\theta=(K,V)$. 
The proof follows from a straightforward computation and hence is omitted. 
\begin{Lemma}
\label{lemma:representation_KV}
Suppose (H.\ref{assum:coefficident}) holds.
Then for all $\nu^\theta\in \cV$ and $(t,x)\in [0,T]\t \sR^d$,
\begin{align*}
	\Phi_t(x, \nu^\theta_t(x)) 
	&= (A_t+B_tK_t) x,
\\ 
	\Gamma_t(x, \nu^\theta_t(x) )&=\left( (C_t+D_tK_t) xx^\top (C_t +D_tK_t)^\top+D_tV_tD^\top_t \right)^{\frac{1}{2}},
	\\
\int_{\sR^k} 
\left \la \begin{pmatrix}
Q_t &S^\top_t
\\
S_t & R_t
\end{pmatrix}
  \begin{pmatrix}
x
\\
a 
\end{pmatrix},
  \begin{pmatrix}
x
\\
a 
\end{pmatrix}
\right \ra \,\nu^\theta_t(x;\d a)
&=
\left \la \begin{pmatrix}
Q_t &S^\top_t
\\
S_t & R_t
\end{pmatrix}
  \begin{pmatrix}
x
\\
K_t x
\end{pmatrix},
  \begin{pmatrix}
x
\\
K_t x
\end{pmatrix}
\right \ra
+\tr (R_tV_t),
\\
\cH(\nu^\theta_t(x)\| \ol{\mathfrak{m}}_t)
&=  \frac{1}{2}\left((K_tx)^\top\bar{V}_t^{{-1}} K_tx+\tr(\bar{V}_t^{{-1}}V_t)-k+\ln\left(\frac{\det(\bar{V}_t)}{\det(V_t)} \right)\right).
\end{align*}
\end{Lemma}

The next lemma represents 
  the cost $\cC(\theta)$ in terms of   $P^\theta$ defined in \eqref{eq:lyapunov_reg}. 
%

%

\begin{Lemma}
\label{lemma:value_PDE}
Suppose (H.\ref{assum:coefficident}) holds.
For each $\theta\in \Theta$,
let    $P^\theta\in C([0,T];\sS^d)$ satisfy     \eqref{eq:lyapunov_reg},
let  $\varphi^\theta\in C([0,T];\sR^d)$ 
satisfy   
for a.e.~$t \in [0,T]$,
\begin{align*}
		(\tfrac{\d }{\d t}{\varphi})_t + &  \tfrac{1}{2}\tr\left(  (D_t^\top P^\theta_tD_t+R_t+\rho \bar{V}_t^{-1} )V_t\right)  +\tfrac{\rho}{2}\left( -k+\ln\left(\tfrac{\det(\bar{V}_t)}{\det(V_t)} \right)\right) =0;
		\quad 
		\varphi_T  = 0,
\end{align*}
and let
 $u^\theta:[0,T]\t \sR^d \to \sR$ be such that 
$u^\theta_t(x) = \frac{1}{2}x^\top P^\theta_t x+\varphi^\theta_t$
for all $(t,x)\in [0,T]\t \sR^d $.
Then 
for all $\theta \in \Theta$ and $x\in \sR^d$,
\begin{align}\label{eq:PDE_reg}
	\begin{split}
	(\tfrac{\d }{\d t}u^\theta)_t
	&+	\frac{1}{2}\tr\left(\Gamma_t(x, \nu_t^\theta(x) )\Gamma_t(x, \nu_t^\theta(x) )^\top (\nabla^2_xu^\theta)_t(x)\right)+ \Phi_t(x, \nu_t^\theta(x) )^\top (\nabla_x u^\theta)_t (x)
		\\
&   +\frac{1}{2}\left(x^\top Q_tx
	+x^\top S^\top_t  K_t x
	+	(K_tx)^\top  S_t x
	+ (K_tx)^\top R_t  K_tx 
	+  \tr(R_t V_t)	\right)
	\\
	& 	+ \frac{\rho}{2}\left((K_tx)^\top\bar{V}_t^{{-1}} K_tx+\tr(\bar{V}_t^{{-1}}V_t)-k+\ln\left(\frac{\det(\bar{V}_t)}{\det(V_t)} \right)\right)=0,
	\q \textnormal{a.e.~$t\in [0,T]$,}
	\end{split}
\end{align}
and $u^{\theta}_T(x)=\frac{1}{2}x^\top Gx$,
where $\nabla_x u^\theta$ and 
$\nabla^2_x u^\theta$ are the gradient and Hessian of $u^\theta$ in $x$, respectively. 
Moreover, it holds that 
$\cC(\theta)=\sE[u_0^\theta(\xi_0)]$.
\end{Lemma}

\begin{proof}
Let $X^\theta\in \cS^2(0,T;\sR^d)$ be the solution to \eqref{eq:reg_state_closed_loop}.
For notational simplicity, 
we omit  $\theta$ in the superscripts of all variables.

 By Lemma \ref{lemma:representation_KV} and the definition of $u$,
 for all $(t,x)\in [0,T]$,
 \begin{align*}
 \Phi_t(x, \nu_t^\theta(x) )^\top (\nabla_x u^\theta)_t (x)
& = \tfrac{1}{2}x^\top\left((A_t+B_tK_t)^\top P_t+P_t(A_t+B_tK_t)\right) x,
 \\
 \tr\left(\Gamma_t(x, \nu_t^\theta(x) )\Gamma_t(x, \nu_t^\theta(x) )^\top (\nabla^2_xu^\theta)_t(x)\right)
 &=
 \tr\left( \left((C_t+D_tK_t) xx^\top (C_t +D_tK_t)^\top+DV_tD^\top\right)
P_t\right).
\end{align*}
Then one can easily see from the definitions of $P$ and $\varphi$
that $u$ satisfies \eqref{eq:PDE_reg} 
for a.e.~$t\in [0,T]$ and all $x\in \sR^d $,
and $u_T(x)=\frac{1}{2}x^\top Gx$.

Now 
applying It\^{o}'s formula to 
$t\mapsto  u_t(X_t)$ implies that 
\begin{align}
\label{eq:ito_u}
\begin{split}
u_T(X_T)
 &=
u_0(X_0)+
 \int_{0}^{T}
\bigg((\tfrac{\d }{\d t}u)_t(X_t)
 +	\frac{1}{2}\tr\left(\Gamma_t(X_t, \nu_t(X_t) )\Gamma_t(X_t, \nu_t(X_t) )^\top (\nabla^2_xu)_t(X_t)\right)
\\
&\quad 
+ \Phi_t(X_t, \nu_t(X_t) )^\top (\nabla_x u)_t(X_t) \bigg)\, \d t 
 +\int_0^T (\nabla_x u)_t(X_t)^\top \Gamma_t(X_t, \nu_t(X_t) )\,\d W_t.
 \end{split}
\end{align}
By the identity  $\nabla _x u_t = P_t x$
and the integrability of $C,D, \theta$ and $X$,
$\int_0^{\cdot} (\nabla_x u)_t(X_t)^\top \Gamma_t(X_t, \nu_t(X_t) )\,\d W_t$ is a martingale.
Hence taking expectations on both sides of \eqref{eq:ito_u}
and using \eqref{eq:PDE_reg} give that
 \begin{align}
 \label{eq:C_theta_X}
 \begin{split}
\sE[u_0(\xi_0)]
 &=
\sE\left[\frac{1}{2}X^\top_TGX_T\right]+\sE\bigg[\int_{0}^{T}
\bigg\{
\frac{1}{2}\left(
\left \la \begin{pmatrix}
Q_t &S^\top_t
\\
S_t & R_t
\end{pmatrix}
  \begin{pmatrix}
X_t
\\
K_t X_t
\end{pmatrix},
  \begin{pmatrix}
X_t
\\
K_t X_t
\end{pmatrix}
\right \ra
+  \tr(R_t V_t)	\right)
\\
&\quad 
+ 
\frac{\rho}{2}\left((K_tX_t )^\top\bar{V}_t^{{-1}} K_tX_t +\tr(\bar{V}_t^{{-1}}V_t)-k+\ln\left(\frac{\det(\bar{V}_t)}{\det(V_t)} \right)\right)
 \bigg\}\, \d t\bigg],
 \end{split}
\end{align}
which along with Lemma \ref{lemma:representation_KV} leads to 
the 
 desired identity $\cC(\theta)= \sE[u_0(\xi_0)]$.
 \end{proof}


The following lemma  quantifies the difference of   value functions for two   policies.

\begin{Lemma}
\label{lemma:performance_gap}
Suppose (H.\ref{assum:coefficident}) holds.
For each $\theta\in \Theta$,
let
   $P^\theta\in C([0,T];\sS^d)$ satisfy     \eqref{eq:lyapunov_reg},
and  let 
   $\Sigma^\theta \in C([0,T];\ol{\sS^d_+})$ satisfy \eqref{eq:lq_sde_K_reg_cov}.
Then for all $\theta,\theta'\in \Theta$,
 \begin{align*}
 \cC(\theta')-\cC(\theta)
 &=
\int_0^T 
\bigg(
\la  K'_t-K_t, \cD_K(\theta)_t\Sigma^{\theta'}_t\ra 
+
\frac{1}{2}\la K'_t-K_t , (D^\top_t  P^\theta_t  D_t   +R_t+\rho \bar{V}_t^{-1}) (K'_t-K_t)\Sigma^{\theta'}_t\ra
\\
&\quad +\ell_t(V'_t,P^\theta_t)-\ell_t(V_t,P^\theta_t)
\bigg)\, \d t,
\end{align*}
where
$\cD_K(\theta)_t$   is defined by  \eqref{eq:D_K},
and 
 $\ell:[0,T]\t \sS^k_+\t \sR^{d\t d}\to \sR$ is given by
\bb\label{eq:entropy_V}
\ell_t(V,Z)=\frac{1}{2} \left(  \la  D^\top_t Z D_t + R_t +\rho \bar{V}_t^{{-1} }, V\ra -\rho\ln(\det(V))\right) 
\q 
\fa (t,V,Z)\in [0,T]\t \sS^k_+\t \sR^{d\t d}.
\ee

\end{Lemma}

\begin{proof}
Throughout this proof, let $\theta,\theta'\in \Theta$ be given, 
let $(P,\Sigma)=(P^\theta,\Sigma^\theta)$,
 $(P',\Sigma')=(P^{\theta'},\Sigma^{\theta'})$,
 $u=u^\theta$ and 
  $u'=u^{\theta'}$,
  where for each $\theta\in \Theta$,
  $u^\theta:[0,T]\t \sR^d\to \sR$ is defined as in  Lemma \ref{lemma:value_PDE}. 
  By \eqref{eq:PDE_reg},
for all   $x\in \sR^d$,
$(u'-u)_T(x)=0$, 
  \begin{align}
  \label{eq:pde_u'-u}
	\begin{split}
	(\tfrac{\d }{\d t} (u'-u))_t
	&+	\frac{1}{2}\tr\left(\Gamma_t(x, \nu_t^{\theta'}(x) )\Gamma_t(x, \nu_t^{\theta'}(x) )^\top (\nabla^2_x (u'-u))_t(x)\right)
		\\
&   + \Phi_t(x, \nu_t^{\theta'}(x) )^\top (\nabla_x (u'-u))_t (x) +F_t(x)=0,
\quad \textnormal{a.e.~$t\in [0,T]$,}
 	\end{split}
\end{align}
where $F:[0,T]\t \sR^d\to \sR$ is given by
  \begin{align*}
	\begin{split}
	F_t(x)
&  	=
		\frac{1}{2}\tr\left(\Gamma_t(x, \nu_t^{\theta'}(x) )\Gamma_t(x, \nu_t^{\theta'}(x) )^\top (\nabla^2_x u)_t(x)\right)+ \Phi_t(x, \nu_t^{\theta'}(x) )^\top (\nabla_x u)_t (x)
		\\
&\quad {- \frac{1}{2}\tr\left(\Gamma_t(x, \nu_t^{\theta}(x) )\Gamma_t(x, \nu_t^{\theta}(x) )^\top (\nabla^2_x u)_t(x)\right)-\Phi_t(x, \nu_t^{\theta}(x) )^\top (\nabla_x u)_t (x)}
\\
& \quad 
+\frac{1}{2}\bigg[  \left(x^\top Q_tx
	+x^\top S^\top_t  K'_t x
	+	(K'_tx)^\top  S_t x
	+(K'_tx)^\top R_t  K'_tx +  \tr(R_t V'_t)	\right)
\\
&\q 	
	-\left(x^\top Q_tx
	+x^\top S^\top_t  K_t x
	+	(K_tx)^\top  S_t x
	+(K_tx)^\top R_t  K_tx +  \tr(R_t V_t)	\right)
	\bigg]
	\\
	& \quad 
	+ \frac{\rho}{2}
	\bigg[\left((K'_tx)^\top\bar{V}_t^{{-1}} K'_tx+\tr(\bar{V}_t^{{-1}}V'_t)
	-\ln\left(\det(V'_t)\right)\right)
	\\
	&\quad 
	-\left((K_tx)^\top\bar{V}_t^{{-1}} K_tx+\tr(\bar{V}_t^{{-1}}V_t)-\ln\left(\det(V_t)\right)\right)
	\bigg].
 	\end{split}
\end{align*}
 Applying It\^{o}'s formula to 
$t\mapsto (u'-u)_t(X_t^{\theta'})$ 
(recall the definition of $u^\theta$ in 
Lemma \ref{lemma:value_PDE})
and  using \eqref{eq:pde_u'-u}
yield  that 
$$
\sE[(u'-u)_T(X_T^{\theta'})]-\sE[(u'-u)_0(X_0^{\theta'})]=
\sE\left[\int_0^T 
-F_t(X_t^{\theta'})\, \d t
\right],
$$
which along with 
$\cC(\theta)=\sE[u^\theta_0(\xi_0)]$
(see  Lemma  \ref{lemma:value_PDE})
and $(u'-u)_T=0$ implies that 
\bb
\label{eq:difference_C_F}
\cC(\theta')-\cC(\theta)=
\sE\left[\int_0^T 
F_t(X_t^{\theta'})\, \d t
\right].
\ee
We now simplify the expression of  $F_t(x)$ for any given $(t,x )\in [0,T]\t \sR^d$. 
To this end, let $\ol{H}:[0,T]\t \sR^d\t \sR^k\t \sR^d\t \sR^{d\t d}\to \sR$ 
be a modified Hamiltonian such that $(t,x,a,y,z)\in [0,T]\t \sR^d\t \sR^k\t \sR^d\t \sR^{d\t d}$, 
\begin{align*}
\ol{H}_t(x,a,y,z)
&=
\tfrac{1}{2}\tr\left((C_tx+D_t a)(C_tx+D_ta)^\top z\right)+\la  A_tx+B_t a,  y  \ra 
\\
& \q +\tfrac{1}{2}  \left(x^\top Q_tx
+x^\top S^\top_t  a
	+	a^\top  S_t x
+	a^\top (R_t+\rho \bar{V}_t^{-1})  a 	\right),
\end{align*}
and let $\ell:[0,T]\t \sS^k_+\t \sR^{d\t d}\to \sR$ be defined as in 
\eqref{eq:entropy_V}.
Recall that  $(\nabla_x u)_t(x) = P_t x$ and $(\nabla^2_x u)_t(x) = P_t$.
Hence by   Lemma \ref{lemma:representation_KV},
 \begin{align}
 \label{eq:F_t_H_difference}
	\begin{split}
	F_t(x)
	&=\ol{H}_t(x,K'_tx,P_t x ,P_t)-\ol{H}_t(x,K_tx,P_t x ,P_t)
	 +\ell_t(V'_t,P_t)-\ell_t(V_t,P_t).
 	\end{split}
\end{align}
Observe that
for all $(t,x,y,z)\in [0,T]\t   \sR^k\t \sR^d\t \sS^{d}$,
 $a\mapsto \ol{H}_t(x,a,y,z)$ is a quadratic function, and hence Taylor's expansion shows that for all $a,a\in \sR^k$,
 \begin{align*}
\ol{H}_t(x,a',y,z)-\ol{H}_t(x,a,y,z)
&=
\la a'-a, \p_a \ol{H}_t(x,a,y,z)  \ra 
+\frac{1}{2} 
\la a'-a, \p^2_a \ol{H}_t(x,a,y,z)(a'-a)  \ra,
\end{align*}
where $\p_a \ol{H}_t(x,a,y,z) $ and $\p^2_a \ol{H}_t(x,a,y,z)$ are given by
 \begin{align*}
\p_a\ol{H}_t(x,a,y,z)
&= D^\top_t  z (C_tx+D_ta)+ B_t^\top y +S_t x+(R_t+\rho \bar{V}_t^{-1})  a,
\\
\p^2_a\ol{H}_t(x,a,y,z)
&=D^\top_t  z  D_t   +R_t+\rho \bar{V}_t^{-1}.
\end{align*}
Substituting the above identities into  
\eqref{eq:F_t_H_difference} yields
\begin{align*}
	\begin{split}
	F_t(x)
	&=\ell_t(V'_t,P_t)-\ell_t(V_t,P_t)
	\\
& \q 	 +	\la (K'_t-K_t)x, D^\top_t  P_t (C_tx+D_t K_tx)+ B_t^\top P_t x
+S_t x
 +(R_t+\rho \bar{V}_t^{-1})  K_tx\ra 
\\
&\quad 
+\frac{1}{2}\la (K'_t-K_t)x, (D^\top_t  P_t  D_t   +R_t+\rho \bar{V}_t^{-1}) (K'_t-K_t)x\ra,
	\end{split}
\end{align*}
which along with \eqref{eq:difference_C_F},
the definition of $\cD_K(\theta)$ in \eqref{eq:D_K},
and $\Sigma'_t=\sE[X^{\theta'}_t(X^{\theta'}_t)^\top]$ 
leads to the desired conclusion.
\end{proof}

\begin{proof}[Proof of Proposition \ref{prop:Gateaux}]
For each $\theta\in \Theta$, by \eqref{eq:C_theta_X},
\begin{align}
	\begin{split}\label{eq:cost_K_reg_cov}
		\cC(\theta) =& \frac{1}{2}
		\int_0^T\bigg( \tr\left(
		(Q_t 
		+ K_t^\top S_t +S_t^\top K_t
		+K_t^\top ( R_t + {\rho}\bar{V}_t^{{-1}})K_t)\Sigma^\theta_t\right) 
		\\
		&+\tr(R_t V_t )
		+\rho\left(
		\tr(\bar{V}_t^{{-1}}V_t)-k+\ln\left(\frac{\det(\bar{V}_t)}{\det(V_t)} 			\right)
		\right)
		\bigg)\d t + \frac{1}{2}\tr (G\Sigma^\theta_T),
	\end{split}
\end{align} 
where 
 $\Sigma^\theta\in C([0,T]; \ol{\sS^d_{+}})$ satisfies 
\eqref{eq:lq_sde_K_reg_cov}.
We then  apply  
\cite[Corollary 4.11]{carmona2016lectures}
to   characterise the
Gateaux derivatives.
Let 
$H:[0,T]\t \ol{\sS^d_{+}}\t 
\sR^{k\t d}\t  \sS^k_{+}\t 
\sR^d\to \sR$ be the Hamiltonian 
of \eqref{eq:cost_K_reg_cov}-\eqref{eq:lq_sde_K_reg_cov}
such that for all $(t,\Sigma, K,V, Y)\in [0,T]\t \ol{\sS^d_{+}}\t 
\sR^{k\t d}\t  \sS^k_{+}\t 
\sR^{d\t d}$,
\begin{align*}
H_t(\Sigma, K,V, Y)
&  =
\la 
(A_t+B_tK)\Sigma+ \Sigma(A_t+B_tK)^\top +  (C_t+D_tK)\Sigma(C_t+D_tK)^\top +  D_t V D_t^\top, 
Y
 \ra
 \\
&   \q +
 \frac{1}{2}
		\bigg\{ \tr\left(
		(Q_t 
		+ K^\top S_t +S_t^\top K
		+K^\top ( R_t + {\rho}\bar{V}_t^{{-1}})K)\Sigma\right) 
		+\tr(R_t V )
\\
&   \q 		+\rho\left(
		\tr(\bar{V}_t^{{-1}}V)-k+\ln\left(\frac{\det(\bar{V}_t)}{\det(V)} 			\right)
		\right)
		\bigg\},
\end{align*}
and
for each $\theta\in \Theta$,
 let 
$Y^\theta\in C([0,T]; \sR^{d\t d})$ be the adjoint process satisfying 
\begin{align*}
(\tfrac{\d }{\d t}Y)_t = -\p_\Sigma H_t(\Sigma^\theta_t, K_t,V_t, Y_t),
\quad \textnormal{a.e.~$t\in [0,T]$};
\quad Y_T= \tfrac{1}{2}G.
\end{align*}
Then by \cite[Corollary 4.11]{carmona2016lectures}, for all $\theta,\theta\in \Theta$,
\begin{align*}
\frac{\d}{\d \eps} \cC(K+\eps K', V)\Big|_{\eps=0}
&=\int_0^T  \la 
\p_K H_t(\Sigma^\theta_t, K_t,V_t, Y^\theta_t),
K'_t
\ra\,\d t,
\\
\frac{\d}{\d \eps} \cC(K, V+\eps( V'-V))\Big|_{\eps=0}
&=\int_0^T  \la 
\p_V H_t(\Sigma^\theta_t, K_t,V_t, Y^\theta_t),
V'_t-V_t
\ra\,\d t.
\end{align*}
Observe that $Y^\theta=\frac{1}{2}P^\theta\in C([0,T];\sS^d)$, 
and 
for all $(t,\Sigma, K,V, Y)\in [0,T]\t \ol{\sS^d_{+}}\t 
\sR^{k\t d}\t  \sS^k_{+}\t 
\sS^{d}$,
\begin{align*}
\p_K H_t(\Sigma , K ,V, Y)
&=
\left(
2B_t^\top Y 
+2D_t^\top Y(C_t+D_tK) 
+ S_t+
 ( R_t + {\rho}\bar{V}_t^{{-1}})K
 \right)\Sigma,
\\
\p_V H_t(\Sigma , K ,V, Y)
&=
D_t^\top YD_t
+\tfrac{1}{2}(R_t
+{\rho}(\bar{V}_t^{{-1}}-{V}^{{-1}})).
\end{align*}
 This proves the desired claims.
\end{proof}

\begin{proof}[Proof of Proposition \ref{prop:Lojaiewicz}]
Observe from a  direct computation that  for all
$Z,\Gamma\in \sR^{k\t d}$,
$\Sigma\in \ol{\sS^k_+}$ and
  $M\in \sS^k_+$,
\begin{align}
\label{eq:quadratic_lower_bound}
\begin{split}
\la  Z, \Gamma  \Sigma \ra 
+
\frac{1}{2}\la Z , M Z \Sigma\ra
&=
\frac{1}{2}\left\la  Z +M^{-1} \Gamma, 
M\left(Z+M^{-1} \Gamma \right)\Sigma \right\ra 
-
\frac{1}{2}\la
M^{-1} \Gamma, \Gamma\Sigma \ra
\\
&\ge 
-
\frac{1}{2}\la
M^{-1} \Gamma, \Gamma\Sigma \ra,
\end{split}
\end{align}
where the last inequality uses the fact that 
$\tr (AB)\ge 0$ if $A,B\in \ol{\sS^d_+}$. 
Hence for all $\theta\in \Theta$ and $t\in [0,T]$,
    substituting \eqref{eq:quadratic_lower_bound}   with
$Z=K^\star_t-K_t$,
 $\Gamma=\cD_K(\theta)_t$,
 $\Sigma=\Sigma^{\theta^\star}_t$  and 
$M=D^\top_t  P^\theta_t  D_t   +R_t+\rho \bar{V}_t^{-1}$
yields that
\begin{align}
\label{eq:gd_K}
\begin{split}
&\int_0^T 
\left(
\la  K^\star_t-K_t, \cD_K(\theta)_t\Sigma^{\theta^\star}_t\ra 
+
\frac{1}{2}\la K^\star_t-K_t , (D^\top_t  P^\theta_t  D_t   +R_t+\rho \bar{V}_t^{-1}) (K^\star_t-K_t)\Sigma^{\theta^\star}_t\ra
\right)\, \d t
\\
&\quad 
\ge 
-
\frac{1}{2} \int_0^T 
\la
(D^\top_t  P^\theta_t  D_t   +R_t+\rho \bar{V}_t^{-1})^{-1} \cD_K(\theta)_t, \cD_K(\theta)_t\Sigma^{\theta^\star}_t\ra
\, \d t.
\end{split}
\end{align}
Then by  Lemma  \ref{lemma:performance_gap}
and \eqref{eq:gd_K}:
\begin{align}
\label{eq:PL_K}
\begin{split}
 &\cC(\theta^\star)- \cC ( {\theta} )
 \\
  &\quad =
\int_0^T 
\bigg(
\la  K^\star_t-K_t, \cD_K( {\theta})_t\Sigma^{\theta^\star}_t\ra 
+
\frac{1}{2}\la K^\star_t-K_t , (D^\top_t  P^{{\theta}}_t  D_t   +R_t+\rho \bar{V}_t^{-1}) (K^\star_t-K_t)\Sigma^{\theta^\star}_t\ra
\\
&\quad \q +\ell_t(V^\star_t,P^\theta_t)-\ell_t(V_t,P^\theta_t)
\bigg)\, \d t
\\
&\quad 
\ge 
 \int_0^T 
\left( -
\frac{1}{2}
\la
(D^\top_t  P^\theta_t  D_t   +R_t+\rho \bar{V}_t^{-1})^{-1} \cD_K(\theta)_t, \cD_K(\theta)_t\Sigma^{\theta^\star}_t\ra
+\ell_t(V^\star_t,P^\theta_t)-\ell_t(V_t,P^\theta_t)
\right)
\, \d t.
\end{split}
\end{align}

Now
by \eqref{eq:entropy_V},
for all 
$(t,Z)\in [0,T]\t  \sR^{d\t d}$
and $V,V'\in \sS^k_+$,
\begin{align*}
&\ell_t(V',Z)-\ell_t(V,Z)
\\
&
\q =\la \p_V\ell_t(V,Z), V'-V\ra 
+
\int_0^1 \left( \tfrac{\d }{\d s}\ell_t(V+s(V'-V),Z)-\la \p_V\ell_t(V,Z), V'-V\ra\right)\,\d s 
\\
&
\q =\la \p_V\ell_t(V,Z), V'-V\ra 
+
\int_0^1  
\la
\p_V\ell_t(V+s(V'-V),Z) -  \p_V\ell_t(V,Z), V'-V\ra \,\d s. 
\end{align*}
Recall  that 
$\p_V\ell_t(V,Z)= \frac{1}{2}(D^\top_t Z D_t + R_t +\rho \bar{V}_t^{{-1} }-\rho V^{-1})$,
and $A^{-1}-B^{-1}=B^{-1}(B-A)A^{-1}$ for all $A,B\in \sS^k_+$.
Then for all 
$(t,Z)\in [0,T]\t  \sR^{d\t d}$
and $V,V'\in \sS^k_+$,
\begin{align}
\label{eq:ell_expansion}
\begin{split}
&\ell_t(V',Z)-\ell_t(V,Z)
\\
& \q =\la \p_V\ell_t(V,Z), V'-V\ra 
+
\frac{\rho}{2}
\int_0^1  
\la 
V^{-1}\big( s(V'-V)\big)(V+s(V'-V))^{-1}  , V'-V\ra \,\d s.
\end{split}
\end{align}
Hence for all $\theta,\theta'\in \Theta$ and $t\in [0,T]$,
by using 
   \eqref{eq:D_V},
 the fact that 
$\tr (AB)\ge 0$ for all $A,B\in \ol{\sS^d_+}$,
and    \eqref{eq:quadratic_lower_bound}
(with $Z=V'_t-V_t$,
   $\Gamma=\cD_V(\theta)_t$,
   $\Sigma=I_k$,
   $M=\frac{\rho}{2} {\Lambda}(V'_t,V_t)^2I_k$),
\begin{align}
\label{eq:PL_V}
\begin{split}
&  \ell_t(V'_t,P^\theta_t)-\ell_t(V_t,P^\theta_t)
\\
& \q =\la \p_V\ell_t(V_t,P^\theta_t ), V'_t-V_t\ra 
+
\frac{\rho}{2}
\int_0^1  
\la 
V^{-1}_t\big( s(V'_t-V_t)\big)(V_t+s(V'_t-V_t))^{-1}  , V'_t-V_t \ra \,\d s
\\
& \q \ge 
\la 
\cD_V(\theta)_t, V'_t-V_t\ra 
+
\frac{\rho}{4} {\Lambda}(V'_t,V_t)^2
 \la 
  V'_t-V_t    , V'_t-V_t \ra
  \ge -\frac{1}{\rho {\Lambda}(V'_t,V_t)^2}|\cD_V(\theta)_t|^2,
  \end{split}
\end{align}
with $ {\Lambda}(V'_t,V_t)>0$ defined as  
\begin{align*}
{\Lambda}(V'_t,V_t)
 &\coloneqq \min_{s\in [0,1]}
 \lambda_{\min}
 \left(
 (V_t+s(V'_t-V_t))^{-1}
 \right)
 =\frac{1}{
 \max_{s\in [0,1]}
  \lambda_{\max}
 \left(
V_t+s(V'_t-V_t)
 \right)}
 \\
& =
 \frac{1}{
 \max(\lambda_{\max}(V_t),    \lambda_{\max}(V'_t) )
 }
 =
 \frac{1}{ \max(\|V_t\|_2,  \|  V'_t\|_2 ) },
\end{align*}
due to the convexity of $[0,1]\ni s\mapsto  \lambda_{\max}
 \left(
V_t+s(V'_t-V_t)
 \right)\in \sR$,
and $\|V\|_2=\lambda_{\max}(V)$ for all $V\in \ol{\sS^k_+}$.
Substituting 
\eqref{eq:PL_V}
with $V'=V^\star$ and 
using  \eqref{eq:PL_K}
yield the desired estimate \eqref{eq:gd_KV_statement}.
%
\end{proof}

\begin{proof}[Proof of Proposition \ref{prop:Lipshcitz_smooth}]
By    \eqref{eq:D_V} and \eqref{eq:ell_expansion},
for all $\theta,\theta'\in \Theta$ and $t\in [0,T]$,
\begin{align*}
&  \ell_t(V'_t,P^\theta_t)-\ell_t(V_t,P^\theta_t)
\\
& \q =\la \p_V\ell_t(V_t,P^\theta_t ), V'_t-V_t\ra 
+
\frac{\rho}{2}
\int_0^1  
\la 
V^{-1}_t\big( s(V'_t-V_t)\big)(V_t+s(V'_t-V_t))^{-1}  , V'_t-V_t \ra \,\d s
\\
& \q \le 
\la 
\cD_V(\theta)_t, V'_t-V_t\ra 
+
\frac{\rho}{4} \ol{\Lambda}(V'_t,V_t)^2
 \la 
  V'_t-V_t    , V'_t-V_t \ra,
\end{align*}
where $ \ol{\Lambda}(V'_t,V_t)>0$ is given by
\begin{align*}
\ol{\Lambda}(V'_t,V_t)
 &\coloneqq \max_{s\in [0,1]}
 \lambda_{\max}
 \left(
 (V_t+s(V'_t-V_t))^{-1}
 \right)
 =\frac{1}{
 \min_{s\in [0,1]}
  \lambda_{\min}
 \left(
V_t+s(V'_t-V_t)
 \right)}
 \\
& =
 \frac{1}{
 \min(\lambda_{\min}(V_t),    \lambda_{\min}(V'_t) )
 }.
\end{align*}
Combining this and Lemma  \ref{lemma:performance_gap}
yields the desired estimate. 
\end{proof}


\subsection{Proof of  Proposition  \ref{prop:uniform_bound}}
 \label{sec:proof_a_priori_bound}
 
 
 The following lemma   compares solutions to 
\eqref{eq:lyapunov_reg}
for different 
 $\theta,\theta'\in \Theta$.

 \begin{Lemma}
 \label{lemma:laypunov_difference}
 Suppose (H.\ref{assum:coefficident}) holds.
For each $\theta\in \Theta$,
let $P^\theta\in C([0,T];\sS^d)$ satisfy \eqref{eq:lyapunov_reg}.
Then for all $\theta,\theta'\in \Theta$,
 $\Delta P\coloneqq P^{\theta'}-P^\theta$ satisfies 
for a.e.~$t\in [0,T]$,
\begin{align*}
 \begin{split}
(\tfrac{\d }{\d t}\Delta P)_t  
&+
 (A_t+B_tK'_t)^\top \Delta P_t + \Delta P_t^\top (A_t+B_tK'_t) 
+(C_t+D_tK'_t)^\top \Delta P_t (C_t+D_tK'_t) 
\\
&+
(K'_t-K_t)^\top \cD_K(\theta)_t   
+
\cD_K(\theta)_t^\top (K'_t-K_t)
\\
& 
+(K'_t-K_t)^\top
(D^\top_t P^\theta_t D_t+R_t+\rho\bar{V}_t^{{-1}})
 (K'_t-K_t),
=0;\quad \Delta P_T=0,
\end{split}
\end{align*}
where $\cD_K(\theta)_t$ is defined in \eqref{eq:D_K}.
 \end{Lemma}
\begin{proof}
By  \eqref{eq:lyapunov_reg},
$\Delta P_T=0$ and 
 for a.e.~$t\in [0,T]$, 
\begin{align*}
 \begin{split}
(\tfrac{\d }{\d t}\Delta P)_t  
&+
 (A_t+B_tK'_t)^\top \Delta P_t + \Delta P_t^\top (A_t+B_tK'_t) 
+(C_t+D_tK'_t)^\top \Delta P_t (C_t+D_tK'_t) 
\\
&  +q_t(K'_t)-q_t(K_t)=0,
\end{split}
\end{align*}
where for all $K\in \sR^{k\t d}$,
\begin{align*}
q_t(K)&\coloneqq 
 (A_t+B_tK)^\top  P^\theta_t + ( P^\theta_t)^\top (A_t+B_tK) 
+(C_t+D_tK)^\top  P^\theta_t (C_t+D_tK) 
\\
&\q + S_t^\top K+K^\top S_t
+K^\top (R_t+\rho\bar{V}_t^{{-1}})K.
\end{align*}
Observe that for any $K_1,K_2\in \sR^{k\t d}$
and $P\in \sS^d$,
$$
K_1^\top P K_1-K_2^\top P K_2
=(K_1-K_2)^\top P K_2+K_2^\top P(K_1-K_2)
+(K_1-K_2)^\top P (K_1-K_2).
$$
Thus for a.e.~$t\in [0,T]$,
\begin{align*}
q_t(K'_t)-q_t(K_t)
&=
(K'_t-K_t)^\top
\left(B_t^\top P^\theta_t +D_t^\top P^\theta_t (C_t+D_tK_t)
+S_t+ (R_t+\rho\bar{V}_t^{{-1}})K_t
\right)
\\
&\q 
+
\left(B_t^\top P^\theta_t +D_t^\top P^\theta_t (C_t+D_tK_t)
+S_t+(R_t+\rho\bar{V}_t^{{-1}})K_t
\right)^\top (K'_t-K_t)
\\
&\q 
+(K'_t-K_t)^\top
(D^\top_t P^\theta_t D_t+R_t+\rho\bar{V}_t^{{-1}})
 (K'_t-K_t),
\end{align*}
which along with the definition of $\cD_K(\theta)_t$ leads to the desired identity.
\end{proof}

  Based on Lemma \ref{lemma:laypunov_difference}, we establish a uniform bound of  $(P^{\theta^n})_{n\in \sN}$ and $(K^n)_{n\in \sN}$. 

\begin{Proposition}
\label{prop:K_bdd}
 Suppose (H.\ref{assum:coefficident}) 
 and (H.\ref{assum:nondegeneracy})
hold.
For each  $\theta\in \Theta$,
let   $P^{\theta}\in C([0,T];\sS^d)$ satisfy  \eqref{eq:lyapunov_reg}.
Let $\theta^0\in \Theta$, 
  $\ol{\lambda}_{ 0}>0$ be
 such that $\ol{\lambda}_{0} I_k \succeq D^\top P^{\theta^0} D+R+\rho \bar{V}^{-1}$,
 and
for each 
$\tau>0$,  
let $(K^n)_{n\in \sN}\subset   \cB(0,T; \sR^{d\t k} )$ be defined in \eqref{eq:NPG}.
 Then
\begin{enumerate}[(1)]
\item
\label{item:P_monotone}
 for all 
$\tau \in (0,2/\ol{\lambda}_0]$ 
 and 
$n\in \sN_{ 0}$,  $P^{\theta^{n}}\succeq P^{\theta^{n+1}}\succeq P^\star$,
and 
  $
 \widetilde{\delta}I_k
 \preceq D^\top P^{\theta^n}D+ R + {\rho}\bar{V}^{{-1}}
 \preceq \ol{\lambda}_0 I_k
$, with $P^\star\in C([0,T];\sS^d)$ and $ \widetilde{\delta}>0$ in 
(H.\ref{assum:nondegeneracy});
\item
\label{item:K_bound}
there exists  $\widetilde{C}_{(\theta^0)}\ge 0$ such that 
for all  
$\tau \in (0,{1}/\ol{\lambda}_0]$
 and $n\in \sN_{  0}$, 
 $\|K^n\|_{L^2}\le \widetilde{C}_{(\theta^0)} $. 
\end{enumerate}
\end{Proposition}

\begin{proof}
We write $P^n=P^{\theta^n}$ for notational simplicity. 
For each $n\in \sN$, applying 
 \eqref{eq:NPG} and 
Lemma \ref{lemma:laypunov_difference}
with $\theta'=\theta^n$ and $\theta=\theta^{n-1}$,
$\Delta P\coloneqq P^{n}-P^{n-1}\in C([0,T];\sS^d)$ satisfies
$\Delta P_T=0$, and 
for a.e.~$t\in [0,T]$,
\begin{align*}
 \begin{split}
&(\tfrac{\d }{\d t}\Delta P)_t  
 +
 (A_t+B_tK^{n+1}_t)^\top \Delta P_t + \Delta P_t^\top (A_t+B_tK^{n+1}_t) 
+(C_t+D_tK^{n+1}_t)^\top \Delta P_t (C_t+D_tK^{n+1}_t) 
\\
&\quad  
=
-
(K^{n+1}_t-K^n_t)^\top \cD_K(\theta^n)_t   
-
\cD_K(\theta^n)_t^\top (K^{n+1}_t-K^n_t)
\\
& 
\quad \quad
 -(K^{n+1}_t-K^n_t)^\top
(D^\top_t P^n_t D_t+R_t+\rho\bar{V}_t^{{-1}})
 (K^{n+1}_t-K^n_t) 
 \\
 &\quad  
=
2\tau 
\cD_K(\theta^n)_t^\top 
\left(
I_k- \tfrac{\tau}{2}(D^\top_t P^n_t D_t+R_t+\rho\bar{V}_t^{{-1}})
\right)
\cD_K(\theta^n)_t.
\end{split}
\end{align*}
Now suppose that $\tau \in (0,2/\ol{\lambda}_0]$, then 
$I_k- \tfrac{\tau}{2}(D^\top_t P^0_t D_t+R_t+\rho\bar{V}_t^{{-1}})\succeq 0$,
which 
implies that $P^{1}\preceq P^{0}$
(see e.g., \cite[Lemma 7.3, p.~320]{yong1999stochastic}), and 
hence 
$$
I_k- \tfrac{\tau}{2}(D^\top P^1D+R+\rho\bar{V}^{{-1}})
\succeq 
I_k- \tfrac{\tau}{2}(D^\top P^0D+R+\rho\bar{V}^{{-1}})\succeq 0.$$
An induction argument shows that $P^n\succeq P^{n+1}$ for all $n\in \sN_0$.
Moreover, observe from
\eqref{eq:riccati_star_ent} and  \eqref{eq:KV_star}
that $\cD_K(\theta^\star)=0$ and $P^\star=P^{\theta^\star}$.
By applying Lemma \ref{lemma:laypunov_difference}
with $\theta'=\theta^n$ and $\theta=\theta^\star$,
one can deduce from similar arguments that  $P^n\succeq P^{\theta^\star}$ for all $n\in \sN_0$.
Consequently,  by (H.\ref{assum:nondegeneracy}),
  $$
\ol{\lambda}_0 I_k \succeq D^\top P^0D+ R + {\rho}\bar{V}^{{-1}}
\succeq 
D^\top P^nD+ R + {\rho}\bar{V}^{{-1}}
\succeq D^\top P^\star D+ R + {\rho}\bar{V}^{{-1}}
\succeq \widetilde{\delta}I_k.
$$
This proves Item \ref{item:P_monotone}.

Item \ref{item:P_monotone} implies that 
there exists $\widetilde{C}_{(\theta^0)}>0$ such that
$\|P^n\|_{L^\infty}\le \widetilde{C}_{(\theta^0)}$ for all $n\in \sN_0$. 
Then for all $n\in \sN_0$, 
by  \eqref{eq:D_K} and \eqref{eq:NPG},
\begin{align*}
|K^{n+1}_t|&=
\left|
K^n_t-\tau \left(
(B_t^\top P^n_t+D_t^\top P^n_t C_t+S_t)
+
 (D_t^\top P^n_tD_t+ R_t + {\rho}\bar{V}_t^{{-1}})K^n_t
 \right)
 \right|
 \\
 &\le 
| I_k-
 \tau  (D_t^\top P^n_tD_t+ R_t + {\rho}\bar{V}_t^{{-1}})
 ||K^n_t|
 +\tau |B_t^\top P^n_t+D_t^\top P^n_t C_t+S_t|.
\end{align*}
Thus for all  $\tau \in (0,1/\ol{\lambda}_0]$ and $n\in \sN_0$,
\begin{align*}
\|K^{n+1}\|_{L^2}&\le 
(1-\tau \ol{\lambda}_0)
 \|K^n\|_{L^2}
 +\tau \|B^\top P^n+D^\top P^n C+S\|_{L^2}
 \\
 &\le \|K^0\|_{L^2}+
\sup_{n\in \sN_0}
\frac{1}{\ol{\lambda}_0}\|B^\top P^n+D^\top P^n C+S\|_{L^2} 
<\infty,
\end{align*}
where the last inequality follows from a straightforward induction argument.
\end{proof}

The next proposition proves     a uniform 
upper and lower 
bound of    $(V^n)_{n\in \sN}$. 

\begin{Proposition}
\label{prop:V_bound}
 Suppose (H.\ref{assum:coefficident}) 
and  (H.\ref{assum:nondegeneracy})
hold.
Let $\theta^0\in \Theta$, 
 and
for each 
$\tau>0$,  
let $(\theta^n)_{n\in \sN}\subset   \cB(0,T; \sR^{k\t d}\t \sS^k)$ be defined in \eqref{eq:NPG}.
Let  $\ol{\lambda}_{ 0}>0$ be
 such that $\ol{\lambda}_{0} I_k \succeq D^\top P^{\theta^0} D+R+\rho \bar{V}^{-1}$ 
with $ P^{\theta_0}\in C([0,T];\sS^d)$  
defined in 
\eqref{eq:lyapunov_reg},
let 
$\ul{\lambda}_V= 
 \min\left(
 \min_{t\in [0,T]}
 \lambda_{\min}( V^{0}_t), \frac{\rho}{\ol{\lambda}_0} \right)$,
 and 
 let
$\ol{\lambda}_V
=
\max\left(
 \max_{t\in [0,T]}
 \lambda_{\max}( V^{0}_t), \frac{\rho}{ \widetilde{\delta}} \right)
$. 
Then  
 for all 
$\tau \in (0,1/\ol{\lambda}_0]$
and $n\in \sN_0$, 
$\ul{\lambda}_V I_k
 \preceq V^n
 \preceq \ol{\lambda}_V I_k
$.

\end{Proposition}
 
 \begin{proof}
 For each $n\in \sN_0$,
 let 
 $M^n=D^\top   P^{\theta^n}  D   +R +\rho \bar{V}^{-1}$.
 By \eqref{eq:D_V}, for each $n\in \sN_0$ and a.e.~$t\in [0,T]$, 
 \begin{align*}
 V^{n+1}_t
 &=V^n_t- {\tau}\left(
\frac{1}{2} \left(M^n_t-\rho(V_t^n)^{-1}\right)
  V_t^n+V^n_t \frac{1}{2} \left(M^n_t-\rho(V_t^n)^{-1}\right)\right)
  \\
  &
  =\frac{1}{2}\left(I_k- {\tau} M^n_t\right)V^n_t
  +\frac{1}{2} V^n_t\left( I_k- {\tau}   M^n_t\right)
  + {\rho\tau} I_k.
 \end{align*}
Let  $\tau \in (0,1/\ol{\lambda}_0]$. 
By
 Proposition
\ref{prop:K_bdd} Item 
\ref{item:P_monotone}, 
for all $n\in \sN_0$, 
$
\widetilde{\delta}I_k
 \preceq M^n
 \preceq \ol{\lambda}_0 I_k
$, 
and hence
$
0\preceq (1- {\tau \ol{\lambda}_0 }  )I_k
\preceq
I_k- {\tau}   M^n
 \preceq (1-  {\tau \widetilde{\delta}}  )I_k
$. Thus for all $n\in \sN_0$ and a.e.~$t\in [0,T]$, 
 \begin{align*}
\lambda_{\min} (V^{n+1}_t)
 &  \ge  
 \lambda_{\min} \left(I_k- {\tau}   M^n_t\right)\lambda_{\min} (V^n_t)
  + {\rho\tau}   \ge 
   \left(1-  {\tau \ol{\lambda}_0 } \right)\lambda_{\min} (V^n_t)
  + {\rho\tau}.
 \end{align*}
 Setting $v^n_t=  \lambda_{\min} (V^{n}_t)$ for all $n\in \sN_0$. An induction argument shows that 
 $$
v^n_t\ge 
 \left(1-  {\tau \ol{\lambda}_0 } \right)^n 
v^0_t
 + {\rho\tau} \sum_{i=0}^{n-1}  \left(1-  {\tau \ol{\lambda}_0 } \right)^i
 =\left( v^0_t
-\frac{\rho}{\ol{\lambda}_0}\right)
 \left(1-  {\tau \ol{\lambda}_0 } \right)^n 
+\frac{\rho}{\ol{\lambda}_0}
\ge \min\left(v^0_t, \frac{\rho}{\ol{\lambda}_0} \right).
 $$
Similarly, 
    for all $n\in \sN_0$ and a.e.~$t\in [0,T]$,
    \begin{align*}
  \lambda_{\max}( V^{n+1}_t)
 &\le 
 \lambda_{\max} \left(I_k- {\tau}   M^n_t\right) \lambda_{\max} (V^n_t)
  + {\rho\tau} 
  \le 
  \left(1-  {\tau \widetilde{\delta}}  \right)
  \lambda_{\max} (V^n_t)+ {\rho\tau},
 \end{align*}
 which implies that 
 $  \lambda_{\max}( V^{n}_t)
\le \max\left( \lambda_{\max}( V^{0}_t), \frac{\rho}{ \widetilde{\delta}} \right)$. 
 \end{proof}

The following lemma  establishes  an upper and lower bounds of the state covariance matrices for any $\theta\in \Theta$,
which is crucial for the convergence analysis 
of \eqref{eq:NPG}.
 
\begin{Lemma}
\label{lemma:Sigma_bdd}
 Suppose (H.\ref{assum:coefficident}) 
 and (H.\ref{assum:nondegeneracy})
hold.
For each 
 $\theta\in \Theta$,
  let 
   $\Sigma^\theta \in C([0,T];\ol{\sS^d_+})$ satisfy \eqref{eq:lq_sde_K_reg_cov}.
Then there exists  $\widetilde{C}>0$ such that for all  
$\theta\in \Theta$,
$$
 \lambda_{\min }(\mathbb{E}[\xi_0\xi^\top_0])\exp\left(
-  \widetilde{C}(1+\|K\|_{L^2}^2)
  \right)I_d
\preceq
\Sigma^\theta 
\preceq
\widetilde{C}
 \left(|\Sigma_0|
 +\|V\|_{L^1}
 \right)
 \exp\left( \widetilde{C}(1+\|K\|_{L^2}^2)\right)I_d.
$$

  \end{Lemma}

\begin{proof}
 Let $\theta\in \Theta $ be fixed.
 We omit the superscript of $\Sigma^\theta$ to simplify the notation. 
To estimate $\lambda_{\max}(\Sigma_t)$,
by \eqref{eq:lq_sde_K_reg_cov},
for all $t\in [0,T]$,
\begin{align*}
\|\Sigma_t\|_2\le \|\Sigma_0\|_2+
\int_0^t
\left(
(2 \|\widetilde{A}_s\|_2+
 \|\widetilde{C}_s\|^2_2)\|\Sigma_s\|_2
+    \|D_s\|^2_2 \|V_s\|_2
 \right)\, \d s,
\end{align*}
where 
$\widetilde{A}_t= A_t+B_tK_t$ and 
$\widetilde{C}_t =C_t+D_tK_t$.
Then by 
 (H.\ref{assum:coefficident}\ref{assum:integrability}) 
 and Gronwall's inequality,
 $\|\Sigma\|_{L^\infty}\le  \widetilde{C}
 \left(|\Sigma_0|
 +\|V\|_{L^1}
 \right)
 \exp\left( \widetilde{C}(1+\|K\|_{L^2}^2)\right)$ for some $\tilde{C}_1>0$.
 
 Now we   obtain a lower bound of $\lambda_{\min}(\Sigma_t)$.
As  $ (C+D K )\Sigma  (C +D K )^\top
+D  V D^\top
 \succeq 0$,
by \eqref{eq:lq_sde_K_reg_cov}, 
$\Sigma\succeq \widetilde{\Sigma}$, where 
$\widetilde{\Sigma}\in C([0,T]; \ol{\sS^d_{+}})$ satisfies 
 for a.e.~$t\in [0,T]$,
 \begin{align}
	\begin{split}
			(\tfrac{\d }{\d t}{\Sigma})_t=& (A_t+B_tK_t)\Sigma_t+ \Sigma_t(A_t+B_tK_t)^\top
			; \q 
		\Sigma_0=\mathbb{E}[\xi_0\xi^\top_0].
	\end{split}
\end{align}
Note that 
  for all $t\in [0,T]$,
$
\widetilde{\Sigma}_t =  \Psi_t^\top\mathbb{E}[\xi_0\xi^\top_0] \Psi_t$,
where 
  $\Psi\in C([0,T];\sR^{d\t d})$ 
satisfies $\Psi_0= I_d$ and for a.e.~$t\in [0,T]$, 
 $\d \Psi_t = \Psi_t \widetilde{A}_t^\top\, \d t$,
 with 
  $\widetilde{A} = A+B K\in L^1(0,T;\sR^{d\t d})$.
  For each $t\in [0,T]$, 
let $x_t\in \sR^d$ be such that $|x_t |=1$ and 
 $\lambda_{\min }(\widetilde{\Sigma}_t)=x_t^\top \widetilde{\Sigma}_tx_t$,
 and let $y_t=\Psi_t   x_t$.
 Then
 $$
 \lambda_{\min }( {\Sigma}_t)\ge 
 \lambda_{\min }(\widetilde{\Sigma}_t)
 =x_t^\top\left( (\Psi_t)^\top \mathbb{E}[\xi_0\xi^\top_0] \Psi_t\right)x_t
 =\frac{y_t^\top   \mathbb{E}[\xi_0\xi^\top_0] y_t}{|y_t|^2}|y_t|^2
 \ge  \frac{\lambda_{\min }(\mathbb{E}[\xi_0\xi^\top_0])}{\left\|\Psi_t^{-1}\right\|^{2}_{2}},
 $$
 where the last inequality uses 
 $1= |x_t|\le \|(\Psi_t)^{-1} \|_{2}|y_t|$,
 with the spectral norm $ \| \cdot \|_{2}$.
Observe that  
$\Psi^{-1}\in C([0,T];\sR^{d\t d})$ 
be such that $\Psi_0^{-1}= I_d$ and for a.e.~$t\in [0,T]$, 
 $\d \Psi_t^{-1} = - \widetilde{A}_t^\top \Psi_t^{-1}\, \d t$. 
Hence for all $t\in [0,T]$,
\begin{align*}
\| \Psi_t^{-1} \|_{2}\le 1+\int_0^t
\|\widetilde{A}_s\|_{2}\|\Psi_s^{-1}\|_{2}\, \d s 
\le 1+\int_0^t
|\widetilde{A}_s| \|\Psi_s^{-1}\|_{2}\, \d s, 
\end{align*}
which along with Gronwall's inequality shows that 
$ \| \Psi_t^{-1} \|_{L^\infty}
\le \exp\left( \|\widetilde{A}\|_{L^1} \right)$. 
Consequently, 
$ \inf_{t\in [0,T]}\lambda_{\min }( {\Sigma}_t)
 \ge  \lambda_{\min }(\mathbb{E}[\xi_0\xi^\top_0])\exp\left(-2 \|\widetilde{A}\|_{L^1} \right)$,
 which along with 
  (H.\ref{assum:coefficident}\ref{assum:integrability}) 
leads to the desired lower bound of $\lambda_{\min }( {\Sigma}_t)$.
\end{proof}

A direct consequence 
of Proposition \ref{prop:V_bound} and Lemma \ref{lemma:Sigma_bdd}
are the following  uniform bounds of the state covariance matrices along the iterates $(\theta^n)_{n\in \sN}$ generated by  \eqref{eq:NPG}.

\begin{Proposition}
\label{prop:Sigma_bdd}
 Suppose (H.\ref{assum:coefficident}) 
 and (H.\ref{assum:nondegeneracy})
hold, 
 and $\sE[\xi_0\xi_0^\top]\succ0$.
For each 
 $\theta\in \Theta$,
let   $P^{\theta}\in C([0,T];\sS^d)$ satisfy  \eqref{eq:lyapunov_reg},
and  let 
   $\Sigma^\theta \in C([0,T];\ol{\sS^d_+})$ satisfy \eqref{eq:lq_sde_K_reg_cov}.
Let   $\theta^0\in \Theta$,  
let 
 $\ol{\lambda}_{ 0}>0$ be
such that $\ol{\lambda}_{0} I_k \succeq D^\top P^{\theta^0} D+R+\rho \bar{V}^{-1}$,
and 
for  each 
$\tau\in (0,{1}/\ol{\lambda}_0]$,
let $(\theta^n)_{n\in \sN}\subset   \Theta $ be defined in \eqref{eq:NPG}.
Then
there exists  $\ol{\lambda}_X,\ul{\lambda}_X> 0$,
depending on $\theta_0$,
 such that 
for all $\tau \in (0,{1}/\ol{\lambda}_0]$ and $n\in \sN_{  0}$, 
$\ul{\lambda}_X I_d\preceq \Sigma^{\theta^n}\preceq \ol{\lambda}_X I_d$.
\end{Proposition}

\begin{proof}
By Proposition \ref{prop:K_bdd},
for all  $\tau \in (0,{1}/\ol{\lambda}_0]$,
$\sup_{n\in \sN_0}\|K^n\|_{L^2}\le    \widetilde{C}_{(\theta^0)} $
for some  $\widetilde{C}_{(\theta^0)}>0$.
The uniform lower and upper bounds of 
$(\Sigma^{\theta^n})_{n\in \sN_0}$ 
follow from 
Proposition \ref{prop:V_bound}
and 
Lemma \ref{lemma:Sigma_bdd}.
\end{proof}

\subsection{Proof of Theorem \ref{thm:linear_conv}}
\label{sec:proof_linear_conv}


The following proposition compares 
the value functions of  two consecutive iterates.

\begin{Proposition}
\label{prop:local_smooth}
 Suppose (H.\ref{assum:coefficident}) 
 and (H.\ref{assum:nondegeneracy})
 hold,
 and $\sE[\xi_0\xi_0^\top]\succ 0$.
Let  
 $\theta^0 \in \Theta$,
 and 
  $\ol{\lambda}_{ 0}>0$ be 
 such that $\ol{\lambda}_{0} I_k \succeq D^\top P^{\theta^0} D+R+\rho \bar{V}^{-1}$
 with 
$ P^{\theta_0}\in C([0,T];\sS^d)$ defined in  \eqref{eq:lyapunov_reg}.
For  each 
$\tau \in (0,1/\ol{\lambda}_0]$,
let $(\theta^n)_{n\in \sN}\subset \Theta$ be defined in \eqref{eq:NPG},
 let 
$\ul{\lambda}_V, \ol{\lambda}_V>0$ 
be 
such that 
$\ul{\lambda}_V I_k
 \preceq V^n
 \preceq \ol{\lambda}_V I_k
$ for all $n\in \sN_0$
(cf.~Proposition   \ref{prop:V_bound}),
and
let 
$\ul{\lambda}_X, \ol{\lambda}_X>0$ 
be such that  $\ul{\lambda}_X I_k
 \preceq \Sigma^{\theta^n}
 \preceq \ol{\lambda}_X I_k
$ for all $n\in \sN_0$
(cf.~Proposition   \ref{prop:Sigma_bdd}). 
Then
for all $\tau \in 
(0,  {1}/{\ol{\lambda}_0}]$
and $n\in \sN_0$, 
\begin{align*}
\cC(\theta^{n+1})-\cC(\theta^n)
  \le 
-\tau \int_0^T 
\bigg(
 \left( \ul{\lambda}_{X}-  
\frac{\tau }{2} \ol{\lambda}_{0}  \ol{\lambda}_{X}
\right) |\cD_K(\theta^n)_t|^2 
+
\left(
2\ul{\lambda}_V
-
\frac{\rho\tau  \ol{\lambda}_V^2}{\ul{\lambda}_V^2 
 }  
 \right)|\cD^n_{V}(\theta^n)_t|^2
\bigg)\, \d t.
\end{align*}
\end{Proposition}

\begin{proof}
For each $n\in \sN_{  0}$, 
let $\Sigma^{n}=\Sigma^{\theta^n}$,
 $P^{n}=P^{\theta^n}$,
$\Delta K^{n}=K^{n+1}-K^{n}$,
$\Delta V^{n}=V^{n+1}-V^{n}$,
$\cD^n_{K}=\cD_K(\theta^n)$,
and 
$\cD^n_{V}=\cD_V(\theta^n)$.
By 
using Proposition \ref{prop:Lipshcitz_smooth}
and
the fact that $\ul{\lambda}_V I_k
 \preceq V^n
 \preceq \ol{\lambda}_V I_k$ 
for all $n\in \sN_0$,
  \begin{align*}
 \cC(\theta^{n+1})-\cC(\theta^n)
  & \le 
\int_0^T 
\bigg(
\la  \Delta K^n_t, \cD^n_{K,t}\Sigma^{n+1}_t \ra 
+
\frac{1}{2}\la \Delta K^n_t , (D^\top_t  P^n_t  D_t   +R_t+\rho \bar{V}_t^{-1}) (\Delta K^n_t)\Sigma^{{n+1}}_t\ra
\\
&\quad 
  +\la 
\cD^n_{V, t}, \Delta  V^n_t\ra 
+
\frac{\rho}{4\ul{\lambda}_V^2 
 } | \Delta  V^n_t|^2 
\bigg)\, \d t
\\
 &  \le 
\int_0^T 
\bigg(
\la -\tau\cD^n_{K,t}, \cD^n_{K,t}\Sigma^{n+1}_t \ra 
+
\frac{\tau^2 }{2}\la  \cD^n_{K,t} , (D^\top_t  P^n_t  D_t   +R_t+\rho \bar{V}_t^{-1}) \cD^n_{K,t} \Sigma^{{n+1}}_t\ra
\\
&\quad     -\tau 
\left[
\la 
\cD^n_{V, t}, \{\cD^n_{V, t}  V^n_t\}_S\ra 
-
\frac{\rho\tau}{4\ul{\lambda}_V^2 
 } | \{\cD^n_{V, t}  V^n_t\}_S|^2 
 \right]
\bigg)\, \d t
\end{align*}
with
$\{\cD^n_{V, t}  V^n_t\}_S\coloneqq \cD^n_{V, t} V^n_t+V^n_t \cD^n_{V, t}$, 
where the last inequality used 
  \eqref{eq:NPG}.
Recall    that for all $S_1,S_2\in \ol{ \sS^k_{+}}$,
 $ \lambda_{\min}(S_1)\tr(S_2)\le \tr(S_1S_2)\le
  \lambda_{\max}(S_1)
 \tr(S_2) $.
 Hence 
 $\la 
\cD^n_{V, t}, \{\cD^n_{V, t}  V^n_t\}_S\ra 
\ge 2 \ul{\lambda}_V |\cD^n_{V, t}|^2
$,
and 
$| \{\cD^n_{V, t}  V^n_t\}_S|^2\le
 4\ol{\lambda}_V^2  |  \cD^n_{V, t} |^2
 $. 
 Hence for all $n\in \sN_0$,
  \begin{align*}
  \cC(\theta^{n+1})-\cC(\theta^n)
  &   
 \le 
\int_0^T 
\bigg(
-\tau \left( \lambda_{\min}(\Sigma^{n+1}_t)-
\frac{\tau }{2} \lambda_{\max}( (D^\top_t  P^n_t  D_t   +R_t+\rho \bar{V}_t^{-1}) )\lambda_{\max}(\Sigma^{{n+1}}_t)
\right) |\cD^n_{K,t}|^2 
\\
&\q
-\tau 
\left(
2\ul{\lambda}_V
-
\frac{\rho\tau  \ol{\lambda}_V^2}{\ul{\lambda}_V^2 
 }  
 \right)|\cD^n_{V, t}|^2
\bigg)\, \d t.
\end{align*}
The desired inequality then follows from 
  Propositions \ref{prop:K_bdd} and \ref{prop:Sigma_bdd}.
\end{proof}

The next proposition establishes
a uniform {\L}ojasiewicz property of the cost      $\cC:\Theta \to \sR$
along the iterates \eqref{eq:NPG}.

\begin{Proposition}
\label{prop:local_Lojasiewicz}
 Suppose (H.\ref{assum:coefficident}) 
  and (H.\ref{assum:nondegeneracy})
 hold,
 and $\sE[\xi_0\xi_0^\top]\succ 0$.
 Let $\theta^\star\in \Theta$ be defined in \eqref{eq:KV_star}.
 For each 
 $\theta\in \Theta$,
let   $P^{\theta}\in C([0,T];\sS^d)$ satisfy  \eqref{eq:lyapunov_reg},
and   let 
   $\Sigma^\theta \in C([0,T];\ol{\sS^d_+})$ satisfy \eqref{eq:lq_sde_K_reg_cov}.
Let 
 $\theta^0 \in \Theta$
 and 
  $\ol{\lambda}_{ 0}>0$
 such that $\ol{\lambda}_{0} I_k \succeq D^\top P^{\theta^0} D+R+\rho \bar{V}^{-1}$.
 For  each 
$\tau\in  (0,1/ \ol{\lambda}_0 ]$, 
let $(\theta^n)_{n\in \sN_0}\subset \Theta$ be defined in \eqref{eq:NPG}.
Then
for all $\tau \in  (0,1/ \ol{\lambda}_0 ]$ and $n\in \sN_{  0}$, 
\begin{align*}
\cC(\theta^{n})-\cC(\theta^\star)
  \le 
\max\left(
\frac{\ol{\lambda}^\star_{X}}{2\widetilde{\delta}},
\frac{  \max( \ol{\lambda}_V ,   \ol{\lambda}^\star_V)^2}{\rho}
\right)\int_0^T  
\left( 
|\cD_{K}(\theta^n)_t |^2
+
|\cD_{V}(\theta^n)_t  |^2
\right)
\, \d t,
\end{align*}
where
$ \widetilde{\delta}>0$ is the same as in 
(H.\ref{assum:nondegeneracy}),
$\ol{\lambda}^\star_X>0 $ satisfies 
$   \Sigma^{\theta^\star}
\preceq \ol{\lambda}^\star_{X} I_d $,
$ \ol{\lambda}^\star_V>0$ 
satisfies  
$ V^\star
 \preceq \ol{\lambda}^\star_V I_k$,
and 
$ \ol{\lambda}_V>0$ 
satisfies  
$ V^n
 \preceq \ol{\lambda}_V I_k
$ for all $n\in \sN_0$.

\end{Proposition}

\begin{proof}
Let $  \ol{\lambda}_V^\star>0$ be such that 
$  V^\star \preceq \ol{\lambda}_V^\star I_k $.
For each $n\in \sN_{  0}$, 
let $\Sigma^{n}=\Sigma^{\theta^n}$,
 $P^{n}=P^{\theta^n}$,
$\cD^n_{K}=\cD_K(\theta^n)$,
and 
$\cD^n_{V}=\cD_V(\theta^n)$.
Recall that 
there exists $\ul{\lambda}_V, \ol{\lambda}_V>0$ 
such that 
$\ul{\lambda}_V I_k
 \preceq V^n
 \preceq \ol{\lambda}_V I_k
$ for all $n\in \sN_0$.
Then for all $n\in \sN_0$,
      Proposition  \ref{prop:K_bdd} shows that
$   D^\top_t  P^n_t  D_t   +R_t+\rho \bar{V}_t^{-1}\succeq \widetilde{\delta} I_k$,
which along with Proposition   \ref{prop:Lojaiewicz} shows that
\begin{align*}
\cC(\theta^n)-\cC(\theta^\star)
& \le
 \int_0^T 
 \left(
\frac{1}{2}  \la
(D^\top_t  P^n_t  D_t   +R_t+\rho \bar{V}_t^{-1})^{-1} \cD^n_{K,t},  \cD^n_{K,t}\Sigma^{\theta^\star}_t\ra
+\frac{  \max( \ol{\lambda}_V ,   \ol{\lambda}^\star_V)^2}{\rho}
|\cD^n_{V,t} |^2
\right)
\, \d t
\\
&  \le
\int_0^T  
\left(
\frac{\ol{\lambda}^\star_{X}}{2\widetilde{\delta}}  
|\cD^n_{K,t} |^2
+\frac{ \max( \ol{\lambda}_V ,   \ol{\lambda}^\star_V)^2}{\rho}
|\cD^n_{V,t} |^2
\right)
\, \d t,
\end{align*}
with  
$\ol{\lambda}^\star_{X}>0$ such that 
$\Sigma^{\theta^\star}\preceq \ol{\lambda}^\star_{X} I_d$ (cf.~Lemma \ref{lemma:Sigma_bdd}).
This proves the desired estimate.  
\end{proof}

\begin{proof}[Proof of Theorem \ref{thm:linear_conv}]
Let  $\ol{\lambda}_{ 0}>0$
be  such that $\ol{\lambda}_{0} I_k \succeq D^\top P^{\theta^0} D+R+\rho \bar{V}^{-1}$,
 where 
$ P^{\theta_0}\in C([0,T];\sS^d)$ satisfies \eqref{eq:lyapunov_reg}
with $\theta=\theta_0$.
Then 
by Proposition \ref{prop:local_smooth},
for all 
$\tau \in  (0,1/{\ol{\lambda}_0}]$ and 
$n\in \sN_0$,
\begin{align*}
\cC(\theta^{n+1})-\cC(\theta^n)
  \le 
-\tau \int_0^T 
\bigg(
 \left( \ul{\lambda}_{X}-  
\frac{\tau }{2} \ol{\lambda}_{0}  \ol{\lambda}_{X}
\right) |\cD_K(\theta^n)_t|^2 
+
\left(
2\ul{\lambda}_V
-
\frac{\rho\tau  \ol{\lambda}_V^2}{
\ul{\lambda}_V^2 
 }  
 \right)|\cD^n_{V}(\theta^n)_t|^2
\bigg)\, \d t,
\end{align*}
with the constants $\ul{\lambda}_X, \ol{\lambda}_X>0$  in
  Proposition   \ref{prop:Sigma_bdd}.
Hence by setting $\widetilde{C}_1=\max({\ol{\lambda}_0}, \frac{2\rho \ol{\lambda}_V^2 }{3 \ul{\lambda}_V^3 }, \frac{\ol{\lambda}_0\ol{\lambda}_X}{\ul{\lambda}_X})$, it holds 
for all 
$\tau \in  (0,1/\widetilde{C}_1] $ and 
$n\in \sN_0$,
\begin{align*}
\cC(\theta^{n+1})-\cC(\theta^n)
&  \le 
-\tau \int_0^T 
\bigg(
  \frac{\ul{\lambda}_{X}}{2} |\cD_K(\theta^n)_t|^2 
+ \frac{\ul{\lambda}_{V}}{2} |\cD_V(\theta^n)_t|^2 
\bigg)\, \d t
\\
&\le 
-\tau \frac{1}{2}\min
(\ul{\lambda}_{X},\ul{\lambda}_{V})
\int_0^T 
\bigg( |\cD_K(\theta^n)_t|^2 
+|\cD_V(\theta^n)_t|^2  
\bigg)\, \d t
\\
&\le 
-\tau C_1
\left(\cC(\theta^{n})-\cC(\theta^\star)\right),
\quad
\textnormal{with
$C_1\coloneqq \frac{\min(\ul{\lambda}_{X},\ul{\lambda}_{V})}
{2\max\Big(\frac{\ol{\lambda}^\star_{X}}{2\widetilde{\delta}}  ,
\frac{  \max( \ol{\lambda}_V ,   \ol{\lambda}^\star_V)^2 }{\rho} 
\Big)}
$,}
\end{align*}
where the last inequality used Proposition \ref{prop:local_Lojasiewicz}.
Thus,
for all $\tau \in  (0,\tau_0]$ with $ \tau_0>0$ satisfying
$$
\frac{1}{\tau_0}\ge
 \max\Bigg( {\ol{\lambda}_0}, \frac{2\rho \ol{\lambda}_V^2 }{3 \ul{\lambda}_V^3 }, \frac{\ol{\lambda}_0\ol{\lambda}_X}{\ul{\lambda}_X},
 \frac{\min(\ul{\lambda}_{X},\ul{\lambda}_{V})}
{2\max\Big(\frac{\ol{\lambda}^\star_{X}}{2\widetilde{\delta}}  ,
\frac{  \max( \ol{\lambda}_V ,   \ol{\lambda}^\star_V)^2 }{\rho} 
\Big)}
  \Bigg),
$$
we have 
 for all $n\in \sN_0$,
$\cC(\theta^{n+1})\le \cC(\theta^n)$
and 
 \begin{align} 
 \label{eq:linear_value_proof}
\cC(\theta^{n+1})-\cC(\theta^\star)
\le
\cC(\theta^{n+1})-\cC(\theta^n)+
\cC(\theta^{n})-\cC(\theta^\star)
\le
(1-\tau  {C}_1
)
 \big(
 \cC(\theta^n)-\cC(\theta^\star)
 \big).
\end{align}
 
To prove Item \ref{item:control_conv},
observe that $\cD_K(\theta^\star)=0$ and $\cD_V(\theta^\star)=0$.
Hence by Lemma \ref{lemma:performance_gap} and \eqref{eq:PL_V},
for all $n\in \sN_0$,
 \begin{align*}
 &\cC(\theta^n)-\cC(\theta^\star)
 \\
 &\q \ge
\int_0^T 
\bigg(
\frac{1}{2}\la K^n_t-K^\star_t , (D^\top_t  P^\star_t  D_t   +R_t+\rho \bar{V}_t^{-1}) (K^n_t-K^\star_t)\Sigma^{\theta^n}_t\ra
+\frac{\rho}{4}  
 \frac{|  V^n_t-V^\star_t |^2}{ \max(\|V^n_t\|^2_2,  \|  V^\star_t\|^2_2 ) }
\bigg)\, \d t
 \\
 &\q \ge
\int_0^T 
\bigg(
\frac{1}{2}\widetilde{\delta}\ul{\lambda}_X |K^n_t-K^\star_t |^2
+\frac{\rho}{4\ol{\lambda}_V^2}  
 |  V^n_t-V^\star_t |^2 
\bigg)\, \d t,
\end{align*}
 where the last inequality used (H.\ref{assum:nondegeneracy}),
    Proposition   \ref{prop:Sigma_bdd}
    and $V^\star, V^n\preceq \ol{\lambda}_V I_k$.
This along with Item \ref{item:value_conv} proves Item \ref{item:control_conv}
with $C_2=1/\min\left(\frac{1}{2}\widetilde{\delta}\ul{\lambda}_X, \frac{\rho}{4\ol{\lambda}_V^2} \right)$.
\end{proof}

\subsection{Proofs of Theorem
\ref{thm:mesh_independent_general} and 
Corollary \ref{corollary:mesh_independent_piecewise}}
\label{sec:proof_mesh_independence}

The following lemma proves the optimal costs of piecewise constant policies converges to the optimal cost of continuous-time policies as $|\pi|\to 0$.
\begin{Lemma}
\label{lemma:discrete_value_conv}
Suppose  (H.\ref{assum:coefficident}) and  (H.\ref{assum:nondegeneracy}) hold.
Let $(\pi_m)_{m\in \sN}\subset \mathscr{P}_{[0,T]} $ be such that $\lim_{m\to \infty}|\pi_m|=0$.
Then $\lim_{m\to \infty}C^\star_{\pi_m}=
\inf_{\theta\in \Theta}  \cC(\theta)$.
\end{Lemma}

\begin{proof}
For each $m\in \sN$,
by $\Theta^{\pi_m}\subset \Theta$,
 $\cC_{\pi_m}^\star=  \inf_{\theta\in \Theta^{\pi_m}}  \cC(\theta)\ge \inf_{\theta\in \Theta}  \cC(\theta)
$, which implies that   
$\lim\inf_{m\to \infty} \cC_{\pi_m}^\star
  \ge \inf_{\theta\in \Theta}  \cC(\theta)$.
 On the other hand, 
 let $\theta^\star=(K^\star,V^\star)$ be defined in \eqref{eq:KV_star},
 and 
 for each $m\in \sN$, let $\theta^{m,\star}=(K^{m,\star}, V^{m,\star})$ be  the $L^2$ projection of $\theta^\star$ onto $\Theta^m$ such that 
 $K^{m,\star}_t= \sum^{N_m-1}_{i=0} \ol{K}^\star_{t_i} \mathds{1}_{[t_i,t_{i+1})}(t)$
 and 
  $V^{m,\star}_t= \sum^{N_m-1}_{i=0} \ol{V}^\star_{t_i} \mathds{1}_{[t_i,t_{i+1})}(t)$
  for a.e.~$t\in [0,T]$, 
  where 
  $$
  \ol{K}^\star_{t_i} =\frac{1}{t_{i+1}-t_i}\int_{t_i}^{t_{i+1}}
  {K}^\star_t\, \d t,
  \quad 
    \ol{V}^\star_{t_i} =\frac{1}{t_{i+1}-t_i}\int_{t_i}^{t_{i+1}}
  {V}^\star_t\, \d t,
  \q \fa i\in \{0,\ldots, N_m-1\}.
  $$
A standard mollification argument shows that 
  $\lim_{m\to \infty}\|\theta^{m,\star}-\theta^\star\|_{L^2}=0$.
  Moreover, the fact that 
$\eps I_k \preceq V^\star\preceq  \frac{1}{\eps} I_k$ for  some $\eps>0$
implies that 
$\eps I_k \preceq V^{m,\star}\preceq  \frac{1}{\eps} I_k$ for all $m\in \sN$.
By  the uniform $L^2$-bound of $(K^{m,\star})_{m\in\sN}$ and 
the $L^\infty$-bound of $(V^{m,\star})_{m\in\sN}$,
there exists $C\ge 0$ such that
$\Sigma^{\theta^{ m,\star}}\preceq C I_d$ for all $m\in \sN$ due to Lemma \ref{lemma:Sigma_bdd}.
Then by Proposition \ref{prop:Lipshcitz_smooth},
for all $m\in \sN$,
 \begin{align*}
 \cC(\theta^{ m,\star})-\cC(\theta^\star)
 &\le 
\int_0^T 
\bigg(
  \frac{1}{2}\la K^{ m,\star}_t-K^\star_t , (D^\top_t  P^{\theta^\star}_t  D_t   +R_t+\rho \bar{V}_t^{-1}) (K^{ m,\star}_t-K^\star_t)\Sigma^{\theta^{ m,\star}}_t\ra
\\
&\quad 
+
\frac{\rho}{4} 
\frac{ | V^{ m,\star}_t-V^\star_t  |^2}{
 \min(\lambda^2_{\min}(V^\star_t),    \lambda^2_{\min}(V^{ m,\star}_t) )
 }
\bigg)\, \d t,
\end{align*}
which along with $\lim_{m\to \infty}\|\theta^{m,\star}-\theta^\star\|_{L^2}=0$ and 
$ V^{m,\star}\succeq \eps I_k$, 
$\Sigma^{\theta^{ m,\star}}\preceq C I_d$ for all $m\in \sN$  implies that 
$\lim_{m\to \infty}   \cC(\theta^{ m,\star})
=
 \inf_{\theta\in \Theta}  \cC(\theta)$.
 As  $\cC_{\pi_m}^\star\le  \cC(\theta^{ m,\star}) $ for all $m\in \sN$,
 $$\inf_{\theta\in \Theta}  \cC(\theta)
 \le\lim\inf_{m\to \infty} \cC_{\pi_m}^\star
 \le 
\lim\sup_{m\to \infty} \cC_{\pi_m}^\star
  \le  \lim\sup_{m\to \infty}   \cC(\theta^{ m,\star}) 
  =\inf_{\theta\in \Theta}  \cC(\theta).
  $$ 
This leads to the desired convergence result.\end{proof}

The following proposition proves that 
when the mesh size $|\pi|$ are sufficiently small,
the policies from \eqref{eq:NPG_discrete}
have similar costs as those from
\eqref{eq:NPG}.
 
 \begin{Proposition}  
\label{prop:discrete_convergence}
Suppose (H.\ref{assum:coefficident}),
  (H.\ref{assum:nondegeneracy})
and (H.\ref{assum:D_pi})
 hold.
Assume further that 
 $D\in C([0,T];\sR^{d\t k})$,
 $R\in C([0,T]; \sS^k)$ 
 and 
  $\bar{V} \in C([0,T]; \sS^k_+)$.
 Let  
 $\theta^0  \in 
 L^2(0,T; \sR^{k\t d})\t C([0,T]; \sS^k_+)
 $,
 let $(\pi_m)_{m\in \sN}\subset \mathscr{P}_{[0,T]}$ be  such that $\lim_{m\to \infty}|\pi_m|=0$,
and 
let
$(\theta^{\pi_m,0})_{m\in \sN}\subset \Theta$
be such that 
$\theta^{\pi_m,0} \in \Theta^{\pi_m}$ for all $m\in \sN$, 
  $\lim_{m\to \infty}\|\theta^{\pi_m,0}-\theta^0\|_{L^2\t L^\infty}=0$.
   Let
  $\ol{\lambda}_{ 0}>0$
be  such that $\ol{\lambda}_{0} I_k \succeq D^\top P^{\theta^0} D+R+\rho \bar{V}^{-1}$,
with 
$ P^{\theta_0}\in C([0,T];\sS^d)$ defined in  \eqref{eq:lyapunov_reg},
 and 
for  each 
$\tau>0$, 
let $(\theta^n)_{n\in \sN}$ 
and $(\theta^{\pi_m,n})_{m,n\in \sN}$
be defined in \eqref{eq:NPG}
and \eqref{eq:NPG_discrete}, respectively.
Then for all 
$\tau \in (0,{1}/\ol{\lambda}_0]$
and $N\in \sN_0$,
 $$
\lim_{m\to \infty}
\sup_{n=0,\ldots, N}
| \cC(\theta^{\pi_m,n})-\cC(\theta^{n})|=0.
$$
\end{Proposition}

\begin{proof}
For each $L>0$, define 
$
\Theta_L= \left\{\theta=(K,V)\in \Theta \,\Big\vert\, \tfrac{1}{L}I_k\preceq V\preceq LI_k \right\}
$.
Let $\tau \in (0,{1}/\ol{\lambda}_0]$
 be fixed.
By Proposition   \ref{prop:V_bound},
 there exists 
$\ul{\lambda}_V, \ol{\lambda}_V>0$ 
such that 
$\ul{\lambda}_V I_k
 \preceq V^n
 \preceq \ol{\lambda}_V I_k
$ for all $n\in \sN_0$. 
Moreover, by  the continuity of $D$, $R$ and $\bar{V}$, 
and the expressions    \eqref{eq:D_V} and \eqref{eq:NPG}, 
a straightforward induction argument  
 shows that
$V^n\in C([0,T]; \sS^k_+)$ for all $n\in \sN_0$.

 We first prove  by induction
that for all $n\in \sN_{  0}$,
there exists $L>0, m_0\in \sN$ such that 
\bb\label{eq:induction_n}
\lim_{m\to \infty}\|\theta^{\pi_m,n}-\theta^n\|_{L^2\t L^\infty}=0,
\; \textnormal{and
}
 \theta^{\pi_m,n}\in \Theta_L \cap \Theta^{\pi_m}, \;
 \fa 
 m\ge m_0.
\ee
Note that as $\theta^0 \in \Theta$
and     $\lim_{m\to \infty}\|V^{\pi_m,0}-V^0\|_{L^\infty}=0$,
there exists $L>0$ such that 
$\tfrac{1}{L}I_k\preceq 
V^{\pi_m,0}
\preceq LI_k $ 
for all large $m\in \sN$. 
This proves  \eqref{eq:induction_n} for $n=0$.
Now suppose that the induction statement \eqref{eq:induction_n} holds
for some $n\in \sN_{  0}$. 
As $V^n \in  C([0,T]; \sS^k_+)$, 
by \eqref{eq:NPG} and (H.\ref{assum:D_pi}),
the triangle inequality shows  that 
$\lim_{m\to \infty}\|\theta^{\pi_m,n+1}-\theta^{n+1}\|_{L^2\t L^\infty}=0$,
which subsequently implies that 
there exists $L>0$ such that $\tfrac{1}{L}I_k\preceq V^{\pi_m,n+1}\preceq LI_k$ 
for all sufficiently large $m$.
This   proves  the   statement 
\eqref{eq:induction_n}  for $n+1$.

By \eqref{eq:induction_n}, for each $n\in \sN$,
$\sup_{m\in \sN}\|K^{\pi_m,n}\|_{L^2}<\infty$
and $\limsup_{m\in \sN}\|V^{\pi_m,n}\|_{L^\infty}<\infty$.
Thus by Lemma \ref{lemma:Sigma_bdd},
there exists $C\ge 0$ such that
$0 \preceq \Sigma^{\theta^{\pi_m,n}}\preceq C I_d$ for all $m\in \sN$. 
Then
$\lim_{m\to \infty}
| \cC(\theta^{\pi_m,n})-\cC(\theta^{n})|=0$
follows from 
Proposition \ref{prop:Lipshcitz_smooth}
and 
$\lim_{m\to \infty}\|\theta^{\pi_m,n}-\theta^{n}\|_{L^2\t L^\infty}=0$.
This implies  the desired convergence result for any given $N\in \sN$.
\end{proof}

\begin{proof}[Proof of Theorem \ref{thm:mesh_independent_general}]
Let $C^\star=\inf_{\theta\in \Theta}\cC(\theta)=C(\theta^\star)$,
and
for each $\tau>0$ and $m\in \sN$,
let $(\theta^n)_{n\in \sN}$ 
and 
$(\theta^{\pi_m,n})_{n\in \sN}$ be defined by 
  \eqref{eq:NPG} and \eqref{eq:NPG_discrete}
with stepsize $\tau$,
respectively.  
Then 
by Theorem \ref{thm:linear_conv}
and  Proposition \ref{prop:discrete_convergence},
there exists  $\tau_0>0$   such that 
for all $\tau\in (0,\tau_0]$ and $n\in \sN_0$,
$\cC(\theta^{n+1})\le \cC(\theta^{n})$,
$\cC(\theta^{n+1})-\cC^\star\le \eta(\cC(\theta^{n})-\cC^\star)$ for some $\eta \in [0,1)$ 
(independent of $n$),
and 
 $\lim_{m\to \infty}
| \cC(\theta^{\pi_m,n})-\cC(\theta^{n})|=0$.
Moreover, for all $\eps>0$, 
$N(\eps)=\frac{\widetilde{C}}{\tau }\log(\frac{\widetilde{C}}{\eps})$ for some 
$\widetilde{C}>0$ independent of $\tau$ and $\eps$.

 We first prove 
  for all $\tau\in (0,\tau_0]$ and 
all $\eps,\gamma>0$, there exists $\mathfrak{m}_{\eps,\gamma}\in \sN$ such that 
for all $m\ge \mathfrak{m}_{\eps,\gamma}$,
\bb\label{eq:compare_N_eps_delta}
N(\eps+\gamma ) \le N^{\pi_m}(\eps)\le N(\eps).
\ee
To prove $N^{\pi_m}(\eps)\le N(\eps)$,
by Lemma \ref{lemma:discrete_value_conv}
and the choice of $\tau_0$,
 for all $n\in \sN_{ 0}$, 
 $\lim_{m\to \infty}
 ( \cC(\theta^{\pi_m, n})-\cC^\star_{\pi_m})= \cC(\theta^{n})-\cC^\star$.
 Hence,  for all $\eps>0$ and $n\in \sN_{  0}$, 
if $ \cC(\theta^{n})-\cC^\star<\eps$, then 
there exists $\mathfrak{m}_{\eps}\in \sN$ such that 
for all  $m\ge \mathfrak{m}_{\eps}$, 
 $ \cC(K^{\pi_m, n})-\cC^\star_{\pi_m}<\eps$,
 which implies 
$N^{\pi_m}(\eps)\le N(\eps)$
 for all $m\ge \mathfrak{m}_{\eps}$.
 We then prove  
 $N(\eps+\gamma ) \le N^{\pi_m}(\eps)$ with a given 
 $\gamma>0$.
 The convergence of 
 $(\cC(\theta^{n}))_{n\in \sN}$
 implies that $ N(\eps+\gamma)\in \sN_{ 0}$, which along with Lemma \ref{lemma:discrete_value_conv} and Proposition \ref{prop:discrete_convergence} shows that 
\bb\label{eq:convergence_pi_delta_eps}
 \lim_{m\to \infty} \max_{0\le n\le N(\eps+\gamma)}
 \big|
  ( \cC(\theta^{\pi_m, n})-\cC^\star_{\pi_m})
  -(\cC(\theta^{n})-\cC^\star)
 \big|=0.
 \ee
The definition of $N(\eps+\gamma)$ implies that 
$\cC(\theta^{n})-\cC^\star\ge \eps+\gamma $ for all $n<N(\eps+\gamma)$.
Moreover, by 
\eqref{eq:convergence_pi_delta_eps}, 
there exists $\mathfrak{m}_{\gamma}\in \sN$ such that 
for all $m\ge \mathfrak{m}_{\gamma}$, 
 $$
\max_{0\le n < N(\eps+\gamma)}
 \big|
  ( \cC(\theta^{\pi_m, n})-\cC^\star_{\pi_m})
  -(\cC(\theta^{n})-\cC^\star)
\big|\le \gamma.
  $$
Hence
  for all $m\ge \mathfrak{m}_{\gamma}$
and 
  $n<N(\eps+\gamma)$,
   \begin{align*}
 \cC(\theta^{\pi_m, n})-\cC^\star_{\pi_m}
&
= (\cC(\theta^{\pi_m, n})-\cC^\star_{\pi_m})
- (\cC(\theta^{n})-\cC^\star)+ (\cC(\theta^{n})-\cC^\star)
\\
& \ge (\cC(\theta^{n})-\cC^\star)
- \big|
 (\cC(\theta^{\pi_m, n})-\cC^\star_{\pi_m})
- (\cC(\theta^{n})-\cC^\star)
\big|
\\
&\ge 
(\cC(\theta^{n})-\cC^\star)-
\max_{0\le n < N(\eps+\gamma)}
 \big|
 (\cC(\theta^{\pi_m, n})-\cC^\star_{\pi_m})
- (\cC(\theta^{n})-\cC^\star)
\big|\ge \eps.
\end{align*}
 This implies that 
 $N^{\pi_m}(\eps)\ge N(\eps+\gamma)$
for all $m\ge \mathfrak{m}_{\gamma}$.
 Taking $\mathfrak{m}_{\eps,\gamma}=\max(\mathfrak{m}_{\eps},\mathfrak{m}_{\gamma})$ completes the proof of \eqref{eq:compare_N_eps_delta}.

Now we are ready to establish \eqref{eq:N_pi_bound_1}
for    fixed $\tau\in (0,\tau_0]$ and $\eps>0$. 
By the choice of $\tau_0$,
there exists  $\eta\in [0,1)$, independent of $\eps$,  such that for all $n\in \sN_{ 0}$,
$\cC(\theta^{n+1})-\cC^{\star} \le \eta(\cC(\theta^{n})-\cC^{\star})$. Then, by the definition of $N(\eps)$, $\cC(\theta^{n})-\cC^\star\ge \eps$ for all $n<N(\eps)$, 
which yields  the estimate
\begin{align*}
\eta^{N(\eps)-1-n} (\cC(\theta^{n})-\cC^\star)
\ge  \cC(\theta^{N(\eps)-1})-\cC^\star\ge
 \eps, 
\q \fa n< N(\eps)-1.
\end{align*}
This implies that  
$\cC(\theta^{n})-\cC^{\star}\ge \frac{\eps}{\eta}>\eps$ for all 
$n< N(\eps)-1$. 
Now let 
 $\gamma_\eps\coloneqq \min\{\cC(\theta^{n})-\cC^{\star}-\eps\mid n< N(\eps)-1\}$.
Note that $\gamma_\eps>0$
as  $N(\eps)<\infty$. By the definition of $\gamma_\eps$,  for all $n< N(\eps)-1$, 
$\cC(\theta^{n})-\cC^{\star}\ge \eps+\gamma_\eps$,
which implies   that $N(\eps+\gamma_\eps )\ge N(\eps)-1$. Hence, by 
\eqref{eq:compare_N_eps_delta},
   there exists $\mathfrak{m}_{\eps}\in \sN$ such that 
$$
N(\eps)-1\le N(\eps+\gamma_\eps )\
   \le N^{\pi_m}(\eps)\le N(\eps),
   \q \fa m\ge \mathfrak{m}_{\eps}.
$$
This  proves the desired estimate \eqref{eq:N_pi_bound_1}.
 \end{proof}

\begin{proof}[Proof of Corollary \ref{corollary:mesh_independent_piecewise}]
By Proposition \ref{prop:Gateaux}
and \eqref{eq:finite_parameterisation},
 for all
$\pi\in \mathscr{P}_{[0,T]}$,
$\theta\in \Theta^\pi$
and $i\in \{0,\ldots, N-1\}$,
 \begin{align*}
{\nabla_{K_i}\cC}(\theta)
=\int_{t_i}^{t_{i+1}}    
\cD_K(\theta)_t\Sigma^\theta_t\,\d t,
\q
{\nabla_{V_i}\cC}(\theta)
=\int_{t_i}^{t_{i+1}}    
\cD_V(\theta)_t\,\d t,
\end{align*}
where $\cD_K(\theta)$ and 
$\cD_V(\theta)$ are defined by \eqref{eq:D_K}
and \eqref{eq:D_V}, respectively. 
Hence  
$(\cD^\pi_K,\cD^\pi_V):  \Theta^\pi\to  \Theta^\pi$
in \eqref{eq:NPG_discrete_finite}
satisfies  for all $\theta\in \Theta^\pi$,
and  a.e.~$t\in [0,T]$,
 \begin{align}
 \label{eq:D_K_V_finite}
 \begin{split}
\cD^\pi_K(\theta)_t&= \sum^{N-1}_{i=0} 
\left(\frac{1}{t_{i+1}-t_i}
\int_{t_i}^{t_{i+1}}    
\cD_K(\theta)_t\Sigma^\theta_t\,\d t
\right)
\left(\Sigma_{t_i}^{\theta}
\right)^{-1}\mathds{1}_{[t_i,t_{i+1})}(t),
\\
\cD^\pi_V(\theta)_t&= \sum^{N-1}_{i=0} 
\left(\frac{1 }{t_{i+1}-t_i}
\int_{t_i}^{t_{i+1}} 
(V_{t_i}   
\cD_V(\theta)_t
+\cD_V(\theta)_t V_{t_i})
\, \d t
\right)
\mathds{1}_{[t_i,t_{i+1})}(t).
\end{split}
\end{align}
To simplify the notation, 
for each Euclidean space $E$,
let $\cP\cC_\pi(E)$ be   
 the space of piecewise constant functions
 $f:[0,T]\to E$   on  $\pi$,
let $\Pi^\pi:L^2(0,T;E) \to \cP\cC_\pi(E) $ be
such that for all $f\in L^2(0,T;E)$, 
   $\Pi^\pi(f)_t \coloneqq \sum^{N-1}_{i=0} 
  \left( \frac{1}{t_{i+1}-t_i}\int_{t_i}^{t_{i+1}}
f_t\, \d t\right)
    \mathds{1}_{[t_i,t_{i+1})}(t)$
  for all $t\in [0,T]$,
  and let  $\cT^\pi: C([0,T];E)\to \cP\cC_\pi(E) $ 
   be  such that for all $f\in C([0,T];E)$, 
   $\cT^\pi(f)_t \coloneqq \sum^{N-1}_{i=0} 
 f_{t_i}
    \mathds{1}_{[t_i,t_{i+1})}(t)$
  for all $t\in [0,T]$.
  Note that 
$\Pi^{\pi_m} $
 is  the  orthogonal projection with respect to the $\|\cdot\|_{L^2}$ norm, 
 and hence is 
$1$-Lipschitz continuous with respect to the $\|\cdot\|_{L^2}$ norm.
Moreover,  by \eqref{eq:D_K_V_finite},  for all $\theta\in \Theta^\pi$,
 \begin{align}
 \label{eq:D_K_V_projection}
 \begin{split}
\cD^\pi_K(\theta)&= 
\Pi^\pi\left(\cD_K(\theta)\Sigma^\theta
\left(\cT^\pi\big(\Sigma^\theta \big)\right)^{-1}
 \right),
\q 
\cD^\pi_V(\theta) = 
\Pi^\pi\left(
\cT^\pi(V)
\cD_V(\theta)
+\cD_V(\theta)\cT^\pi(V)
 \right).
\end{split}
\end{align}
The definition
of $( \cD^\pi_K,\cD^\pi_V)$ 
in \eqref{eq:D_K_V_projection} can be naturally extended to all $\theta\in L^2(0,T; \sR^{k\t d})\t C([0,T]; \sS^k_+) $.
Note that $\Sigma^\theta$ is pointwise invertible due to   $\sE[\xi_0\xi_0^\top]\succ 0$ (see Lemma \ref{lemma:Sigma_bdd}).

We are now ready to  verify  (H.\ref{assum:D_pi})
for \eqref{eq:D_K_V_projection}.
Let  $\theta \in L^2(0,T; \sR^{k\t d})\t C([0,T]; \sS^k_+) $, 
 $(\pi_m)_{m\in \sN}\subset \mathscr{P}_{[0,T]}$
 be 
such that $\lim_{m\to \infty} |\pi_m|=0$,
and     $(\theta^{m})_{m\in \sN}\subset \Theta$  
be such that 
$\theta^{m} \in \Theta^{\pi_m}$ for all $m\in \sN$
and 
  $\lim_{m\to \infty}\|\theta^{m}-\theta\|_{L^2\t L^\infty}=0$.
 Then for all $m\in \sN$,
by the Lipschitz continuity of $\Pi^{\pi_m}$,
\begin{align}
 &\|\cD^{\pi_m}_K(\theta^m)-\cD_K(\theta)\|_{L^2}
 \nb
\\
&\q \le 
\|\cD^{\pi_m}_K(\theta^m)-\Pi^{\pi_m}
\left(\cD_K(\theta)\right)
\|_{L^2}
+\|\Pi^{\pi_m} \left(\cD_K(\theta)\right)
-\cD_K(\theta)\|_{L^2}
\nb
\\
&\q \le 
\left\|
\cD_K(\theta^m)\Sigma^{\theta^m}
\left(\cT^{\pi_m}\big(\Sigma^{\theta^m} \big)\right)^{-1}
-\cD_K(\theta)
\right\|_{L^2}
+\|\Pi^{\pi_m} \left(\cD_K(\theta)\right)
-\cD_K(\theta)\|_{L^2}.
\label{eq:D_K_pi_conv}
\end{align}
The density of $(\cP\cC_\pi(\sR^{k\t d}))_{m\in \sN}$ in $L^2(0,T;\sR^{k\t d})$ shows that the second term of \eqref{eq:D_K_pi_conv}
tends to zero as $m\to \infty$.
Standard stability results of 
\eqref{eq:lyapunov_reg} and \eqref{eq:lq_sde_K_reg_cov}
(see, e.g., Lemma \ref{lemma:laypunov_difference})
 show that 
$\lim_{m\to \infty}\| P^{\theta^m} - P^\theta\|_{L^\infty}=0$
and 
$\lim_{m\to \infty}\|\Sigma^{\theta^m}- \Sigma^\theta\|_{L^\infty}=0$.
Thus 
by  (H.\ref{assum:coefficident}) and  \eqref{eq:D_K},
$\lim_{m\to \infty}\|\cD_K(\theta^m)-\cD_K(\theta)\|_{L^2}=0$.
Moreover,  as  $\inf_{  t\in [0,T]}\lambda_{\min}(\Sigma^{\theta}_t)>0$ (see  Lemma \ref{lemma:Sigma_bdd}),
$\Sigma^{\theta^m}
\left(\cT^{\pi_m}\big(\Sigma^{\theta^m} \big)\right)^{-1}
$ tends to the identity function in $L^\infty$ as $m\to \infty$.
Consequently, the first term of \eqref{eq:D_K_pi_conv}
tends to zero as $m\to \infty$, which proves  
$\lim_{m\to \infty} \|\cD^{\pi_m}_K(\theta^m)-\cD_K(\theta)\|_{L^2}=0$. 

We then prove  the convergence of 
$(\cD^{\pi_m}_V(\theta^m))_{m\in \sN}$.
Note  that 
for each $m\in \sN$ and 
Euclidean space $E$, 
$\|\Pi^{\pi_m}(f)\|_{L^\infty}\le \|f\|_{L^\infty}$
if $f\in L^\infty( 0,T;  E )$,
and 
$\lim_{m\to \infty}
 \|\Pi^{\pi_m}(f)-f\|_{L^\infty}=0$
if  $f\in C([ 0,T] ;  E )$. 
The same property also holds for the operator $\cT^{\pi_m}$. 
Then
for all $m\in \sN$,
 \begin{align}
\label{eq:D_V_pi_conv}
 \begin{split}
&\|\cD^{\pi_m}_V(\theta^m)-\cD^{\rm bw}(\theta)\|_{L^\infty}
\\
&\q  \le 
\|\cD^{\pi_m}_V(\theta^m)-\Pi^{\pi_m} (\cD^{\rm bw}_V(\theta))\|_{L^\infty}
+\|\Pi^{\pi_m} (\cD^{\rm bw}_V(\theta))
-\cD^{\rm bw}_V(\theta)\|_{L^\infty}.
\end{split}
\end{align}
By 
the   continuity of $D$, $R$, $\bar{V}$ and $V$,
$ \cD^{\rm bw}_V(\theta)\in C([0,T];\sS^k_+)$
(cf.~\eqref{eq:D_V} and \eqref{eq:gradient_BW}),
and hence the second term in 
\eqref{eq:D_V_pi_conv} tends to zero as $m\to \infty$.
To show the first term tends to zero, 
 by \eqref{eq:gradient_BW} and  \eqref{eq:D_K_V_projection},
it suffices to prove    
$\lim_{m\to \infty}
 \|\cD_V(\theta^m)-\cD_V(\theta)\|_{L^\infty}=0$. 
This follows directly from 
the facts that
$\lim_{m\to \infty}\| P^{\theta^m} - P^\theta\|_{L^\infty}=0$,
$\lim_{m\to \infty}\|V^{m}-V\|_{ L^\infty}=0$
and 
$V\in C([0,T]; \sS^k_+) $. 
This  verifies   (H.\ref{assum:D_pi})
for \eqref{eq:D_K_V_projection}.
\end{proof}

\section{Numerical experiments} 
\label{sec:numerical}

In this section, we  test the theoretical findings 
 through a numerical experiment on an   
exploratory   LQC  problem 
arising from mean-variance portfolio
selection.
%
Our experiments confirm  that  the proposed  iteration 
\eqref{eq:NPG_discrete_finite}
converges linearly to the optimal policy.
They also show that  
conventional PG methods 
exhibit a degraded performance for small timesteps in the policy updates, while our algorithm
demonstrates robustness across different step sizes.

 \paragraph{Problem setup.} 
We minimise the  following cost  
 $\cC:\Theta\to \sR$ (cf.~\eqref{eq:reg_cost_closed_loop}):
  \begin{equation}
  \label{eq:cost_numerical}
  \cC(\theta)=\sE\left[\frac{1}{2} \mu \, (X^\theta_T)^2 + {\rho} \int_0^T  \cH(\nu^\theta_t(X^\theta_t)\| \ol{\mathfrak{m}}_t) \, \d t\right],
\end{equation}
   where
      $\ol{\mathfrak{m}}_t=\cN(0,\bar{V})$ with $\bar{V}\in \sS^3_+$,
and
    for each $\theta\in \Theta$, 
   $X^\theta \in\cS^2(0,T;\sR)$ satisfies  for all $ t \in[0, T]$,  
\begin{align}
\label{eq:state_numerical}
    \begin{split}
\mathrm{d} X_t=
\int_{\mathbb{R}^3}\left(B_t a\, \nu^\theta_t(X_t; \d a)\right) 
  \mathrm{d} t+
  \bigg(\int_{\mathbb{R}^3}\sum_{j=1}^3\Big( D^{(j)} a\Big)^2  \, \nu^\theta_t(X_t; \d a)  \bigg)^{\frac{1}{2}}
 \mathrm{d} W_t,\quad 
  X_0=\xi_0,
    \end{split}
\end{align}
for some $B:[0,T]\to \sR^{1\t 3}$ and $D^{(j)} \in \sR^{1\t 3}$, $j=1,2,3$.
The coefficients are chosen as follows: 
  $T=1$, 
$\mu =0.5$, $\rho=0.01$, 
$\bar{V}= 0.1 I_3$, 
$\xi_0\sim\mathcal{N}(0.5,0.01)$, 
$B_t= (0.4, 0.8,0.4) +0.2 \sin(2\pi t)\mathbf{1}_{3}  $ for all $t\in [0,T]$,
and 
$D = \left(\begin{smallmatrix} 
D^{(1)}\\
D^{(2)}\\
D^{(3)}
\end{smallmatrix}\right) $ with  
$D^\top D = 
\left(\begin{smallmatrix}
0.5 & 0.25 & -0.125\\
0.25 & 1 & -0.25\\
-0.125 & -0.25 & 0.5
\end{smallmatrix}\right)
$. 
Note that    $D^\top D\in \sS^3_+$, 
and hence
(H.\ref{assum:nondegeneracy}) holds for all $\rho\ge 0$
(see  \cite{zhou2000continuous} and Remark \ref{rmk:solvability_riccati}).

The   problem 
\eqref{eq:cost_numerical}-\eqref{eq:state_numerical}
arises from an  
  exploratory mean-variance portfolio selection problem,
  where the agent allocates their wealth among three risky assets by sampling from  the   policy 
  $\nu^\theta$ (see \cite{wang2020continuous}). Indeed, as illustrated at the end of Section \ref{sec:mesh-independence}, 
for each  $\theta=(K,V)\in \Theta$,  $\cC(\theta)$  can be approximated by replacing  
  \eqref{eq:state_numerical} with the following   dynamics:
$  X_0=\xi_0$, and for all $t\in [0,T]$,
 \begin{align}
\label{eq:state_numerical_randomised}
    \begin{split}
\mathrm{d} X_t=
B_t \left( K_tX_t+V^{\frac{1}{2}}_t\xi_t\right)
  \mathrm{d} t+
\sum_{j=1}^3 D^{(j)} \left( K_tX_t+V^{\frac{1}{2}}_t\xi_t\right) 
 \mathrm{d} W^{(j)}_t
    \end{split}
\end{align}
with 
$\xi_t=\sum_{i=1}^n \zeta_i \mathds{1}_{[t_i,t_{i+1})}(t)$,
where  $(W^{(j)})_{j=1}^3$ are independent Brownian motions, 
$(\zeta_i)_{i=1}^n$ are independent standard normal random vectors, 
and $(t_i)_{i=1}^n$ is a sufficiently fine time mesh.

\paragraph{Linear convergence.} 
 We first 
 implement   \eqref{eq:NPG_discrete_finite}   on the  uniform time mesh   $\pi_c$ with  mesh size $1/128$,
 and examine its convergence.
The scheme    is initialised with 
$K^0\equiv (1/3,1/3,1/3)$ and $V^0 \equiv 0.1D^\top D$.
 For each $ n\in \sN_0$, given    $ {\theta^n} \subset \Theta^{\pi_c}$, 
 we simulate $10^5$ independent   trajectories of 
 \eqref{eq:state_numerical_randomised} (with $\theta= \theta^n$)
 using the Euler--Maruyama method on the mesh $\pi_c$,
 evaluate the approximate value $\widehat{\cC}(\theta^n)$ 
 and  state covariance $\widehat{\Sigma}^n$
using the empirical distribution of  these sample paths, 
and compute an approximate gradient 
$(\widehat{{\nabla_{\theta^n_i}\cC}})_{i=0}^{127}$ using automatic differentiation. 
The  iterate $\theta^{n}$ is updated
by  \eqref{eq:NPG_discrete_finite} 
with
  $\widehat{\Sigma}^n$, $\widehat{{\nabla_{\theta^n}\cC}}$
 and the  stepsize $\tau=0.01$.
  The performance  of the scheme is measured by the errors $(\widehat{\cC}(\theta^n)-\cC^\star)_{n\in \sN_0}$,
where $\cC^\star$ is the optimal cost of \eqref{eq:cost_numerical} 
obtained by   Riccati equations. 
Further implementation details are given in Appendix  \ref{sec:experiment_detail}. 
 
 Figure \ref{fig:pgm} (left) exhibits the decay of 
 $(\widehat{\cC}(\theta^n)-\cC^\star)_{n\in \sN_0}$
with respect to the number of iterations,   where the solid line and the shaded area  indicate the sample mean and the spread   over 10 repeated experiments, respectively.
 It clearly shows the linear convergence of \eqref{eq:NPG_discrete_finite}, 
 as indicated in  Theorems    \ref{thm:linear_conv}  and \ref{thm:mesh_independent_general}. 
  {The seemingly higher noise for larger iteration numbers results from the small errors in this case, so that the 
 fluctuations appear larger on the log scale. The variance could be reduced by increasing the number of samples.}
 
 \paragraph{Robustness in action frequency.} 
We then compare the performance of  \eqref{eq:NPG_discrete_finite} 
with a standard PG method 
for different policy discretisation timescales. 
The former  (termed ``scaled PG")  scales the gradients with the discretisation
mesh size, while the latter  (termed ``unscaled PG")  updates the policy with unscaled gradients. 
More precisely, 
let $\theta^0=(K^0,V^0)$ be a fixed initial guess given  as above,
 and $\pi_m=\{i\frac{1}{m}\}_{i=0}^m $, $m\in \{8,16,32,64,128\}$ be a family of   time meshes.
 For each $m\in \{8,16,32,64,128\}$, 
 the scaled PG method  generates the iterates $(\theta^{\pi_m,n})_{n\in \sN_{ 0}}\subset  \Theta^{\pi_m}$
according to  \eqref{eq:NPG_discrete_finite} with  $\tau=0.01$ and $\Delta_i=1/m$,
where  the required gradients  for each iteration  are computed as above.
The unscaled PG  method follows  \eqref{eq:NPG_discrete_finite} with
$\tau= 0.08$ and
 $\Delta_i=1$ for all $m$. 
Here,  a   larger stepsize has been adopted for the unscaled PG  method
so that the two algorithms coincide  for the coarsest mesh $\pi_8$. 

  Figure \ref{fig:pgm} (right) compares, for different discretisation timescales,
  the   numbers of required iterations 
  $N^{\pi_m}(0.01)$
  for both schemes to achieve an accuracy of $\epsilon=0.01$ (cf.~\eqref{eq:N_pi_eps}). 
   One can observe clearly   that the number of required iterations for the unscaled PG  method   exhibits a linear growth in the number of action time points. 
 In constrast,  
the     number of   iterations for the scaled PG method remains constant       for all     meshes.
   This 
   confirms the theoretical result  in Theorem \ref{thm:mesh_independent_general},
   and 
   shows that the scaled PG method outperforms conventional PG methods for fine meshes.

 \begin{figure}[H]
\centering

\includegraphics[trim=5 5 30 25, clip,  width=0.49\textwidth]{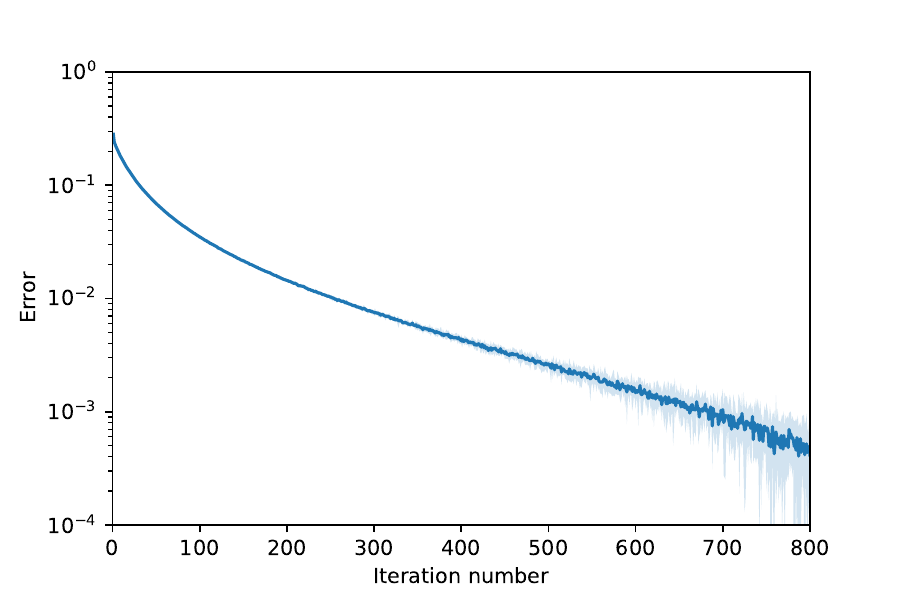}
\hfill
\includegraphics[trim=5 5 30 25, clip,  width=0.49\textwidth]{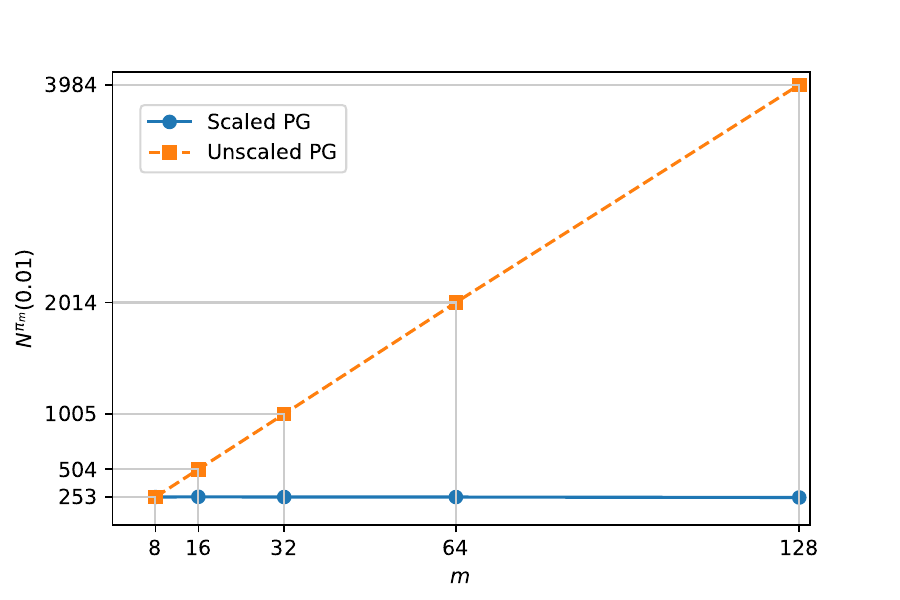}
 \caption{Convergence and robustness  of the PG method \eqref{eq:NPG_discrete_finite}.}
\label{fig:pgm}
\end{figure}

\appendix

\section{Proofs of technical results}
\label{appendix:technical}
 
 The following lemma establishes  the well-posedness of stochastic differential equations, 
whose coefficients   
are  Lipschitz continuous in state  with time-dependent Lipschitz constants.
The proof follows essentially  
the lines of 
  Theorems 3.2.2
  and   3.3.1 (Method 2) in \cite{zhang2017backward},
  and hence is omitted.

\begin{Lemma}
\label{lemma:sde_wp}
Let   
$T>0$,
 $(\Om,\cF,\sF,\sP)$ be a filtered probability space
 satisfying the usual condition,
  $b:\Om\t [0,T]\t \sR^d\to \sR^d$
 and ${\sigma}:\Om\t [0,T]\t \sR^d\to \sR^{d\t d}$ 
  be  progressively measurable functions such that $b_\cdot(\cdot,0)  \in L^1(\Om\t [0,T];\sR^d)$
  and 
   $\sigma_\cdot( \cdot,0)  \in L^2(\Om\t [0,T];\sR^{d\t d})$.
   Assume that 
  there exists  
  $A\in L^1(0,T; \sR )$ and 
    $C\in L^2(0,T; \sR )$ such that  
    for all $(\om,t)\in\Om\t [0,T]$ and $x,x'\in \sR^d$,
$|b_t(\om,x)-b_t(\om,x')|\le |A_t| |x-x'|$ and 
$|\sigma_t(\om,x)-\sigma_t(\om,x')|\le |C_t| |x-x'|$.
Then for all $\xi_0\in L^2(\cF_0;\sR^d)$, 
there exists a  unique strong solution $X\in \cS^2(0,T;\sR^d)$
to
the following equation
\begin{equation}
\d X_t= b_t(X_t)\, \d s +\sigma_t(X_t)\, \d W_t, 
\q t\in [0,T]; 
\quad X_0 = \xi_0.
\end{equation}
\end{Lemma}

\begin{Proposition}
\label{prop:state_wp}
Suppose (H.\ref{assum:coefficident}\ref{assum:integrability}) holds.
Then 
\begin{enumerate}[(1)]
\item \label{eq:state_wp_open}
  for all    $\mathfrak{m}\in \cA$,
\eqref{eq:reg_state_open_loop} admits a unique  strong solution $X^\mathfrak{m}\in \cS^2(0,T;\sR^d)$.
\item \label{eq:state_wp_close}
for all $\nu^\theta\in \cV$, 
\eqref{eq:reg_state_closed_loop} admits a unique  strong solution 
$X^\theta\in \cS^2(0,T;\sR^d)$.
\end{enumerate}
\end{Proposition}

\begin{proof}
Let 
$E=[0,T]\t \sR^k$.
We verify that 
the coefficients of \eqref{eq:reg_state_open_loop}
and \eqref{eq:reg_state_closed_loop}
satisfy 
 the conditions of Lemma \ref{lemma:sde_wp}.
   

To prove Item \ref{eq:state_wp_open},
let   $\mathfrak{m}\in \cA$ be given, and 
define 
$\Phi^\mathfrak{m}:\Om\t [0,T]\t \sR^d\to \sR^d$ and 
$\Gamma^\mathfrak{m}:\Om\t [0,T]\t \sR^d\to \ol {\sS^d_+}$ 
such that for all $(\om,t,x)\in \Om\t [0,T]\t \sR^d$, 
$\Phi^\mathfrak{m}_t(\om,x)= \Phi_t(x, \mathfrak{m}_t(\om) )$ 
and 
$\Gamma^\mathfrak{m}_t(\om,x)= \Gamma_t(x, \mathfrak{m}_t(\om) )$,
with $\Phi$ and $\Gamma$  defined in \eqref{eq:reg_coefficient}.
By Fubini's theorem and  H\"{o}lder's inequality,
\begin{align*}
\sE\left[\int_0^T|\Phi^\mathfrak{m}_t(\cdot,0)| \right]\, \d t
&\le \int_0^T\left(\sE\left[\int_{\sR^k} |a|\, \mathfrak{m}_t(  \d a)\right]|B_t|\right)\d t
\le \|B\|_{L^2}\left(\sE\left[\int_{E} |a|^2\, \mathfrak{m}_t( \d t, \d a )\right]\right)^{\frac{1}{2}}<\infty,
\\
\sE\left[\int_0^T|\Gamma^\mathfrak{m}_t(\cdot,0)|^2\, \d t \right]
&\le \widetilde{C} 
\|D\|_{L^\infty}^2 \sE\left[\int_{E} |a|^2\, \mathfrak{m}_t( \d t, \d a )\right]<\infty.
\end{align*}
For all $(\om,t)\in \Om\t [0,T]$, 
using   $\mathfrak{m}_t(\om)\in \cP(\sR^k)$, 
 $|\Phi^\mathfrak{m}_t(\om,x)-\Phi^\mathfrak{m}_t(\om,x')|\le |A_t| |x-x'|$
for all $x,x'\in\sR^d$.
To prove the Lipschitz continuity of $\Gamma^\mathfrak{m}$,
observe that  
 for all $(t,x,m)\in  [0,T] \t \sR^d \t \cP_2(\sR^k)$, 
$  \Gamma_t(x, m) = (M_{m,t}(x)M_{m,t}(x)^\top +N_{m,t} N_{m,t}^\top)^{1/2}$,
where 
$M_{m,t}(x) \coloneqq C_tx+D_t\int_{\sR^k}a\, m( \d a)$ 
and 
$N_{m,t} \coloneqq D_t\left(\int_{\sR^k} aa^\top \,  m( \d a)\right)^{1/2}$. 
This implies that
$$
\begin{pmatrix}
\Gamma_t(x, m) & 0_{d\t d}
\\
0_{d\t d} & 0_{d\t d}
\end{pmatrix}
=
\begin{pmatrix}
M_{m,t}(x)M_{m,t}(x)^\top +N_{m,t} N_{m,t}^\top & 0_{d\t d}
\\
0_{d\t d} & 0_{d\t d}
\end{pmatrix}^{\frac{1}{2}}
=
\left|
\begin{pmatrix}
M_{m,t}(x) & N_{m,t} 
\\
0_{d \t 1} & 0_{d\t 1}
\end{pmatrix}
\right|_{\rm mat},
$$
where 
$0_{m\t n}$ is $m\t n$ zero matrix,
and $|\cdot|_{\rm mat}$
is the matrix absolute value defined by 
$|M|_{\rm mat}=(M M^\top)^{1/2}$ for any matrix $M$.
Then,  for all $(t,m)\in  [0,T]  \t \cP_2(\sR^k)$ and 
$x,x'\in \sR^d$,
\begin{align}
\label{eq:Sigma_lipschitz}
\begin{split}
&| \Gamma_t(x, m)- \Gamma_t(x', m)|
\\
&\quad =
\left|
\begin{pmatrix}
\Gamma_t(x, m)- \Gamma_t(x', m) & 0_{d\t d}
\\
0_{d\t d} & 0_{d\t d}
\end{pmatrix}
\right|
=
\left|
\begin{pmatrix}
\Gamma_t(x, m) & 0_{d\t d}
\\
0_{d\t d} & 0_{d\t d}
\end{pmatrix}
-\begin{pmatrix}
\Gamma_t(x', m) & 0_{d\t d}
\\
0_{d\t d} & 0_{d\t d}
\end{pmatrix}
\right|
\\
&\quad =\left|
\left|
\begin{pmatrix}
M_{m,t}(x) & N_{m,t} 
\\
0_{d\t 1} & 0_{d\t 1}
\end{pmatrix}
\right|_{\rm mat}
-
\left|
\begin{pmatrix}
M_{m,t}(x') & N_{m,t} 
\\
0_{d\t 1} & 0_{d\t 1}
\end{pmatrix}
\right|_{\rm mat}
\right|
\\
&\quad \le 
\sqrt{2}
\left|
 \begin{pmatrix}
M_{m,t}(x) & N_{m,t} 
\\
0_{d\t 1} & 0_{d\t 1}
\end{pmatrix}
 -
 \begin{pmatrix}
M_{m,t}(x') & N_{m,t} 
\\
0_{d\t 1} & 0_{d\t 1}
\end{pmatrix}
 \right|
 =\sqrt{2}|M_{m,t}(x)-M_{m,t}(x')|,
 \end{split}
\end{align}
where the last inequality used the Lipschitz continuity of the matrix absolute value 
$|\cdot|_{\rm mat}$ (see \cite{araki1981inequality}). 
Therefore, 
by the definition of $M_{m,t}(x)$,
 for all $(\om,t)\in \Om\t [0,T]$ and $x,x'\in\sR^d$, 
$$
|\Gamma^\mathfrak{m}_t(\om,x)-\Gamma^\mathfrak{m}_t(\om,x')|
\le \sqrt{2}|M_{\mathfrak{m}_t(\om),t}(x)-M_{\mathfrak{m}_t(\om),t}(x')|
\le 
\sqrt{2} |C_t| |x-x'|.
$$
As $A\in L^1(0,T;\sR^{d\t d})$ and $C\in L^2(0,T;\sR^{d\t d})$,
the coefficients of \eqref{eq:reg_state_open_loop}
satisfy 
 the conditions of Lemma \ref{lemma:sde_wp},
 which subsequently implies the well-posedness 
of \eqref{eq:reg_state_open_loop}.
 
 To prove Item \ref{eq:state_wp_close},
let   $\nu^\theta\in \cV$ be given, and 
define 
$\Phi^\theta:  [0,T]\t \sR^d\to \sR^d$ and 
$\Gamma^\theta:  [0,T]\t \sR^d\to \ol {\sS^d_+}$ 
such that for all $(t,x)\in   [0,T]\t \sR^d$, 
$\Phi^\theta_t( x)= \Phi_t(x, \nu^\theta_t(x) )$ 
and 
$\Gamma^\theta_t(x)= \Gamma_t(x, \nu^\theta_t(x) )$,
with $\Phi$ and $\Gamma$  defined in \eqref{eq:reg_coefficient}.
Then by Lemma \ref{lemma:representation_KV},
 $\Phi^\theta_\cdot(0)=0$ and 
$\Gamma^\theta_\cdot(0)=DV_\cdot^{\frac{1}{2}}\in L^2(0,T;\sR^{d \t d})$.
Moreover, for all $(t,x)\in [0,T]\t \sR^d$, 
$|\Phi^\theta_t(x)-\Phi^\theta_t(x')| \le (|A_t|+|B_tK_t|)|x-x'|$
and by \eqref{eq:Sigma_lipschitz}, 
$|\Gamma^\theta_t(x)-\Gamma^\theta_t(x')| \le (|C_t|+|D_tK_t|)|x-x'|$. 
By  (H.\ref{assum:coefficident}\ref{assum:integrability})  and $K\in L^2(0,T;\sR^{d\t k})$,
$|A|+|BK|\in L^1(0,T;\sR)$ and 
$|C|+|DK|\in L^2(0,T;\sR)$.
This proves that 
the coefficients of \eqref{eq:reg_state_closed_loop}
satisfy 
 the conditions of Lemma \ref{lemma:sde_wp},
and hence  the well-posedness 
of \eqref{eq:reg_state_closed_loop}.
\end{proof}

\begin{proof}[Proof of Proposition \ref{prop:non_coercive}]
 For each $\eps>0$, let $X^\eps=X^{K^\eps}$ be such that $X^\eps_t =\exp(-\int_0^t (1+\eps-s)^{-1}\, \d s)=\frac{1+\eps-t}{1+\eps}$ for all $t\in [0,1]$.
 Thus  for all $\eps>0$,
 $ \cC(K^\eps)= \frac{1}{(1+\eps)^2}  $ but $\|K^\eps\|_{L^1}=\log\left(\frac{1+\eps}{\eps}\right)$.
 
 Now let $\tilde{K}^\eps = 0.5 K^\eps$,   and 
 $\tilde{X}^\eps$ be such that $\tilde{X}^\eps_t =\exp(-0.5\int_0^t (1+\eps-s)^{-1}\, \d s)=\sqrt{\frac{1+\eps-t}{1+\eps}}$ for all $t\in [0,1]$.
 Thus  for all $\eps>0$,
 $$
  \cC(\tilde{K}^\eps)= 0.5^2 \int_0^1 (K^\eps_t \tilde{X}^\eps_t)^2 \,\d t
  = \frac{0.25}{1+\eps} \int_0^1 (1+\eps-t)^{-1} \, \d t
  = \frac{0.25}{1+\eps}\log\left(\frac{1+\eps}{\eps}\right)>0.
  $$
  Hence 
  $  \cC(\tilde{K}^\eps)>\cC(\boldsymbol{0})$ and 
   $\lim_{\eps\to 0}\frac{ \cC(\tilde{K}^\eps)}{ \cC({K})}
   = \lim_{\eps\to 0} 0.25 (1+\eps)\log\left(\frac{1+\eps}{\eps}\right)=\infty $. 
 \end{proof}

\section{Experiment details}
\label{sec:experiment_detail}

This section   presents   additional details for the numerical experiments   in Section \ref{sec:numerical}. 

\paragraph{Optimal    cost.} 
  Let  $P^\star\in C([0,T];\sR)$ solve  the  following Riccati equation: 
for all $t\in [0,T]$, 
\begin{align}\label{eq:riccati_num}
\begin{split}
     \left(\tfrac{\d P}{\d t}\right)_t&-B_t\left(P_t\textstyle\sum_{j=1}^3 (D^{(j)})^\top D^{(j)} +\rho \bar{V}^{-1}\right)^{-1}B^\top_tP^2_t=0;
     \quad 
     P_T = \tfrac{\mu}{2}.
\end{split}
\end{align}
Then
 the optimal policy of  \eqref{eq:cost_numerical}-\eqref{eq:state_numerical}
  satisfies 
$\nu^\star_t(x)=\cN(K^\star_t x,V_t^\star)$
for all $(t,x)\in [0,T]\t \sR$,
where 
\begin{align*}
\begin{split}
     K^\star_t & = -\left(P_t^\star
     \textstyle\sum_{j=1}^3 (D^{(j)})^\top D^{(j)} +\rho \bar{V}^{-1}\right)^{-1}B^\top_t P_t^\star,
     \quad
     V^\star_t  = \rho\left(P_t^\star
     \textstyle\sum_{j=1}^3 (D^{(j)})^\top D^{(j)} +\rho \bar{V}^{-1}\right)^{-1}.
\end{split}
\end{align*}
 Moreover, let    $\varphi^\star\in C([0,T];\sR)$ 
satisfy    
for all $t \in [0,T]$,
\begin{align}\label{eq:lyapunov_reg_num_2}
		(\tfrac{\d }{\d t}{\varphi})_t + &  \tfrac{1}{2}\tr\left(  \left(
		P^\star_t
		\textstyle\sum_{j=1}^3  (D^{(j)})^\top D^{(j)} +\rho \bar{V}^{-1} \right)
		V^\star_t\right)  +\tfrac{\rho}{2}\left( -3+\ln\left(\tfrac{\det(\bar{V})}{\det(V_t)} \right)\right) =0;
		\quad 
		\varphi_T  = 0,
\end{align}
Then the optimal cost of  \eqref{eq:cost_numerical}-\eqref{eq:state_numerical}
is given by
$\cC^\star=
\frac{1}{2}\sE[\xi_0^\top  \xi_0]P^\star_0+\varphi^\star_0$.

\paragraph{Implementation details.}

The numerical experiments are coded by using Tensorflow.
To    examine the linear convergence, 
 the scheme \eqref{eq:NPG_discrete_finite}  is implemented 
  on the  uniform time grid   $\pi_c$ with  mesh size $\Delta t =1/128$.
Indeed, let 
$K^0\equiv (1/3,1/3,1/3)$ and $V^0 \equiv 0.1D^\top D$
be the   initial guess.      
For  each $n\in \sN_0$, given $\theta^n=(K^n_i, V^n_i)_{i=0}^{127}$, 
 consider the Euler--Maruyama discretisation of    \eqref{eq:state_numerical_randomised}:
 $X_0 = \xi_0$ and 
 for all $i=0,\ldots, 127$, 
\begin{align}
\label{eq:EM_numerical}
    \begin{split}
                X_{i+1} =  X_{i }+ 
        B_{i\Delta t} \left( K^n_{i}X_i+(V^n_{i})^{\frac{1}{2}}\zeta_i\right)
 \Delta t+
\sum_{j=1}^3 D^{(j)} \left( K^n_{i}X_i+(V^n_{i})^{\frac{1}{2}} \zeta_{i}\right) 
 \Delta W^{(j)}_i,
    \end{split}
\end{align}
where $(\Delta W^{(j)}_i)_{i=0,\ldots, 127, j=1,\ldots 3}$
are independent normal random variables with mean zero and variance $1/128$,
and $(\zeta_{i})_{i=0}^{ 127}$ are  independent standard normal random vectors in $\sR^3$. 
We simulate $N_{\rm MC}=10^5$ independent trajectories
of \eqref{eq:EM_numerical} and approximate  $\cC(\theta^n)$ as follows
(cf.~\eqref{eq:C_theta_X}):
\begin{align*}
    \begin{split}
        \widehat{\mathcal{C}}(\theta^n)\coloneqq 
        & \frac{1}{N_{\rm MC}}\sum_{l=1}^{N_{\rm MC}}\frac{1}{2}\Bigg( \mu X_{128,l}^2 +\rho\sum_{i=0}^{127}\Big( (K^n_{i})^\top\bar{V}^{-1}K^n_{i}X_{i,l}^2+\tr(\bar{V}^{-1}V^n_{i})
        -3 +\ln\left(\tfrac{\det(\bar{V})}{\det(V^n_i)} \right) 
        \Big)\Delta t\Bigg),
    \end{split}
\end{align*}
where  $  (X_{i,l})_{i=0}^{128}$, $l=1,\ldots, N_{\rm MC}$,  represents the 
$l$-th trajectory of \eqref{eq:EM_numerical}. 
The required  gradients   $(\widehat{\nabla_{K^n_{i}} \mathcal{C}},\widehat{\nabla_{V^n_{i}} \mathcal{C}} )_{i=1}^{127} $ 
are computed using automatic differentiation along these paths, 
and
for each $i=0,\ldots, 127$, 
 the state covariance $\Sigma^{\theta^n}_{i\Delta t}$  is  estimated  
by      $\widehat{\Sigma}^n_{i} \coloneqq \frac{1}{N_{\rm MC}}\sum_{l=1}^{N_{\rm MC}}X^2_{i,l}$. 
The policy is then updated as follows (cf.~\eqref{eq:NPG_discrete_finite}):   
for all $i=0,\cdots,127$,
\begin{equation*}
 K_{i}^{n+1}=K_{i}^n-\frac{\tau}{\Delta t \widehat{\Sigma}^n_{i} }    \widehat{\nabla_{K^n_{i}} \mathcal{C}},
  \quad 
  V_{i}^{n+1}=V_{i}^n-\frac{\tau}{\Delta t } 
  \left(
   \widehat{\nabla_{V^n_{i}} \mathcal{C}}\, V_{i}^n+V_{i}^n\,  \widehat{\nabla_{V^n_{i}} \mathcal{C}}
  \right).
\end{equation*}
The optimal cost of  \eqref{eq:cost_numerical}-\eqref{eq:state_numerical}
is computed by 
solving    \eqref{eq:riccati_num}
and \eqref{eq:lyapunov_reg_num_2} with the  explicit Euler scheme on $\pi_c$, 
which leads to the value 
 $\cC^\star = 0.0402$.

To examine the robustness of \eqref{eq:NPG_discrete_finite} in time discretisation,  
a family of coarser time grids 
$\pi_m=\{i\frac{1}{m}\}_{i=0}^m\subset \pi_c $, $m\in \{8,16,32,64,128\}$, have been introduced.
The PG scheme only updates   policy parameters at the grid points of these coarser grids.
However, to mimic a continuous-time environment, 
  the performance   of each policy iterate is still evaluated by simulating 
\eqref{eq:EM_numerical} on the fine grid $\pi_c$ (with mesh size $\Delta t= 1/128$). 
In particular,
let $(K^0,V^0)$ be given as above. 
For each  $m\in \{8,16,32,64,128\}$ and  $n\in \sN_0$, given 
$\theta^n=(K^n_{j},V^n_{j})_{j=0}^{m-1}$, consider  the following   Euler-Maruyama discretisation of    \eqref{eq:state_numerical_randomised}:
$ X_0 =  \xi_0$ and 
 for all $j=0,\ldots, m-1$, and all $i=0,\ldots, 127$
 such that $\frac{j}{m} \leq i\Delta t<\frac{j+1}{m}$,
\begin{align}
\label{eq:EM_numerical_m}
    \begin{split}
               \quad
        X_{i+1} =  X_{i }+ 
        B_{i\Delta t} \left( K^n_{j}X_i+(V^n_{j})^{\frac{1}{2}}\zeta_i\right)
 \Delta t+
\sum_{j=1}^3 D^{(j)} \left( K^n_{j}X_i+(V^n_{j})^{\frac{1}{2}} \zeta_{i}\right) 
 \Delta W^{(j)}_i,
    \end{split}
\end{align}
where 
 $(\Delta W^{(j)}_i)_{i=0,\ldots, 127, j=1,\ldots 3}$
 and $(\zeta_{i})_{i=0}^{ 127}$ are independent random variables as in \eqref{eq:EM_numerical}. 
We shall
 sample      $10^5$ independent trajectories
of \eqref{eq:EM_numerical_m},
and use them to  approximate 
the   gradients in    $(K^n_{j},V^n_{j})_{j=0}^{m-1}$ 
 and 
 the state covariance 
 $(\Sigma^{\theta^n}_{j/m})_{j=0}^{m-1}$ with similar methods as above. 
 The scaled PG method \eqref{eq:NPG_discrete_finite} then updates the parameters by: 
  for all $j=0,\ldots, m-1$,
\begin{equation}
\label{eq:scaled_PG}
  K_{j}^{n+1}=K_{j}^n-\frac{m\tau}{\widehat{\Sigma}^n_j }    \widehat{\nabla_{K^n_{j}} \mathcal{C}},
  \quad 
  V_{j}^{n+1}=V_{j}^n- {m\tau} 
  \left(
   \widehat{\nabla_{V^n_{j}} \mathcal{C}}\, V_{j}^n+V_{j}^n\,  \widehat{\nabla_{V^n_{j}} \mathcal{C}}
  \right),
  \q \textnormal{with $\tau =0.01$},
\end{equation}
while the  unscaled PG method   updates the parameters by:
 for all $j=0,\ldots, m-1$,
\begin{equation}
\label{eq:unscaled_PG}
  K_{j}^{n+1}=K_{j}^n-\frac{\tau}{\widehat{\Sigma}^n_j }    \widehat{\nabla_{K^n_{j}} \mathcal{C}},
  \quad 
  V_{j}^{n+1}=V_{j}^n- {\tau} 
  \left(
   \widehat{\nabla_{V^n_{j}} \mathcal{C}}\, V_{j}^n+V_{j}^n\,  \widehat{\nabla_{V^n_{j}} \mathcal{C}}
  \right),
    \q \textnormal{with $\tau =0.08$}.
\end{equation}
Let 
$(\theta^{\pi_m,n})_{n\in \sN_{ 0}}$ be  the policy iterate    generated by \eqref{eq:scaled_PG},
  define $N^{\pi_m}(0.01)$ by  
$$
N^{\pi_m}(0.01)\coloneqq \min
\left\{n\in \sN_{  0} 
\,\vert\,
 \widehat{\cC}(\theta^{\pi_m,n})-\cC^\star_{\pi_m})< 0.01 
\right\},
$$
where   $\cC^\star_{\pi_m}\coloneqq \frac{1}{50}\sum_{n=951}^{1000}  \widehat{\cC}(\theta^{\pi_m,n})$
 approximates the optimal cost  among all piecewise constant polices on $
 \pi_m$. The quantity $N^{\pi_m}(0.01)$ is defined similarly for the iterates  generated by \eqref{eq:unscaled_PG}. 

 
\bibliographystyle{siam}
\bibliography{pgm.bib}

 \end{document}